\documentclass[11pt,letterpaper]{amsart}

\usepackage[T1]{fontenc}
\usepackage[utf8]{inputenc}
\usepackage{amsmath,amssymb}            % 'not essential, but very useful'
\usepackage{mathtools,tabu,mathrsfs}
\usepackage{hyperref}
\usepackage{ae}
\usepackage{comment}
\usepackage{bm}
\usepackage{graphicx}
\usepackage{float}

\newcommand{\N}{\mathbb{N}}
\newcommand{\Z}{\mathbb{Z}}

\newcommand{\R}{\mathbb{R}}
\newcommand{\C}{\mathbb{C}}

\newcommand{\A}{\mathcal{A}}

\DeclareMathOperator{\e}{\mathrm{e}}

\DeclareMathOperator{\erf}{\mathrm{erf}}
\DeclareMathOperator{\erfc}{\mathrm{erfc}}
\DeclareMathOperator{\artanh}{\mathrm{artanh}}

\renewcommand{\le}{\leqslant}
\renewcommand{\ge}{\geqslant}
\renewcommand{\leq}{\leqslant}
\renewcommand{\geq}{\geqslant}  

\newcommand{\floor}[1]{\left\lfloor #1 \right\rfloor}
\newcommand{\ceil}[1]{\left\lceil #1 \right\rceil}

\newcommand{\abs}[1]{\left| #1 \right|}
\newcommand{\dv}{\, \mid \,}
\newcommand{\notdv}{\, \nmid \,}
\newcommand{\norm}[1]{\left\| #1 \right\|}

\newcommand{\conjugate}[1]{\overline{#1}}

\newcommand{\pimirror}{\overleftarrow{\pi}\!}

\newcommand{\one}{\mathbf{1}}

\newtheorem{theorem}{Theorem}
\newtheorem{proposition}{Proposition}

\newtheorem{lemma}{Lemma}

\newtheorem{definition}{Definition}
\newtheorem{notation}{Notation}
\newtheorem{remark}{Remark}

\numberwithin{equation}{section}

\numberwithin{theorem}{section}
\numberwithin{proposition}{section}
\numberwithin{corollary}{section}
\numberwithin{lemma}{section}
\numberwithin{conjecture}{section}
\numberwithin{definition}{section}
\numberwithin{notation}{section}
\numberwithin{remark}{section}
\numberwithin{table}{section}

\begin{document}
\title{Prime numbers with an almost prime reverse}

\author[C. Dartyge]{C\'ecile Dartyge}
\address{
  C\'ecile Dartyge,
  Institut Élie Cartan CNRS UMR 7502, Institut Universitaire de France,
  Universit\'e de Lorraine,
  BP 70239, 54506 Vandœuvre-lès-Nancy Cedex, France
}
\email{cecile.dartyge@univ-lorraine.fr}

\author[J. Rivat]{Jo\"el Rivat}
\address{
  Jo\"el Rivat,
  Universit\'e d'Aix-Marseille, Institut Universitaire de France,
  Institut de Math\'ematiques de Marseille CNRS UMR 7373,
  163 avenue de Luminy, Case 907, 13288 Marseille Cedex 9, France.
}
\email{joel.rivat@univ-amu.fr}

\author[C. Swaenepoel]{Cathy Swaenepoel}
\address{
  Cathy Swaenepoel,
  Universit\'e Paris Cit\'e and Sorbonne Universit\'e,
  CNRS, INRIA, IMJ-PRG, F-75013 Paris, France.
}
\email{cathy.swaenepoel@u-paris.fr}

\begin{abstract} % Ne pas mettre de retour à la ligne (sinon probleme avec Arxiv) 
Let $b$ be an integer greater than or equal to $2$. For any integer $n\in \left[b^{\lambda-1}, b^{\lambda}-1\right]$, we denote by $R_\lambda (n)$ the reverse of $n$ in base~$b$, obtained by reversing the order of the digits of $n$. We establish a Bombieri-Vinogradov type theorem for the set of the reverses of the prime numbers. Combined with sieve methods, this permits us to prove that there exist $\Omega_b\in\mathbb{N}$ and $c_b>0$ such that, for at least  $c_b b^{\lambda} \lambda ^{-2}$ primes $p\in \left[b^{\lambda-1}, b^{\lambda}-1\right]$, the reverse $R_\lambda(p)$ has at most $\Omega_b$ prime factors. Some explicit admissible values of $\Omega_b$ are given.
\end{abstract}


\thanks{This work was supported by the joint ANR-FWF Grant 4945-N and ANR Grant 20-CE91-0006}

\date{\today}

\maketitle

\setcounter{tocdepth}{1}
\tableofcontents 

\section*{Notation} % pas de s en anglais

Throughout the paper, the letter $b$ will denote an integer larger
than or equal to 2 which will be the base of all digit expansions. The
notations $U = O(V)$, $U\ll V$ and $V\gg U$ all mean that there is a
constant $C>0$ depending at most on $b$ such that $\abs{U} \leq C
V$. If we want to indicate that the implicit constant $C$ is allowed
to depend on a parameter $\alpha$ then this dependence is indicated by
writing $U = O_{\alpha}(V)$, $U\ll_{\alpha} V$ or $V \gg_{\alpha}
U$. We also write $U\asymp V$ if $U \ll V \ll U$ and similarly for
$U\asymp_{\alpha} V$.

The letter $p$, with or without indices, always denotes a prime number.

We use $\mu(d)$, $\Lambda(d)$, $\tau(d)$, $\omega(d)$, $\Omega (d)$
and $v_p(d)$ to denote respectively the M{\"o}bius function, the von
Mangoldt function, the number of positive divisors, the number of distinct
prime factors, the number of prime factors counted with multiplicity
and the $p$-adic valuation of an integer $d\ge 1$.  We denote by
$P^{-}(d)$ the smallest prime factors of an integer $d\ge 2$.

For a finite set $S$ we use $\abs{S}$ to denote its cardinality.   

For a real number $x$ we set
\begin{displaymath}
  \e(x) = \exp(2\pi i x) \text{ and } \norm{x} = \min\{|x-k|:~k \in \Z\}. 
\end{displaymath}

For a certain property $\mathsf{P}$, we define
$\mathbf{1}_{\mathsf{P}}$ by $\mathbf{1}_{\mathsf{P}}=1$ if
$\mathsf{P}$ is satisfied and  $\mathbf{1}_{\mathsf{P}}=0$ otherwise.

\section{Introduction}

The investigation of the  Twin Primes conjecture,
the Goldbach conjecture, and more generally the distribution of 
pairs of primes related together (e.g. Sophie Germain primes)
remain among the most famous, difficult and challenging
problems.
Unfortunately,
despite centuries of efforts,
none of these problems could be solved.
However,
for the Goldbach conjecture, Rényi \cite{renyi-1948} successfully 
replaced one of the two prime numbers by an almost prime number
(i.e. with a bounded number of prime factors).
Later,
Chen \cite{chen-1973}
proved that any sufficiently large even integer is a sum of a prime
and a product of at most two primes,
and a similar result for the Twin Primes conjecture.

In the last 30 years, the investigation of integers with digital
properties has attracted a lot of interest.
After partial results by Fouvry and Mauduit
\cite{fouvry-mauduit-1996-1,fouvry-mauduit-1996-2}
and
Dartyge and Tenenbaum
\cite{dartyge-tenenbaum-2005-Fourier,dartyge-tenenbaum-2006},
two questions posed  in 1968 by Gelfond \cite{gelfond-1968} on the sum
of digits of prime numbers and squares 
were solved by Mauduit and
Rivat~\cite{mauduit-rivat-2010,mauduit-rivat-2009}. 
Bourgain \cite{bourgain-2013-IJM,bourgain-2015-IJM} studied the prime
numbers with a positive proportion of preassigned digits in base $2$,
and Swaenepoel \cite{swaenepoel-2020-plms} obtained explicit results
for this question valid in all bases.  A recent breakthrough was
obtained in the detection of prime numbers with missing digits by
Maynard \cite{maynard-2019-inventiones,maynard-2021-IMRN}.  In
particular he proved that there are infinitely many prime numbers with
no digit $3$ in their decimal expansion (in this result $3$ can replaced
by any $a_0\in\{ 0,\ldots ,9\}$).
The case of almost prime numbers with forbidden digits has been studied by
Dartyge and Mauduit~\cite{dartyge-mauduit-2000,dartyge-mauduit-2001}.
They proved for example that there are infinitely many integers $n$
with at most $3$ prime factors and without the digit $2$ in their
expansion in base~$3$. When the base $b \to +\infty$, they proved the
existence of infinitely many integers $n$ with less than
$\frac{8b}{\pi} (1+o(1))$ prime factors and with only the digits $0$
and $1$ in their expansion in base $b$.

The aim of this work, inspired by the Twin Primes conjecture and
related problems,
is to study, in the spirit of Rényi \cite{renyi-1948},
one of the most natural analogue questions concerning pairs of prime numbers
characterized by their digital expansion,
namely the almost primality of the reverse of
prime numbers.  The reverse of an integer in a given base $b$ is
the integer obtained by reversing its digits in this base. For
example in base $b=10$, the integer $321$ is the reverse of $123$.

We formalize this as follows.
Let $b\geq 2$ be an integer.
Any integer~$n\geq 0$ can be written in base $b$ as
\begin{displaymath}
    n = \sum_{j\geq 0} \varepsilon_{j}(n) \,b^j
\end{displaymath}
where $\varepsilon_j(n) \in \{0,\ldots,b-1\}$
for all $j\geq 0$
and $\varepsilon_j(n)=0$ for $j \geq \log(n+1)/\log b$.

\begin{definition}
  For any integers $n\geq 0$ and $\lambda\geq 0$ we define
  $R_\lambda(n)$ to be the `reverse' of the $\lambda$ first digits of
  $n$ in base $b$ by
  \begin{displaymath}
    R_\lambda(n)
    =
    \sum_{j=0}^{\lambda-1} \varepsilon_j(n) \ b^{\lambda-1-j}
    ,
  \end{displaymath}
  with the convention that $R_0(n)=0$.
  
  If $n\in \left\{b^{\lambda -1},\ldots,b^{\lambda}-1\right\}$ then
  $R_\lambda (n)$ is simply called the `reverse' of $n$.
\end{definition}

\begin{remark}
  We notice that $\varepsilon_j$ and $R_\lambda$ implicitly depend on
  the base $b$. 
\end{remark}

The integers equal to their reverse are called palindromes.
For instance, $12321$ is a palindrome in base $b=10$.
Col \cite{col-2009-palindromes} proved, for any $b\ge 2$,
the existence of an integer $k_b\geq 1$ such that there are infinitely
many palindromes in base $b$ with at most $k_b$ prime factors.
Recently Tuxanidy and Panario \cite{tuxanidy-panario-2024}
proved that $k_b=6$ is admissible
for all $b\ge 2$.

Concerning the reverse of integers, in the case $b=2$, it is proved in
\cite{DMRSS_reversible_primes} that there exists $\lambda_0\geq 1$
such that for any integer $\lambda \geq \lambda_0$,
\begin{displaymath}
  \abs{\{ 2^{\lambda-1}\leq n < 2^\lambda :
    \max\left(\Omega (n), \Omega (R_\lambda (n)\right)\le 8\}}
  \gg \frac{2^\lambda}{\lambda^2}.
\end{displaymath}
The proof of \cite{DMRSS_reversible_primes} can easily be generalized
to give a similar result in other bases. 

Finding infinitely many prime numbers whose reverse
is also prime seems to be at least as difficult as the Twin
Primes conjecture.  However we are able to show that there are
infinitely many primes with an almost prime reverse:
\begin{theorem}\label{thm-almost-prime}
  For any integer $b\geq 2$, there exists an explicit $\Omega_b > 0$
  (defined by \eqref{eq:def_Omega_b_weighted_sieves})
  and
  $\lambda_0(b)>0$ such that for any integer
  $\lambda \geq \lambda_0(b)$,
  \begin{equation}\label{eq:lower_bound_thm_almost_primes}
    \abs{\{b^{\lambda-1} \leq p < b^\lambda :
      \Omega(R_\lambda(p))\leq \Omega_b\}}
    \gg
    \frac{b^\lambda}{\lambda^2}
    .
  \end{equation}
  The values of $\Omega_b$ for $2\le b\le 10$ are given in
  Table~\ref{table:values_Omega_b}.
  \begin{table}[H]
    \centering
    \begin{equation*}
      \begin{tabu}{|c|c|c|c|c|c|c|c|c|c|}
        \hline
        b        &  2 &  3 &  4 &  5 &  6 &  7  &  8  &  9  &  10 \\
        \hline
        \Omega_b & 228 & 333 & 443 & 487 & 669 & 759 & 984 & 1137 & 1378 \\
        \hline
      \end{tabu}
    \end{equation*}
    \caption{Values of $\Omega_b$ for $2 \leq b \leq 10$}
    \label{table:values_Omega_b}
  \end{table}
  When $b\to +\infty$ we have
  \begin{equation}\label{eq:maj_Omega_b}
    \Omega_b
    \leq
    \frac{36\, b^2}{\pi ( T_{\infty ,1}^{-1}-1)^2} +O(b)
    ,
  \end{equation}
  where the implicit constant is absolute and
  $T_{\infty ,1}$ is an explicit constant
  defined in \eqref{eq:definition-T_infty_1}
  so that
  \begin{displaymath}
    \frac{36}{\pi ( T_{\infty ,1}^{-1}-1)^2}=538.106849\ldots
  \end{displaymath}
\end{theorem}

The lower bound \eqref{eq:lower_bound_thm_almost_primes} is almost of the
expected order of magnitude, which is
$b^\lambda (\log \lambda)^{\Omega_b-1}\lambda^{-2}$. This is due to
the fact that the prime numbers $p$ detected by our method are such
that $R_\lambda (p)$ is without prime factors $<b^{\xi_1 (b) \lambda}$
for some $\xi_1 (b) >0$.

% The initial steps of the proof of this theorem
% consist in combining sieve methods
% and Vaughan's identity. The role of the sieve is to detect the almost
% prime values of $R_\lambda (p)$. We thus need to study the prime
% numbers $p$ such that $R_\lambda (p)\equiv 0\bmod d$ with $d<D$ and
% $D$ as large as possible. Applying Vaughan's identity,
% we then have to bound
% exponential sums associated with $R_\lambda$.
% This leads us to study interesting questions of harmonic analysis
% involving products of oscillatory functions
% in several variables.

In base $2$, an upper bound of the expected order of magnitude for the
number of primes whose reverse is also prime is given in
\cite{DMRSS_reversible_primes}.  We are able to extend this result to
any base.
\begin{theorem}\label{theorem:true-reversible-primes}
  For any integer $b\geq 2$, we have for any integer $\lambda\geq 1$,
  \begin{equation}
    \label{eq:true-reversible-primes}
    \abs{\{b^{\lambda-1} \leq p < b^\lambda :
      R_\lambda(p) \text{ is prime} \}
    }
    \ll
    \frac{b^\lambda}{\lambda^2}
    .
  \end{equation}
\end{theorem}

A crucial task to apply sieve methods is to understand the
distribution in arithmetic progressions of the reverse of prime numbers.

\begin{notation}\label{notation:pimirror}
  For  $b,a,d,\lambda\in\Z$
  with $b\geq 2$, $d\geq 2$, $\lambda\geq 1$, $t \in \R$
  we define
  \begin{align}
    \pi_\lambda(t)
    & =
      \abs{\{ b^{\lambda-1} \leq p < t \}}
      \label{eq:def_pi_lambda_t}
    \\
    \pimirror_\lambda(t,a,d)
    &=
      \abs{\{
      b^{\lambda-1} \leq p < t:
      R_\lambda(p) \equiv a \bmod d
      \}}
      \label{eq:def_pimirror_lambda_t_a_d}
    \\
    \pimirror_\lambda(a,d)
    & =
      \abs{\{
      b^{\lambda-1} \leq p < b^\lambda:
      R_\lambda(p) \equiv a \bmod d
      \}}
      .
      \label{eq:def_pimirror_lambda_a_d}
  \end{align}
\end{notation}
Let us first consider
\begin{displaymath}
  Z_\lambda(a,d)
  =
  \abs{\{
    b^{\lambda-1} \leq n < b^\lambda:
    R_\lambda(n) \equiv a \bmod d
    \}}
  .
\end{displaymath}
If $\gcd(d,b)=1$ then, writing $m = R_{\lambda}(n)$, we get
\begin{align*}
  Z_\lambda(a,d)
  &=
  \abs{\{0\leq m< b^{\lambda}:
    m \equiv a \bmod d \text{ and } m\not\equiv 0 \bmod b\}}
  \\
  &=
  \frac{b^{\lambda}-b^{\lambda-1}}{d} + O(1),
\end{align*}
hence the reverse of the integers $n$ such that
$b^{\lambda-1} \leq n < b^\lambda$ is well distributed in all
arithmetic progressions modulo $d$, provided that
$d = o(b^{\lambda})$.

When one restricts to prime numbers, this well distribution is no
longer true in general. Indeed, we will see in
\eqref{eq:relations_n_reverse_n} that $R_\lambda(n)$ and $n$ are
connected modulo $b-1$ and $b+1$, 
which implies for instance that if $\gcd(a,d,b^2-1) > 1$ then
$\pimirror_\lambda(a,d)=0$ for any $\lambda \geq 2$.  Nevertheless,
when $\gcd(d,b(b^2-1))=1$, the reverse of the prime numbers $p$
such that $b^{\lambda-1} \leq p < b^\lambda$ is expected to be well
distributed modulo $d$, namely 
\begin{displaymath}
  \pimirror_\lambda(a,d)
  =
  \frac{\pi_\lambda(b^{\lambda})}{d}(1+o(1))
  .
\end{displaymath}

In order to prove Theorem~\ref{thm-almost-prime} and
Theorem~\ref{theorem:true-reversible-primes}, we will need to estimate
$\pimirror_\lambda(a,d)$ for $d <b^{\xi \lambda}$ for some $\xi>0$. An
estimate on average in the spirit of the Bombieri-Vinogradov Theorem
will be sufficient. The following result is the key and most difficult
part of our work.
\begin{theorem}\label{theorem:type_bombieri-vinogradov}
  For any integer $b\geq 2$, there exists an explicit
  $ \xi_0 = \xi_0(b) >0$ (which will be defined by
  \eqref{eq:def_xi_0})
  such that, given $\xi \in\left]0,\xi_0\right[$, we can find
  $c = c(b,\xi) >0$ and $\lambda_0(b,\xi)\geq 2$ with the property
  that for any integer $\lambda\geq\lambda_0(b,\xi)$, we have
  \begin{displaymath}
    \sum_{\substack{d \leq b^{\xi \lambda}\\ \gcd(d,b(b^2-1))=1}}
    \sup_{t \in\left[b^{\lambda-1}, b^{\lambda}\right]}
    \sup_{1\le a\le d}
    \abs{
      \pimirror_\lambda(t,a,d)
      - \frac{ \pi_\lambda(t)}{d}
    }
    \ll
    b^{\lambda-c\sqrt{\lambda}}.
  \end{displaymath}
\end{theorem}
It is probably possible to prove a version of
Theorem~\ref{theorem:type_bombieri-vinogradov} without the condition
$\gcd(d , b(b^2-1))=1$ in the sum over $d$ (and so with
$\frac{ \pi_\lambda(t)}{d}$ replaced by a suitable quantity).  This
generalization would require further decompositions and splitting in
several steps of the proof where we have exploited this coprimality
condition.  In order to avoid such complications we have chosen not to
explore this question and to state a result sufficient for our sieve
applications.

\medskip

Our method permits us, as a byproduct, to obtain the following `individual'
estimate of $\pimirror_\lambda(a,d)$.
 
\begin{theorem}\label{theorem:reversed-siegel-walfisz}
  For any integer $b\geq 2$, there exists $c=c(b)>0$ such that
  for  $a,d,\lambda\in\Z$
  with  $\lambda\geq 2$
  and
  \begin{equation}\label{eq:range_d}
    2 \leq d \leq \exp\left(\frac{c \, \lambda}{\log \lambda}\right)
    ,
  \end{equation}
  we have for any $t \in\left[b^{\lambda-1}, b^{\lambda}\right]$,
  \begin{equation}
    \label{eq:reversed-siegel-walfisz}
    \pimirror_\lambda(t,a,d)
    =
    \frac{\Pi_d(b)}{d}
    \pimirror_\lambda\left(t,a,\Pi_d(b)\right)
    +
    O\left(b^\lambda \exp\left(-\frac{c \, \lambda}{\log d}\right)\right)
  \end{equation}
  where
  \begin{equation}\label{def:Pi_d(b)}
    \Pi_d(b) := \gcd(d,(b^2-1)) \prod_{p\dv b} p^{v_p(d)}
    .
  \end{equation}
  In particular if $\gcd(d,b(b^2-1))=1$ then
  \begin{equation}
    \label{eq:reversed-siegel-walfisz_with_coprimality}
    \pimirror_\lambda( t,a,d)
    =
    \frac{\pi_\lambda(t)}{d}
    +
    O\left(b^\lambda \exp\left(-\frac{c \, \lambda}{\log d}\right)\right)
    .
  \end{equation}
\end{theorem}
In a recent preprint, Bhowmik and
Suzuki~\cite{bhowmik-suzuki-2024-arxiv} obtained
\eqref{eq:reversed-siegel-walfisz} (stated differently) under the
additional assumption that $b\geq 31699$ (and $t=b^\lambda$).  This
range was improved to $b\ge 26000$ by Chourasiya and
Johnston~\cite{chourasiya-johnston-arxiv-2025}.
Theorem~\ref{theorem:reversed-siegel-walfisz} extends these
results to any base $b\geq 2$,
using a different method.

When $\gcd(d,b(b^2-1))>1$, estimating the quantity
$\pimirror_\lambda\left(t,a,\Pi_d(b)\right)$ can be reduced to
counting prime numbers in some subintervals of
$\left[b^{\lambda-1}, t\right[$ and in an arithmetic progression
modulo $\gcd(d,b^2-1)$ (see~\cite{bhowmik-suzuki-2024-arxiv} for
further details), which is a classical task in Analytic Number Theory.

We notice that Theorem
\ref{theorem:reversed-siegel-walfisz} does not provide an asymptotic
formula for the complete range \eqref{eq:range_d}.  In particular,
\eqref{eq:reversed-siegel-walfisz_with_coprimality} leads to an
asymptotic formula in the range
\begin{displaymath}
2\le d\le\exp \left (
 c_1 \sqrt{\lambda}
 \right ),
\end{displaymath} 
for some $c_1=c_1(b)>0$, while for larger values of $d$,
\eqref{eq:reversed-siegel-walfisz_with_coprimality} provides only an
upper bound.
This permits us to formulate the following result in the spirit of the
Siegel-Walfisz Theorem.

\begin{theorem}\label{theorem:corollary_SW}
  For any integer $b\geq 2$, there exists $c_1=c_1(b)>0$ such that
  for  $a,d,\lambda\in\Z$
  with  $\lambda\geq 1$, $\gcd(d,b(b^2-1))=1$
  and
  \begin{equation}\label{eq:range_d_cor}
    2 \leq d \leq \exp(c_1 \sqrt{\lambda})
    ,
  \end{equation}
  we have 
  \begin{equation}
    \label{eq:reversed-siegel-walfisz_cor}
    \pimirror_\lambda(a,d)
    =
    \frac{\pi_\lambda(b^{\lambda})}{d}
    \left(1+O\left(\exp(-c_1 \sqrt{\lambda})\right)\right)
    .
  \end{equation}
\end{theorem}
 
It is interesting to observe that
Theorem~\ref{theorem:corollary_SW} provides,
in any base $b\geq 2$,
an asymptotic formula for
the number of primes $p\le b^\lambda$ such that $R_{\lambda}(p)$ is
squarefree. The interested reader can consult
\cite{chourasiya-johnston-arxiv-2025} and also the recent preprint of
Carillo-Santana \cite{carillo-santana-arxiv-2025} for further details.

\medskip
  
It is important to notice that
Theorem~\ref{theorem:reversed-siegel-walfisz} and
Theorem~\ref{theorem:corollary_SW} are only byproducts of our method
and that the range for $d$ given in \eqref{eq:range_d} is not
sufficient to apply the sieves as needed in the proof of
Theorem~\ref{thm-almost-prime} and
Theorem~\ref{theorem:true-reversible-primes}.
  
\bigskip

A natural approach to detect congruences is to work with exponential sums.
We will prove the two following results.

\begin{theorem}\label{theorem:bound_expo_sum_average}
  For any integer $b\geq 2$, there exists an explicit
  $ \xi_0 = \xi_0(b) >0$ (which will be defined by
  \eqref{eq:def_xi_0}) such that, given $\xi \in\left]0,\xi_0\right[$,
  we can find $c = c(b,\xi) >0$ and $\lambda_0(b,\xi)\geq 2$ with the
  property that for any integer $\lambda\geq\lambda_0(b,\xi)$, we have
  \begin{equation}\label{eq:maj_conclusion_sum_expo}
    \sum_{\substack{d\leq b^{\xi \lambda}\\ \gcd(d,b(b^2-1))=1}}
    \frac{1}{d}
    \sum_{\substack{1\leq h < d\\ \gcd(h,d)=1}}
    \sup_{t \in\left[b^{\lambda-1}, b^{\lambda}\right]}
    \abs{\sum_{b^{\lambda-1} \leq n < t}
      \Lambda(n)\e\left(\frac{hR_{\lambda}(n)}{d}\right)}
    \ll
    b^{\lambda-c\sqrt{\lambda}}.
  \end{equation}
\end{theorem}

\begin{theorem}\label{theorem:bound_expo_sum_conclusion}
  For any integer $b\geq 2$,
  there exists $c(b)>0$
  such that for any integer $\lambda\geq 2$
  and any $(h,d)\in\Z^2$ with $d\geq 2$ and 
  \begin{math}
    (b^2-1) b^\lambda h \not\equiv 0 \bmod d,
  \end{math}
  we have
  \begin{equation}\label{eq:bound_expo_sum_conclusion}
    \sup_{t \in\left[b^{\lambda-1}, b^{\lambda}\right]}
    \abs{
      \sum_{b^{\lambda-1} \leq n < t}
      \Lambda(n)
      \e\left(\frac{h R_{\lambda}(n)}{d}\right)
    }
    \ll
    \lambda^{2+\frac{\omega(b)}2}
    b^{\lambda-\frac{c(b)\lambda}{\log d} }
    .
  \end{equation}
  In particular, for any integer $b\geq 2$ and any real number $A>0$,
  there exists $c'=c'(b,A)>0$ such that for any integer $\lambda\geq 2$
  and any $(h,d)\in\Z^2$ with
  \begin{equation}\label{eq:maj_d_individual}
    2 \leq d \leq \exp\left(\frac{c'\lambda}{\log \lambda}\right)
  \end{equation}
  and 
  \begin{math}
    (b^2-1) b^\lambda h \not\equiv 0 \bmod d,
  \end{math}
  we have
  \begin{equation}\label{eq:bound_expo_sum_conclusion_with_cond_d}
    \sup_{t \in\left[b^{\lambda-1}, b^{\lambda}\right]}
    \abs{
      \sum_{b^{\lambda-1} \leq n < t}
      \Lambda(n)
      \e\left(\frac{h R_{\lambda}(n)}{d}\right)
    }
    \ll
    \frac{b^{\lambda}}{\lambda^A}
    .
  \end{equation}
\end{theorem}

\section{Description of the proof}

In order to prove Theorem~\ref{thm-almost-prime} we apply sieve
methods. The linear sieve enables us to detect prime numbers
whose reverse has no small prime factor and the weighted sieve
permits us to get a better admissible $\Omega_b$ in
Theorem~\ref{thm-almost-prime}. These two sieves are described in
Section \ref{section-sieves}.  The main argument needed to apply these
sieves successfully is a result of the form of
Theorem~\ref{theorem:type_bombieri-vinogradov}. The proof of this
theorem is the essential part of our work.
Theorem~\ref{theorem:reversed-siegel-walfisz} is obtained in parallel
as a byproduct.
 
The condition $R_\lambda (p)\equiv a \bmod d$ is detected by exponential sums.
This yields to estimate sums of the form
\begin{displaymath}
  \sum_{b^{\lambda-1} \leq n < t}
  \Lambda(n)
  \e\left(\frac{h R_{\lambda}(n)}{d}\right)
\end{displaymath}
that we must bound individually when $d$ is small
and on average on the pairs $(h,d)$ for larger $d$.

By a combinatorial identity (Vaughan's identity,
see Section~\ref{section:combinatorial-identity}),
we are left to estimate `Type I' and `Type II' sums.

For the Type II sums,
in Section~\ref{section:type-II-sums},
we adopt the strategy of Mauduit-Rivat \cite{mauduit-rivat-2010}.
After applying the van der Corput inequality
(Lemma \ref{lemma:van-der-corput})
we show that a carry propagation property still exists
in our context and we take advantage of it. 
Then we have to obtain non trivial bounds of sums of type 
\begin{equation*}
  \sum_{b^{\nu-1}\leq n < b^{\nu+1}}
  \abs{
    \sum_{m}
    \e\left(\alpha b^{\lambda-\mu_2} (R_{\mu_2}(m(n+r))- R_{\mu_2}(mn))\right)
  },
\end{equation*}
(defined in \eqref{eq:definition-T2}) where $m$ runs in a sub-interval
of some $b$-adic segment $\left[b^{\mu-1}, b^\mu\right[$.

Since the functions $R_{\mu_2}$ are $b^{\mu_2}$-periodic, we use
discrete Fourier transforms and exponential sums in several variables.
For most terms, which may be called the `probabilistic' terms, the
exponential sums are small and can be estimated easily by
Lemma~\ref{lemma:sum_inverse_sinus}.  We are left with a collection of
exponential sums potentially large but with variables submitted to
diophantine restrictions, which may be called `extended diagonal'
terms.  After a series of delicate manipulations, we are left to
estimate sums (see \eqref{eq:definition-T4}) of the shape
\begin{displaymath}
  \sum_{0\leq a' < b^\delta}
  \min\left(
    \gcd(r,d) b^{\mu+\delta-\mu_2} ,
    \abs{
      \sin\left( \pi \norm{\frac{a'r}{\gcd(r,d) b^\delta}} \right)
    }^{-1}
  \right)
  \abs{
    \mathcal{F}_{\delta}\left(
      \alpha',\frac{a'}{b^{\delta}}
    \right)
  }^2,
\end{displaymath}
for some parameters $r,d,\delta,\mu,\mu_2,\alpha'$
and, see \eqref{eq:FT-product-formula},
\begin{displaymath}
 \abs{\mathcal{F}_\lambda(\alpha,\vartheta)}
    =
    \prod_{j=0}^{\lambda-1}
    \abs{K_b\left( \alpha b^{\lambda-1-j} - \vartheta  b^j\right)}
    ,
\end{displaymath}
where $K_b (t)$ is the normalized Dirichlet Kernel of length $b$
introduced in Lemma \ref{lemma:concavity-of-Dirichlet-kernel}.

The large values of $\delta$ constitute the most difficult part of our
paper and produce the most restrictive constraints for $\Omega_b$
in Theorem~\ref{thm-almost-prime}.
It is the object of
Section~\ref{section:critical-contribution-of-Type-II-sums}. 

For the Type I sums, we need to adapt and extend the method in
\cite{DMRSS_reversible_primes}, where the sums were defined
differently: our Type I sums here are linked to the Vaughan identity
whereas in \cite{DMRSS_reversible_primes} there was a congruence
involving the products $nR_\lambda(n)$.

We are led to study interesting questions of harmonic analysis
involving products of oscillatory functions
in several variables.
All along our proof, we need to handle non trivially both variables
$\alpha$ and $\vartheta$ in $\mathcal{F}_\lambda(\alpha,\vartheta)$,
which is made possible by the product structure of
$\mathcal{F}_\lambda(\alpha,\vartheta)$.  This permits us to apply
several times the Hölder and Cauchy--Schwarz inequalities, combined
with some elimination or separation inequality given in
Lemma~\ref{lemma:pointwise-upperbound-of-F}.  We then need sharp
bounds of $L^\kappa$ norms, with $\kappa \in [0,+\infty]$, of the
function $\mathcal{F}_\lambda(\alpha,\vartheta)$
and auxiliary associated functions. 
This is the subject of a series of lemmas given in Sections
\ref{section:lemmas} and \ref{section:fourier-transform}.  The upper
bounds or estimates given in Sections \ref{section:lemmas} and
\ref{section:fourier-transform} are completely explicit and we expect
that they can be successfully applied in other contexts.
For example, in Section~\ref{section:palindromes},
Proposition~\ref{prop:maj_int_prod_K_b} gives an upper bound of the positive
moments of the Fourier transform associated to palindromes.  This
improves and generalizes \cite{tuxanidy-panario-2024} who considered
only even moments.

Finally let us point out that Lemma \ref{lemma:Lkappa-norm-of-G}
provides a bound of the positive moments of our generating
function $G_\lambda$ with a general periodic density function $\Phi$.
The idea of considering such integrals was fruitful and may have
further applications.

\section{Sieve approach}\label{section-sieves}

In order to prove Theorem~\ref{thm-almost-prime} and
Theorem~\ref{theorem:true-reversible-primes}, we will use the
following versions of the linear sieve and the weighted sieve.

\begin{lemma}\label{lemma:linear-sieve}
  Let $\mathcal{A}$ be a finite sequence of positive integers and let
  $\mathcal{P}$ be a sequence of primes. For any real number
  $z\geq 2$, we denote
  \begin{displaymath}
    P(z) = \prod_{\substack{p<z\\p \in \mathcal{P}}} p
  \end{displaymath}
  and
  \begin{displaymath}
    S(\mathcal{A},\mathcal{P},z) = \abs{\{a \in \mathcal{A}: \gcd(a,P(z))=1\}}.
  \end{displaymath}  
  Let $X$ be a real number and $g$ be a multiplicative function
  such that for any $p\in \mathcal{P}$, $0\leq g(p) < 1$.  For any
  integer $d\geq 1$, let $r_d(\mathcal{A})$ defined by
   \begin{equation}\label{eq:definition-r_d-A}
     \abs{\{a\in \mathcal{A}: d \mid a\}}
     =
     g(d)X+r_d(\mathcal{A}).
   \end{equation}
   If there is an absolute constant $L$ such that for any real numbers
   $2 \leq w < z$,
   \begin{equation}\label{eq:condition-crible-lineaire}
     \prod_{\substack{w\leq p<z\\p\in\mathcal{P}}}(1-g(p))^{-1}
     \leq
     \frac{\log z}{\log w}\left(1+\frac{L}{\log w}\right)
   \end{equation}
   then we have, for any real numbers $2 \leq z \leq D$,
    \begin{displaymath}
    S(\mathcal{A},\mathcal{P},z)
    <
    \left(F(s) + O\left((\log D)^{-1/6}\right)\right)
    X V(z)
    +R(\mathcal{A},D),
  \end{displaymath}
   \begin{displaymath}
    S(\mathcal{A},\mathcal{P},z)
    >
    \left(f(s) + O\left((\log D)^{-1/6}\right)\right)
    X V(z)
    -R(\mathcal{A},D),
  \end{displaymath}
  where $s = \log D/\log z$,
  \begin{displaymath}
    V(z)
    =
    \prod_{\substack{p<z\\p\in\mathcal{P}}}\left(1-g(p)\right),
  \end{displaymath}
  \begin{displaymath}
    R(\mathcal{A},D)
    =
    \sum_{\substack{d<D\\d \mid P(z)}}\abs{r_d(\mathcal{A})}
  \end{displaymath}
  and $F$ and $f$ are the continuous solutions to the system
  \begin{displaymath}
    \left\{
      \begin{array}{r@{\;}ll}
        uF(u) &= 2 e^{\gamma} & \text{ if } u \leq 3,\\
        uf(u) &= 0          & \text{ if } u \leq 2,
      \end{array}
    \right.
    \qquad
    \qquad
    \left\{
      \begin{array}{r@{\;}ll}
        (uF(u))' &= f(u-1) & \text{ if } u > 3,\\
        (uf(u))' &= F(u-1) & \text{ if } u > 2,
      \end{array}
    \right.
  \end{displaymath}
  which are monotonic (decreasing and increasing respectively) and satisfy
  \begin{align*}
    \begin{gathered}
      0 < f(u) < 1 < F(u) \quad \text{ for } u > 2,\\
      F(u)=1+O\left(e^{-u}\right), \quad f(u)=1+O\left(e^{-u}\right)
      \quad \text{ as } u \to +\infty.
    \end{gathered}
  \end{align*}
\end{lemma}
\begin{proof}
  See for
  instance~\cite[Chapter~12]{friedlander-iwaniec-opera-cribro-2010} or
  \cite{iwaniec-1980-RosserSieve}.
\end{proof}

The following lemma is a version of the weighted sieve with Richert's weights.
We introduce for any integer $R\ge 2$, the real
number
\begin{equation}\label{eq:def_Lambda_R}
  \Lambda_R
  =
  R + \frac{1}{\log 3}\log\left(\frac{3}{4} \left(1+3^{-R}\right)\right)
  \in  \left] R-1, R\right].
\end{equation}  

\begin{lemma}\label{lemma-weighted-sieve}
  As in Lemma \ref{lemma:linear-sieve},
  let $\mathcal{A}$, $X$, $d$, $g$ and $r_d(\mathcal{A})$
  satisfying \eqref{eq:definition-r_d-A}
  and \eqref{eq:condition-crible-lineaire}
  with $\mathcal{P}$ 
  being the set of the prime numbers.
  Let $D\ge 2$ be a real number such that 
  \begin{equation}\label{eq:error-term-weighted-sieve}
    \sum_{d\le D} |r_d (\mathcal{A})|\ll X(\log X)^{-3}
  \end{equation}
  and
  \begin{equation}\label{eq:squares}
    \sum_{p>D^{1/4}}\abs{\{ a\in \mathcal{A}: p^2 \mid a\}}\ll X (\log X)^{-3}.
  \end{equation}
  Let $R\geq 2$ be an integer and $\varepsilon>0$ such that
  $D\ge \left(\max\mathcal{A}\right)^{\Lambda_R^{-1} +\varepsilon}$.
  Then there exists $X_0(\varepsilon,R)$ such that
  if $X \geq X_0(\varepsilon,R)$,
  \begin{displaymath}
    \abs{\{ a\in\mathcal{A} : \Omega (a)\le R \text{ and } p \mid
      a\Rightarrow p\ge D^{1/4}\}}
    \asymp_{\varepsilon,R} XV(D^{1/4}).
  \end{displaymath}
\end{lemma}
\begin{proof}
  See for instance \cite[Chapter 9]{halberstam-1974},
  \cite[Chapter 5, Corollary~1.2]{greaves-2001}
  or \cite[Theorem 25.1]{friedlander-iwaniec-opera-cribro-2010}.
\end{proof}
\begin{remark}
  We have chosen to use the weights of Richerts in place of the
  weights of Greaves \cite[Chapter~5]{greaves-2001} which are more
  elaborate but less simple to employ and would only produce an
  improvement numerically small compared to the values of $\Omega_b$
  given in Theorem~\ref{thm-almost-prime} which are all $\ge 200$.
\end{remark} 
 
\bigskip

Let us now specify the situation in which we intend to apply the sieve
methods above.

For any integer $\lambda\geq 1$, any real number $z\geq 2$ and any
$i\in\{1,\ldots,b-1\}$ with $\gcd(i,b)=1$, we define
\begin{equation}\label{eq:def_P_lambda_i}
  \mathscr{P}_{\lambda,i}
  =
  \{i b^{\lambda-1}\leq p < (i+1)b^{\lambda-1}: p \text{ prime}\}
\end{equation}
and
\begin{equation}\label{eq:def_varTheta_lambda_z}
  \varTheta_i(\lambda,z)
  = |\{p \in \mathscr{P}_{\lambda,i}:~P^{-}(R_\lambda(p))\geq z\}|.
\end{equation}
For any integer~$n\in\{0,\ldots,b^\lambda-1\}$, since
\begin{displaymath}
  n = \sum_{j=0}^{\lambda-1} \varepsilon_{j}(n) \,b^j
  \quad\text{ and }\quad
  R_\lambda(n) = \sum_{j=0}^{\lambda-1}
  \varepsilon_j(n) \ b^{\lambda-1-j},
\end{displaymath}
we have the following congruences modulo $b$, $b-1$ and $b+1$:
\begin{align}\label{eq:relations_n_reverse_n}
  \nonumber
  & R_\lambda(n) \equiv \varepsilon_{\lambda-1}(n) \bmod b,
  \\
  & R_\lambda(n) \equiv n \bmod (b-1),
  \\
  \nonumber
  & R_\lambda(n) \equiv (-1)^{\lambda-1} n \bmod (b+1).
\end{align}
We assume that $\lambda \geq 3$. Since $b+1 < b^2 \leq b^{\lambda-1}$
and $\gcd(i,b)=1$, we deduce that for any
$p \in \mathscr{P}_{\lambda,i}$,
\begin{equation}\label{eq:gcd-R_lambda-b^3-b}
  \gcd(R_\lambda(p), b(b^2-1)) = 1.
\end{equation}
This implies that
\begin{displaymath}
  \varTheta_i(\lambda,z)
  =
  |\{p \in \mathscr{P}_{\lambda,i}:~\gcd(R_\lambda(p),P(z))=1\}|
  =
  S(\mathcal{A}_i,\mathcal{P},z)
\end{displaymath}
where
\begin{displaymath}
  \mathcal{A}_i = (R_\lambda(p))_{p \in \mathscr{P}_{\lambda,i}},
  \qquad 
  \mathcal{P}=\{p \text{ prime}\},
  \qquad
  P(z) = \prod_{\substack{p<z\\p \in \mathcal{P}}} p.
\end{displaymath}
For $d\geq 1$, we put
\begin{equation}\label{def_T_lambda_d}
  T_{\lambda,i}(d)
  = |\{ p \in \mathscr{P}_{\lambda,i}:~d \mid R_\lambda(p)\}|.
\end{equation}
For any prime number $p_0$ such that $p_0 \dv b(b^2-1)$,
using \eqref{eq:gcd-R_lambda-b^3-b},
we get $T_{\lambda,i}(p_0)=0$.
For $p \in \mathscr{P}_{\lambda,i}$ randomly chosen and for any prime
number $p_0$ such that $p_0 \nmid b(b^2-1)$, the probability that
$p_0 \mid R_\lambda(p)$ is expected to be about $p_0^{-1}$.  Moreover
for distinct prime numbers $p_1$ and $p_2$, the events
``$p_1 \mid R_\lambda(p)$'' and ``$p_2 \mid R_\lambda(p)$'' are
expected to be independent.  These heuristics lead us to define
$E_{\lambda,i}(d)$ for $d\geq 1$ by
\begin{equation}\label{eq:def_Rd}
  T_{\lambda,i}(d) = g(d)\, |\mathscr{P}_{\lambda,i}| + E_{\lambda,i}(d),
\end{equation} 
where $g$ is the multiplicative function defined for any prime number
$p$ and any integer $\nu\ge 1$ by
\begin{equation}\label{eq:def_g}
  g(p^\nu)=
  \left\{
    \begin{array}{ll}
      p^{-\nu} & \text{ if } p \notdv b(b^2-1),\\
      0     & \text{ otherwise}.
    \end{array}
  \right.
\end{equation} 
We observe that if
\begin{math}
  \gcd(d, b(b^2-1)) \neq 1
\end{math}
then 
\begin{displaymath}
  T_{\lambda,i}(d) = g(d) = E_{\lambda,i}(d) = 0
  .
\end{displaymath}
Let
\begin{displaymath}
  V(z)
  =
  \prod_{\substack{p<z\\p\in\mathcal{P}}}\left(1-g(p)\right)
  =
  \prod_{\substack{p < z\\p \, \nmid \, b(b^2-1)}} \left(1-p^{-1}\right).
\end{displaymath}
The {\it Mertens formula\/} (see for example \cite[Part~I,
Theorem~1.12]{tenenbaum-ITAN-2022}) implies that
\begin{equation}\label{eq:asymp_formula_V}
  V(z)
  =
  \frac{e^{-\gamma}}{\log z}\left(1+O \left(\frac{1}{\log z}\right)\right)
  \prod_{\substack{p<z\\p \dv b(b^2-1)}} \left(1-p^{-1}\right)^{-1}
\end{equation}
where
$\gamma$ denotes Euler's constant and that there exists an absolute
constant $L>0$ such that for any real numbers $2\leq w < z$, we have
\begin{displaymath}
  \prod_{\substack{w\leq p<z\\p\in\mathcal{P}}}(1-g(p))^{-1}
  =
  \frac{V(w)}{V(z)} \leq \left(\frac{\log z}{\log w}\right)
  \left(1+\frac{L}{\log w}\right).
\end{displaymath}
We are now ready to apply Lemma~\ref{lemma:linear-sieve}. For any real
numbers $2 \leq z \leq D$, we have
\begin{equation}\label{eq:upper_bound_Theta_lambda_z}
  \varTheta_i(\lambda,z)
  <
  \left(F\left(\tfrac{\log D}{\log z}\right)
    + O\left((\log D)^{-1/6}\right)\right)
  \abs{\mathscr{P}_{\lambda,i}} V(z)
  +\sum_{\substack{d<D\\d \mid P(z)}}
  \abs{E_{\lambda,i}(d)},
\end{equation}
\begin{equation}\label{eq:lower_bound_Theta_lambda_z}
  \varTheta_i(\lambda,z)
  >
  \left(f\left(\tfrac{\log D}{\log z}\right)
    + O\left((\log D)^{-1/6}\right)\right)
  \abs{\mathscr{P}_{\lambda,i}} V(z)
  -\sum_{\substack{d<D\\d \mid P(z)}}
  \abs{E_{\lambda,i}(d)}.
\end{equation}

In order to bound the sum over
$d$ on the right-hand side of \eqref{eq:upper_bound_Theta_lambda_z}
and \eqref{eq:lower_bound_Theta_lambda_z}, we will establish in
Section~\ref{section:proof_thm_distribution-level} the following
consequence of Theorem~\ref{theorem:type_bombieri-vinogradov}.
\begin{theorem}\label{theorem:distribution-level}
  For any integer $b\geq 2$, there exists an explicit
  $\xi_0 = \xi_0(b) >0$ (defined by \eqref{eq:def_xi_0}) such that,
  given $\xi \in\left]0,\xi_0\right[$, we can find $c = c(b,\xi) >0$
  and $\lambda_0(b,\xi)\geq 2$ with the property that for any integer
  $\lambda\geq\lambda_0(b,\xi)$, we have
  \begin{displaymath}
    \sum_{\substack{d \leq b^{\xi \lambda}\\\gcd(d,b(b^2-1))=1}}
    \sup_{1\leq i \leq b-1}
    \abs{E_{\lambda,i}(d)}
    \ll
    b^{\lambda-c\sqrt{\lambda}}.
  \end{displaymath}
\end{theorem}
Combined with \eqref{eq:upper_bound_Theta_lambda_z} and
\eqref{eq:lower_bound_Theta_lambda_z}, this will allow us to prove in
Section~\ref{section:proof_bounds_Theta_lambda_z} the following upper
and lower bound of $\varTheta_i(\lambda,b^{\xi \lambda})$.
\begin{theorem}\label{theorem:theta_upper_bound}
  For any integer $b\geq 2$ and real number $\xi>0$, we
  have for any integer $\lambda \geq 1$,
  \begin{displaymath}
    \sup_{1\leq i \leq b-1} \varTheta_i(\lambda,b^{\xi\lambda})
    \ll_{\xi} \frac{b^{\lambda}}{\lambda^2}.
  \end{displaymath} 
\end{theorem}
\begin{theorem}\label{theorem:theta_lower_bound}
  For any integer $b\geq 2$, there exists an explicit
  $\xi_0 = \xi_0(b) >0$ (defined by \eqref{eq:def_xi_0}) such that,
  given $\xi \in\left]0,\xi_0/2\right[$, we can find
  $\lambda_0(b,\xi) \geq 2$ with the property that for any integer
  $\lambda \geq \lambda_0(b,\xi)$, we have
  \begin{displaymath}
    \inf_{1\leq i \leq b-1} \varTheta_i(\lambda,b^{\xi\lambda})
    \gg_{\xi} \frac{b^{\lambda}}{\lambda^2}.
  \end{displaymath}
\end{theorem}
We will derive Theorem~\ref{theorem:true-reversible-primes} from
Theorem~\ref{theorem:theta_upper_bound} in
Section~\ref{section:proof_thm_upper_bound_reversible_primes}.

We will see in Section~\ref{section:proof_thm_almost_primes_weaker}
that Theorem~\ref{theorem:theta_lower_bound} implies a weaker version
of Theorem~\ref{thm-almost-prime}, namely
Theorem~\ref{thm-almost-prime} with $\Omega_b$ replaced by
$\widetilde{\Omega}_b \approx 2 \Omega_b$ ($\widetilde{\Omega}_b$ will
be defined by \eqref{eq:def_Omega_b_tilde}).  In order to prove
Theorem~\ref{thm-almost-prime}, we will apply the weighted sieve
stated above: in
Section~\ref{section:proof_theorem:application_weighted_sieve}, we
will combine Theorem~\ref{theorem:distribution-level} with
Lemma~\ref{lemma-weighted-sieve} to obtain the following result, from
which we will derive Theorem~\ref{thm-almost-prime} in
Section~\ref{section:proof_thm_almost_primes}.

\begin{theorem}\label{theorem:application_weighted_sieve}
  For any integer $b\geq 2$, there exists an explicit
  $\xi_0 = \xi_0(b) >0$ (defined by \eqref{eq:def_xi_0}) such that,
  given $\xi \in\left]0,\xi_0\right[$ and an integer
  $R\geq 2$ such that $\Lambda_R > \xi^{-1}$,
  we can find $\lambda_0(b,\xi,R) \geq 2$ with the property that for
  any integer $\lambda \geq \lambda_0(b,\xi,R)$ and any
  $i \in\{1,\ldots,b-1\}$,
  \begin{displaymath}
    |\{p \in \mathscr{P}_{\lambda,i}: 
    \Omega(R_{\lambda}(p)) \leq R \text{ and } 
    P^{-}(R_\lambda(p))\geq b^{\frac{\xi}{4}\lambda}\}|
    \asymp_{\xi,R}
    \frac{b^{\lambda}}{\lambda^2}.
  \end{displaymath}
\end{theorem}

\section{Combinatorial identity}
\label{section:combinatorial-identity}

The following lemma is a consequence of Vaughan's identity
(see for instance~\cite[Prop.\ 13.4 p.\ 345]{iwaniec-kowalski-2004}).

\begin{lemma}\label{lemma:bound_sum_lambda}
  For any real numbers $1 < u^3< y \leq x$ and any function
  $f: \left[y,x\right[\cap\Z \to \C$, we have
  \begin{multline*}
    \abs{\sum_{y \leq n < x} \Lambda(n) f(n)}
    \leq
    \\
    (\log x)\left(
      3
      \sup_{t\in [y,x]}
      \sum_{m\leq u}
      \abs{\sum_{\substack{n\\y \leq mn < t}} f(mn)}
      +
      4\sup_{(z_n)}\sum_{u<m\leq \sqrt{x}}
      \abs{\sum_{\substack{n\\ y \leq mn < x}} z_n \,f(mn)}
    \right)
  \end{multline*}
  where the second supremum is taken over all sequences $(z_n)$ of
  complex numbers with $\abs{z_n}\leq 1$.
\end{lemma}
\begin{proof}
  For $s\in\C$ with $\Re(s)>1$ and $u > 1$
the Dirichlet series
\begin{displaymath}
  \zeta(s) = \sum_{n\geq 1} \frac{1}{n^s},\
  \zeta'(s) = - \sum_{n\geq 1} \frac{\log n}{n^s},\
  G_u(s) = \sum_{n\leq u} \frac{\mu(n)}{n^s},\
  F_u(s) = \sum_{n\leq u} \frac{\Lambda(n)}{n^s}
\end{displaymath}
($\zeta$ denotes the Riemann zeta function)
satisfy
\begin{equation}\label{eq:dirichlet-series-equality}
  - \frac{\zeta'}{\zeta}
  =
  F_u 
  - \zeta' G_u 
  - \zeta F_u G_u 
  + \zeta \left( \frac{1}{\zeta} - G_u \right)
  \left( -\frac{\zeta'}{\zeta} - F_u \right)
  .
\end{equation}
Identifying the coefficient of $n^{-s}$ on each side of
\eqref{eq:dirichlet-series-equality} we obtain for any function
$f: \left[y,x\right[\cap\Z \to \C$ and for $1 < u < y \leq x$:
\begin{displaymath}
  \sum_{y \leq n < x} \Lambda(n) f(n)
  =
  S_1 - S_2 + S_3
\end{displaymath}
with
\begin{align*}
  S_1
  &
    =
    \sum_{\substack{1\leq m \leq u\\ y \leq mn < x}} \mu(m)\, \log(n)\, f(mn)
  \\
  S_2
  &
    =
    \sum_{\substack{1\leq m_1 \leq u\\ 1\leq m_2 \leq u\\ y \leq m_1 m_2 n < x}}
  \mu(m_1)\, \Lambda(m_2)\, f(m_1 m_2 n)
  \\
  S_3
  &
    = 
    \sum_{\substack{u < m < x\\ u < n_1 < x\\ y \leq m n_1 n_2 < x}}
  \mu(m)\, \Lambda(n_1)\, f(m n_1 n_2).
\end{align*}

We assume that $u^3 < y$.
We will see that the estimation of $S_1$, $S_2$, $S_3$ can be reduced
to the estimation of two types of sums:
\begin{itemize}
\item the type I sums which are of the form
  \begin{equation}\label{eq:Vaughan-Type-I-sums}
    \sum_{m\leq u} z_m \sum_{\substack{n\\ y \leq mn < t}} f(mn)
    \qquad
    \text{(Type I)},
  \end{equation}
  where $t\in[y,x]$ and $z_m\in \C$ with
  $\abs{z_m}\leq 1$, whose absolute value is bounded by
  \begin{displaymath}
    \sup_{t\in [y,x]}
      \sum_{m\leq u}
      \abs{\sum_{\substack{n\\y \leq mn < t}} f(mn)},
  \end{displaymath}
\item the type II sums which are of the form
  \begin{equation}\label{eq:Vaughan-Type-II-sums}
    \sum_{u<m\leq \sqrt{x}} z'_m \sum_{\substack{n\\ y \leq mn < x}} z_n \,f(mn)
    \qquad
    \text{(Type II)},
  \end{equation}
  where $z'_m\in\C$ with $\abs{z'_m}\leq 1$, $z_n\in\C$ with
  $\abs{z_n}\leq 1$, whose absolute value is bounded by
  \begin{displaymath}
    \sup_{(z_n)}\sum_{u<m\leq \sqrt{x}}
    \abs{\sum_{\substack{n\\ y \leq mn < x}} z_n \,f(mn)}
  \end{displaymath}
  where the supremum is taken over all sequences $(z_n)$ of
  complex numbers with $\abs{z_n}\leq 1$.
\end{itemize}

In order to estimate $S_1$, we use the equality  
\begin{math}
  \log n = \log\left(\frac{x}{m}\right)-\int_{mn}^{x} \frac{dt}{t}
\end{math}
to write
\begin{displaymath}
  S_1 = (\log x)(S_{11} - S_{12})
\end{displaymath}
where
\begin{displaymath}
  S_{11}
  =
  \sum_{1\leq m\leq u} \frac{\mu(m)\log\left(\frac{x}{m}\right)}{\log x}
  \sum_{\substack{n \\ y \leq mn < x}} f(mn)
\end{displaymath}
is a type I sum and, reverting the summations,
\begin{displaymath}
  S_{12}
  =
  \frac{1}{\log x}\int_{y}^{x}
  \sum_{1\leq m\leq u} \mu(m)
  \sum_{\substack{n \\ y \leq mn < t}} f(mn)
  \ 
  \frac{dt}{t},
\end{displaymath}
hence
\begin{displaymath}
  \abs{S_{12}}
  \leq
  \sup_{t\in [y,x]}\abs{
      \sum_{1\leq m\leq u}
      \mu(m)
      \sum_{\substack{n\\y \leq mn < t}} f(mn)},
\end{displaymath}
which can be estimated by a type I sum.

In order to estimate $S_2$,
we write
\begin{displaymath}
  S_2
  =
  \sum_{m\leq u^2} c_m   \left(\log u^2 \right)
  \sum_{\substack{n \\ y \leq mn < x}} f(mn)
  ,
\end{displaymath}
with, for $m\leq u^2$, 
\begin{displaymath}
  c_m
  =
  \frac1{\log u^2}
  \sum_{\substack{m_1 \leq u\\ m_2 \leq u\\ m=m_1 m_2}}
  \mu(m_1)\Lambda(m_2)
  ,
\end{displaymath}
which satisfies
\begin{displaymath}
  \abs{c_m}
  \leq
  \frac1{\log u^2}
  \sum_{m_2\dv m} \Lambda(m_2) =   \frac{\log m}{\log u^2} \leq 1.
\end{displaymath}
We have
\begin{displaymath}
  S_2
  =
  \left( \log u^2 \right)
  \left( S_{21}+S_{22}+S_{23} \right)
  ,
\end{displaymath}
where
\begin{displaymath}
  S_{21} =
  \sum_{m\leq u} c_m
  \sum_{\substack{n \\ y \leq mn < x}}
  f(mn)
\end{displaymath}
is a type I sum,
\begin{displaymath}
  S_{22} =
  \sum_{u< m\leq \sqrt{x}} \left( c_m \one_{m\leq u^2} \right)
  \sum_{\substack{n \\ y \leq mn < x}} f(mn)
\end{displaymath}
is a type II sum,
and, if $u^2>\sqrt{x}$ (otherwise $S_{23}=0$),
\begin{displaymath}
  S_{23} =
  \sum_{\sqrt{x}< m\leq u^2} c_m
  \sum_{\substack{n \\ y \leq mn < x}} f(mn)
  ,
\end{displaymath}
and since $\frac{y}{u^2} > u$, 
we have
\begin{displaymath}
  S_{23} 
  =
  \sum_{u < n \leq \sqrt{x}} 
  \sum_{\substack{m \\ y \leq mn < x}}
  \left( c_m \one_{\sqrt{x}< m\leq u^2} \right)
  f(mn),
\end{displaymath}
which is a type II sum
(with $m$ and $n$ exchanged).

In order to estimate $S_3$, we write
\begin{displaymath}
  S_3
  =
  (\log x)
  \sum_{u< m < \frac{x}{u}} \mu(m)
  \sum_{\substack{n > u\\ y \leq mn < x}}
  z_n\, f(mn),
\end{displaymath}
with, for $1\leq n\leq x$,
\begin{displaymath}
  0
  \leq
  z_n
  =
  \frac{1}{\log x}
  \sum_{\substack{u< n_1 \\ n=n_1 n_2}} \Lambda(n_1)
  \leq 
  \frac{1}{\log x} \sum\limits_{n_1\dv n} \Lambda(n_1) 
  = 
  \frac{\log n}{\log x}  \leq 1
  .
\end{displaymath}
We have
\begin{displaymath}
  S_3 = (\log x) \left( S_{31} + S_{32} \right)
\end{displaymath}
with
\begin{displaymath}
  S_{31}
  =
  \sum_{u< m \leq \sqrt{x}} \mu(m)
  \sum_{\substack{n \\ y \leq mn < x}}
  (z_n\one_{n>u})\, f(mn),
\end{displaymath}
which is a type II sum and
\begin{align*}
  S_{32}
  &=
  \sum_{\sqrt{x} < m < \frac{x}{u}} \mu(m)
  \sum_{\substack{n > u\\ y \leq mn < x}}
    z_n\, f(mn)
  \\
  &=
  \sum_{u < n \leq \sqrt{x}}
  z_n
  \sum_{\substack{m \\ y \leq mn < x}}
  \left( \mu(m) \one_{\sqrt{x} < m} \right)
  \, f(mn)
  ,
\end{align*}
which is a type II sum (with $m$ and $n$ exchanged). This ends the
proof of the lemma.
\end{proof}

The following result permits us to combine
Lemma~\ref{lemma:bound_sum_lambda} with a further averaging.
\begin{lemma}\label{lemma:maj_sum_alpha_vaughan}
  Let $\mathcal{A}$ be a nonempty set of real numbers and $W$ be a non
  negative real valued function on~$\mathcal{A}$.
  Let $\beta_1,\beta_2$ be real numbers such that $0 < \beta_1 < 1/3$ and 
  $1/2 < \beta_2 < 1$.  
  Let $b\geq2$ and $\lambda \geq 3$ be integers such that 
  \begin{math}
    1 < \left(b^{\floor{\beta_1\lambda}}-1\right)^3 < b^{\lambda-1}
  \end{math}
  and
  \begin{math}
    \ceil{\frac{\lambda}2}+1 \leq \beta_2 \lambda.
  \end{math}  
  Let $f$ be a complex valued function on
  $\mathcal{A}\times\{b^{\lambda-1},\ldots,b^{\lambda}-1\}$. We assume
  that
  \begin{itemize}
  \item (extended Type II)
    for any integer $\mu$ such that
    \begin{math}
      \beta_1\lambda < \mu \leq \beta_2\lambda, 
    \end{math}
    \begin{equation}\label{eq:maj_gen_type_II}
      \sum_{\alpha\in \mathcal{A}}
      W(\alpha)
      \sup_{t \in \left[b^{\lambda-1},b^{\lambda}\right]}
      \sup_{(z_n)}
      \sum_{b^{\mu-1} \leq m < b^{\mu}}
      \abs{
        \sum_{\substack{n\\ b^{\lambda-1} \leq mn < t}} z_{n} f(\alpha,mn)
      }
      \leq U
      ,
    \end{equation}
    where the second supremum is taken over all sequences $(z_n)$ with
    $z_n\in\C$ and $\abs{z_n}\leq 1$,
  \item (extended Type I) for any integer $\mu$ such that
    \begin{math}
      1 \leq \mu \leq \beta_1 \lambda,
    \end{math}
    \begin{equation}\label{eq:maj_gen_type_I}
      \sum_{\alpha \in \mathcal{A}}
      W(\alpha)
      \sup_{t\in\left[b^{\lambda-1},b^{\lambda}\right]}
      \sum_{b^{\mu-1} \leq m < b^{\mu}}
      \abs{
        \sum_{\substack{n\\ b^{\lambda-1}\leq mn < t}}
        f(\alpha,mn)
      }
      \leq U
      .
    \end{equation}
  \end{itemize}
  Then
  \begin{displaymath}
    \sum_{\alpha \in \mathcal{A}}
    W(\alpha)
    \sup_{t\in\left[b^{\lambda-1},b^{\lambda}\right]}
    \abs{\sum_{b^{\lambda-1} \leq n < t}
      \Lambda(n)
      f(\alpha,n)}
    \leq 
     7 \lambda (\log b^{\lambda}) U.
  \end{displaymath}
\end{lemma}
\begin{proof}
  Let $\alpha \in \mathcal{A}$ and
  $t\in\left[b^{\lambda-1},b^{\lambda}\right]$. By applying
  Lemma~\ref{lemma:bound_sum_lambda} with $y=b^{\lambda-1}$, $x=t$ and
  $f(n)$ replaced by $f(\alpha,n)$,
  we obtain for any real number $u$ such that $1<u^3 < b^{\lambda-1}$:
  \begin{equation}\label{eq:decomp_sum_lambda_f_alpha}
    \abs{\sum_{b^{\lambda-1} \leq n < t} \Lambda(n) f(\alpha,n)}
    \leq
    (\log t)\left(
      3 S_I(\alpha,t)
      +
      4 S_{II}(\alpha,t)
    \right)
  \end{equation}
  with
  \begin{displaymath}
    S_I(\alpha,t) = 
    \sup_{t'\in [b^{\lambda-1},t]}
      \sum_{m\leq u}
      \abs{\sum_{\substack{n\\b^{\lambda-1} \leq mn < t'}} f(\alpha,mn)}
  \end{displaymath}
  and
  \begin{displaymath}
    S_{II}(\alpha,t) = 
    \sup_{(z_n)}\sum_{u<m\leq \sqrt{t}}
    \abs{\sum_{\substack{n\\ b^{\lambda-1} \leq mn < t}} z_n \,f(\alpha,mn)}
  \end{displaymath}
  where the supremum is taken over all sequences $(z_n)$ of complex
  numbers with $\abs{z_n}\leq 1$.

  Since 
  \begin{math}
    1 < \left(b^{\floor{\beta_1\lambda}}-1\right)^3 < b^{\lambda-1}
    ,
  \end{math}
  we can choose $u = b^{\floor{\beta_1\lambda}}-1$.  By $b$-adic
  splitting, for any $t'\in [b^{\lambda-1},t]$, we have
  \begin{displaymath}
    \sum_{m\leq u}
    \abs{\sum_{\substack{n\\b^{\lambda-1} \leq mn < t'}} f(\alpha,mn)}
    = 
    \sum_{\mu=1}^{\floor{\beta_1\lambda}}
    \sum_{b^{\mu-1} \leq m < b^{\mu}}
      \abs{
        \sum_{\substack{n\\ b^{\lambda-1}\leq mn < t'}}
        f(\alpha,mn)
    }
  \end{displaymath}
  hence
  \begin{equation}\label{eq:maj_sum_m_sum_n_I}
    \sum_{m\leq u}
    \abs{\sum_{\substack{n\\b^{\lambda-1} \leq mn < t'}} f(\alpha,mn)}
    \leq
    \sum_{\mu=1}^{\floor{\beta_1\lambda}}
      \sup_{t\in\left[b^{\lambda-1},b^{\lambda}\right]}
      \sum_{b^{\mu-1} \leq m < b^{\mu}}
      \abs{
      \sum_{\substack{n\\ b^{\lambda-1}\leq mn < t}}
    f(\alpha,mn)
    }
  \end{equation}
  and for any sequence $(z_n)$ of complex numbers with
  $\abs{z_n}\leq 1$, since 
  \begin{displaymath}
    b^{\ceil{\frac{\lambda}2}+1}-1 \geq b^{\frac{\lambda}2} \geq
    \floor{\sqrt{t}},
  \end{displaymath}
  we have
  \begin{align}\label{eq:maj_sum_m_sum_n_II}
    \sum_{u<m\leq \sqrt{t}}
    &
      \abs{\sum_{\substack{n\\ b^{\lambda-1} \leq mn < t}} z_n \,f(\alpha,mn)}
    \\
    \nonumber
    &\leq
    \sum_{\mu = \floor{\beta_1\lambda}+1}^{\ceil{\frac{\lambda}2}+1}
    \sum_{b^{\mu-1} \leq m < b^{\mu}}
    \abs{\sum_{\substack{n\\ b^{\lambda-1} \leq mn < t}} z_n \,f(\alpha,mn)}
    \\
    \nonumber
    &\leq
    \sum_{\mu = \floor{\beta_1\lambda}+1}^{\ceil{\frac{\lambda}2}+1}
    \sup_{t\in\left[b^{\lambda-1},b^{\lambda}\right]} \sup_{(z_n)}
    \sum_{b^{\mu-1} \leq m < b^{\mu}}
    \abs{\sum_{\substack{n\\ b^{\lambda-1} \leq mn < t}} z_n \,f(\alpha,mn)}
    .
  \end{align}
  Combining~\eqref{eq:maj_sum_m_sum_n_I} and
  \eqref{eq:maj_gen_type_I}, we obtain
  \begin{align*}
    \MoveEqLeft
    \sum_{\alpha \in \mathcal{A}}
      W(\alpha)
    \sup_{t\in\left[b^{\lambda-1},b^{\lambda}\right]}
    S_I(\alpha,t)
    \\
    &
      \leq
      \sum_{\mu=1}^{\floor{\beta_1\lambda}}
      \sum_{\alpha \in \mathcal{A}}
      W(\alpha)
      \sup_{t\in\left[b^{\lambda-1},b^{\lambda}\right]}
      \sum_{b^{\mu-1} \leq m < b^{\mu}}
      \abs{
      \sum_{\substack{n\\ b^{\lambda-1}\leq mn < t}}
    f(\alpha,mn)
    }
    &
      \leq 
      \lambda U
      .
  \end{align*}
  Combining~\eqref{eq:maj_sum_m_sum_n_II} and
  \eqref{eq:maj_gen_type_II}, since 
  \begin{math}
    \ceil{\frac{\lambda}2}+1 \leq \beta_2 \lambda,
  \end{math}
  we obtain
  \begin{align*}
    \MoveEqLeft
    \sum_{\alpha \in \mathcal{A}}
      W(\alpha)
      \sup_{t\in\left[b^{\lambda-1},b^{\lambda}\right]}
      S_{II}(\alpha,t)
    \\
    &
      \leq
      \sum_{\mu = \floor{\beta_1\lambda}+1}^{\ceil{\frac{\lambda}2}+1}
      \sum_{\alpha \in \mathcal{A}}
      W(\alpha)
      \sup_{t\in\left[b^{\lambda-1},b^{\lambda}\right]}
      \sup_{(z_n)}
      \sum_{b^{\mu-1} \leq m < b^{\mu}}
      \abs{
      \sum_{\substack{n\\ b^{\lambda-1}\leq mn < t}}
    z_nf(\alpha,mn)
    }
    \\
    &\leq \lambda U
      .
  \end{align*}
  By~\eqref{eq:decomp_sum_lambda_f_alpha}, it follows that
  \begin{displaymath}
    \sum_{\alpha \in \mathcal{A}}
    W(\alpha)
    \sup_{t\in\left[b^{\lambda-1},b^{\lambda}\right]}
    \abs{\sum_{b^{\lambda-1} \leq n < t}
      \Lambda(n)
      f(\alpha,n)}
    \leq 
    7 \lambda (\log b^{\lambda}) U,
  \end{displaymath}
  which completes the proof.
\end{proof}

\section{Miscellaneous definitions and lemmas}\label{section:lemmas}

\begin{lemma}\label{lemma:van-der-corput}
  For all complex numbers $z_1,\ldots,z_N$ 
  and all integers $k\geq 1$ and $R\geq 1$ we have
  \begin{multline}\label{eq:van-der-corput}
    \abs{\sum_{1\le n\le N} z_n}^2
    \\
    \leq
    \frac{N+kR-k}{R}
    \
    \Re
    \Bigg(
    \sum_{1\le n\le N} \abs{z_n}^2
    + 2
    \sum_{1\le r< R}\left(1-\frac{r}{R}\right)
    \
    \sum_{1\le n\le N-kr} z_{n+kr} \conjugate{z_n} 
    \Bigg),
  \end{multline}
  where $\Re(z)$ denotes the real part of $z$.
\end{lemma}
\begin{proof}
  This is a variant of the van der Corput inequality,
  see for example in Lemma 17 of \cite{mauduit-rivat-2009}.
\end{proof}

\begin{lemma}\label{lemma:sum_inverse_sinus}
  Let $(a,m)\in\Z^2$ with $m\ge 1$, $\delta=\gcd(a,m)$ and $b\in\R$. 
  Let $(M,N)\in\Z^2$ with $N\geq 1$.
  For all real numbers $U>0$ we have
  \begin{multline}
    \label{eq:sum_inverse_sinus}
    \sum_{M+1\le n\le M+N} 
    \min \left( U, \abs{\sin \pi \tfrac{an+b}{m}}^{-1} \right)
    \\
    \leq
    \ceil{\frac{\delta N}{m}}
    \left(
    \min\left( 
      U, \abs{\sin \pi\tfrac{\delta\, \norm{b/\delta}}{m} }^{-1} 
    \right)
    + 
    \frac{2\, m}{\pi \delta} \log \frac{2 \, m}{\delta}
  \right)
  .
  \end{multline}
\end{lemma}
\begin{proof}
  Let $a'=a/\delta$, $m'=m/\delta$ and $b'=b/\delta$.
  The function
  \begin{displaymath}
    n \mapsto
    \min \left( U, \abs{\sin \pi \tfrac{an+b}{m}}^{-1} \right)
    =
    \min \left( U, \abs{\sin \pi \tfrac{a'n+b'}{m'}}^{-1} \right)
  \end{displaymath}
  is periodic modulo $m'$.
  Extending the summation if necessary by a few terms to reach a
  length multiple of $m'$,
  we have
  \begin{displaymath}
    \sum_{M+1\le n\le M+N} 
    \min \left( U, \abs{\sin \pi \tfrac{an+b}{m}}^{-1} \right)
    \leq
    \ceil{\frac{N}{m'}}
    \sum_{n=0}^{m'-1} 
    \min \left( U, \abs{\sin \pi \tfrac{a'n+b'}{m'}}^{-1} \right)
    .
  \end{displaymath}
  If $m'=1$, then \eqref{eq:sum_inverse_sinus} holds immediately.
  If $m'\geq 2$, by \cite[Lemma 6]{mauduit-rivat-2010}
  we have
  \begin{align*}
    \sum_{n=0}^{m'-1} 
    &
      \min \left( U, \abs{\sin \pi \tfrac{a'n+b'}{m'}}^{-1} \right)    
    \\
    &
      \leq 
      \min\left( 
      U, \abs{\sin \pi\tfrac{\norm{b'}}{m'} }^{-1} 
      \right)
      + 
      \frac{2m'}{\pi}
      \left(
      \frac{\frac{\pi}{2 m'}}{\sin\frac{\pi}{2 m'}}
      +
      \log \frac{2m'}{\pi}
      \right)
    \\
    &
      \leq 
      \min\left( 
      U, \abs{\sin \pi\tfrac{\norm{b'}}{m'} }^{-1} 
      \right)
      + 
      \frac{2m'}{\pi}
      \left(
      \frac{\frac{\pi}{4}}{\sin\frac{\pi}{4}}
      +
      \log \frac{2m'}{\pi}
      \right)
    \\
    &
      \leq 
      \min\left( 
      U, \abs{\sin \pi\tfrac{\norm{b'}}{m'} }^{-1} 
      \right)
      + 
      \frac{2\, m'}{\pi} \log (2\, m')
      ,
  \end{align*}
  which completes the proof of \eqref{eq:sum_inverse_sinus}.
\end{proof}

\begin{lemma}\label{lemma:gcd-sum}
  Let $m\geq 1$ be an integer and $A\geq 1$ a real number.
  We have
  \begin{equation}\label{eq:gcd-sum}
    \frac1A \sum_{1\leq a \leq A} \gcd(a,m)
    \leq
    \tau(m)
    .
  \end{equation}    
\end{lemma}
\begin{proof}
  We have
  \begin{displaymath}
    \sum_{1\leq a \leq A} \gcd(a,m)
    =
    \sum_{\substack{d\dv m\\ d\leq A}} d
    \sum_{\substack{1\leq a\leq A\\ \gcd(a,m)=d}} 1
    \leq
    \sum_{\substack{d\dv m\\ d\leq A}} d
    \sum_{\substack{1\leq a\leq A\\ d \dv a}} 1
    =
    \sum_{\substack{d\dv m\\ d\leq A}} d \floor{\frac{A}{d}}
    \leq
    A\ \tau(m),
  \end{displaymath}
  and \eqref{eq:gcd-sum} follows.
\end{proof}

\begin{lemma}\label{lemma:sigma-z}
  For any $z\in\C$, the arithmetic function
  $\sigma_z(n) = \sum_{d \dv n} d^z$ is multiplicative.\\
  For integers $b \geq 2$ and $\lambda \geq 1$, we have
  \begin{equation}\label{eq:tau-b-lambda}
    \left(1+\lambda\right)^{\omega(b)}
    \leq
    \sigma_0\left(b^{\lambda}\right)
    =
    \tau\left(b^{\lambda}\right)
    \leq \tau(b) \lambda^{\omega(b)}    
    .
  \end{equation}
  For $z\in\R$, $z<0$,
  we have the following upper bound,
  independent of $\lambda$:
  \begin{equation}\label{eq:sigma-z<0-b-lambda}
    \sigma_z\left(b^{\lambda}\right)
    =
    \sum_{d\dv b^\lambda} d^z
    <
    \prod_{p\dv b} \frac1{1-p^z}
    ,
  \end{equation}
  and for $z\in\R$, $z>0$,
  \begin{equation}\label{eq:sigma-z>0-b-lambda}
    \sigma_z\left(b^{\lambda}\right)
    <
    b^{z\lambda} \prod_{p\dv b} \frac1{1-p^{-z}}
    .
  \end{equation}
  Moreover, for $z\in\R$, $z<0$,
  we have the following upper bound,
  independent of $\lambda$:
  \begin{equation}\label{eq:tau-sigma-z-b-lambda}
    \sum_{d\dv b^\lambda} \tau(d) \, d^z
    <
    \prod_{p\dv b} \frac1{\left(1-p^z\right)^2}
    .
  \end{equation}
  
\end{lemma}
\begin{proof}
  The function $\sigma_z$ is multiplicative as the Dirichlet
  convolution of two multiplicative functions.

  For $p$ prime and integers $\nu\geq 1$, $\lambda\geq 1$ we have
  \begin{displaymath}
    \left(1+\lambda\right)^{\omega\left(p^\nu\right)}
    =
    1+\lambda
    \leq
    \tau\left(p^{\nu\lambda}\right)
    = 1+\nu\lambda
    \leq (1+\nu)  \lambda
    = \tau\left(p^\nu\right) \lambda^{\omega\left(p^\nu\right)}    
  \end{displaymath}
  and \eqref{eq:tau-b-lambda} follows by multiplicativity.

  For $p$ prime, integers $\nu\geq 1$, $\lambda\geq 1$
  and $z\in\R$, $z<0$ we have $0<p^z<1$ and
  \begin{displaymath}
    \sigma_z\left(p^{\nu \lambda}\right)
    =
    \sum_{j=0}^{\nu \lambda} p^{jz}
    <
    \sum_{j=0}^\infty p^{jz}
    =
    \frac1{1-p^z},
  \end{displaymath}
  and \eqref{eq:sigma-z<0-b-lambda} follows by multiplicativity.

  For $z\in\R$, $z>0$ we have for any integer $n\geq 1$,
  \begin{displaymath}
    \sigma_z(n) = \sum_{d \dv n} \left(\frac{n}{d}\right)^z
    = n^z \sigma_{-z}(n)
    ,
  \end{displaymath}
  and \eqref{eq:sigma-z>0-b-lambda}
  follows from \eqref{eq:sigma-z<0-b-lambda}.
  
  The arithmetic function
  \begin{math}
    f_z(n) = \sum_{d \dv n} \tau(d) d^z
  \end{math}
  is multiplicative as the Dirichlet convolution of two multiplicative
  functions. 
  For $p$ prime, integers $\nu\geq 1$, $\lambda\geq 1$
  and $z\in\R$, $z<0$ we have $0<p^z<1$ and
  \begin{displaymath}
    f_z\left(p^{\nu \lambda}\right)
    =
    \sum_{j=0}^{\nu \lambda} \tau(p^j) p^{jz}
    <
    \sum_{j=0}^\infty (j+1) p^{jz}
    =
    \frac1{\left(1-p^z\right)^2},
  \end{displaymath}
  and \eqref{eq:tau-sigma-z-b-lambda} follows by multiplicativity.
\end{proof}

\begin{lemma}\label{lemma:sum-function-fractions-gcd}
  For any integer $\mu\geq 1$
  and any function $\Phi:[0,1]\to\R^+$ we have
  \begin{displaymath}
    \sum_{b^{\mu-1} \leq m < b^\mu}
    \frac{1}{m}
    \sum_{0\leq k<m}
    \Phi\left(\frac{k}{m}\right)
    \leq
    b
    \sum_{m'<b^{\mu}}
    \frac{1}{m'}
    \sum_{\substack{0\leq k'<m'\\\gcd(k',m')=1}}
    \Phi\left(\frac{k'}{m'}\right)
    .
  \end{displaymath}
\end{lemma}
\begin{proof}
  Using $m\geq b^{\mu-1}$ and introducing $d = \gcd(k,m)$, we have
  \begin{displaymath}
    \sum_{b^{\mu-1} \leq m < b^\mu}
    \frac{1}{m}
    \sum_{0\leq k<m}
    \Phi\left(\frac{k}{m}\right)
    \leq
    \frac{1}{b^{\mu-1}}
    \sum_{1\leq d < b^\mu}
    \sum_{\substack{b^{\mu-1} \leq m < b^\mu\\ d\dv m}}
    \sum_{\substack{0\leq k<m\\\gcd(k,m)=d}}
    \Phi\left(\frac{k}{m}\right)
    .
  \end{displaymath}
  Writing $k=d k'$ and $m = d m'$ this is at most
  \begin{align*}
    \MoveEqLeft[4]
    \frac{1}{b^{\mu-1}}
    \sum_{1\leq d<b^{\mu}}
      \sum_{\frac{b^{\mu-1}}{d} \leq m' < \frac{b^\mu}{d}}
    \sum_{\substack{0\leq k'<m'\\\gcd(k',m')=1}}
    \Phi\left(\frac{k'}{m'}\right)
    \\
    &
      \leq
    \frac{1}{b^{\mu-1}}
    \sum_{1\leq m'<b^{\mu}}
    \sum_{d< \frac{b^{\mu}}{m'}}
    \sum_{\substack{0\leq k'<m'\\\gcd(k',m')=1}}
    \Phi\left(\frac{k'}{m'}\right)
    \\
    &\leq
    \frac{1}{b^{\mu-1}}
    \sum_{m'<b^{\mu}}
    \frac{b^\mu}{m'}
    \sum_{\substack{0\leq k'<m'\\\gcd(k',m')=1}}
    \Phi\left(\frac{k'}{m'}\right)
    ,
  \end{align*}
  and the result follows.
\end{proof}

\begin{lemma}\label{lemma:Euler-MacLaurin-order-1}
  Let $f: \R \to \R$ be a $1$-periodic continuous function of bounded
  variation on $[0,1]$, with total variation $V_f$ on $[0,1]$.
  For any integer $N\geq 1$ we have
  \begin{equation}\label{eq:Euler-MacLaurin-order-1}
    \abs{
      \frac1N
      \sum_{n=1}^N f\left( \frac{n}{N} \right)
      -
      \int_0^1 f(t) \, dt
    }
    \leq \frac1{2N} V_f
    .
  \end{equation}
\end{lemma}
\begin{proof}
  Using Stieltjes integral, by partial summation we have
  \begin{multline*}
    \sum_{n=1}^N f\left( \frac{n}{N} \right)
    -
    N \int_0^1 f\left( t \right) \, dt
    =
    \int_{0^+}^{1^+} f(t)
    \,
    d\left(\floor{Nt}+\frac12-Nt\right)
    \\
    =
    \left[f(t) \left(\floor{Nt}+\frac12-Nt\right) \right]_{0^+}^{1^+} 
    -
    \int_0^1 
    \left(\floor{Nt}+\frac12-Nt\right)
    \,
    df(t)
    \\
    =
    -
    \int_0^1 
    \left(\floor{Nt}+\frac12-Nt\right)
    \,
    df(t)
    ,
  \end{multline*}
  hence, denoting by $V_f(t)$ the total variation of $f$ on $[0,t]$,
  \begin{displaymath}
    \abs{
      \sum_{n=1}^N f\left( \frac{n}{N} \right)
      -
      N \int_0^1 f\left( t \right) \, dt
    }
    \leq
    \frac12 \int_0^1 dV_f(t)
    =
    \frac12 V_f
    ,
  \end{displaymath}
  which gives \eqref{eq:Euler-MacLaurin-order-1}.
\end{proof}

\begin{definition}\label{definition:well-spaced}
  Given $\delta>0$,
  we say that a sequence $(x_1,\ldots,x_N)\in\R^N$ is $\delta$ well
  spaced modulo $1$ if $\norm{x_i-x_j}\geq \delta$
  for all $1\leq i<j\leq N$.
\end{definition}

\begin{lemma}\label{lemma:sobolev-gallagher}
  Let $f: \R \to \R$ be a $1$-periodic continuous function of bounded
  variation on $[0,1]$, with total variation $V_f$ on $[0,1]$.
  For $\delta>0$
  and any sequence $(x_1,\ldots,x_N) \in \R^N$
  which is $\delta$ well spaced modulo~$1$, we have
  \begin{displaymath}
    \sum_{n=1}^N \abs{f(x_n)}
    \leq \frac{1}{\delta} \int_0^1 \abs{f(u)}
    du + \frac12 V_f
    .
  \end{displaymath}
\end{lemma}
\begin{proof}
  This generalized version of the Sobolev-Gallagher inequality
  is Lemma 2.2 of \cite{DMRSS_reversible_primes}.
\end{proof}

\begin{lemma}[Bernstein-Zygmund] \label{lemma:bernstein-zygmund}
  If $N\geq 0$ is an integer and 
  \begin{math}
    P(t) = \sum_{n=-N}^N a_n \e(n t)
  \end{math}
  is a trigonometric polynomial with complex coefficients $a_n$,
  then we have Bernstein's inequality:
  \begin{displaymath}
    \max_{t\in[0,1]} \abs{P'(t)} 
    \leq
    2\pi N\max_{t\in[0,1]} \abs{P(t)}
    ,
  \end{displaymath}  
  and, for any real number $p\geq 1$,
  we have Zygmund's inequality:
  \begin{displaymath}
    \left(\int_0^1 \abs{P'(t)}^p dt\right)^{1/p}
    \leq 2\pi N
    \left(\int_0^1 \abs{P(t)}^p dt\right)^{1/p}
    .
  \end{displaymath}
\end{lemma}
\begin{proof}
  See \cite[Chapter X, (3.18)]{zygmund-2002}.
  For more general inequalities involving entire functions of
  exponential type $\tau$,
  see \cite[Theorem 11.3.3 page 211]{boas-1954}.
\end{proof}
\begin{lemma}[Analytic form of the large sieve]
  \label{lemma:large-sieve}
  If $x_1,\ldots,x_R$ are real numbers $\delta$ well spaced modulo~$1$
  ($0<\delta\leq \frac12$),
  $M\in\Z$ and $a_{M+1},\ldots,a_{M+N} \in\C$ then
\begin{displaymath}
  \sum_{r=1}^R \abs{\sum_{n=M+1}^{M+N} a_n \e(n x_r)}^2
  \leq \left(N-1+\frac{1}{\delta}\right) \sum_{n=M+1}^{M+N} \abs{a_n}^2.
\end{displaymath} 
\end{lemma}
\begin{proof}
  See \cite[Theorem 3]{montgomery-1978}.
\end{proof}
\begin{lemma}\label{lemma:concavity-of-Dirichlet-kernel}
  For any integer $b\geq 2$,
  we define the function $K_b\in C^\infty(\R,\R)$ by
  \begin{equation}\label{eq:definition-Kb}
    \forall x\in\R\setminus\Z,\
    K_b(x) = \frac{\sin(\pi b x)}{b \sin(\pi x)}
    ;\quad
    \forall n\in\Z,\ K_b(n) = (-1)^{n(b-1)}
    .
  \end{equation}
  The function $K_b$ is decreasing on $[0,(2b-2)^{-1}]$,
  concave on $[0,(2b-2)^{-1}]$ and
  \begin{equation}\label{eq:Kb-logarithmic-derivative}
    \forall x \in \left]0,(2b-2)^{-1} \right[,\
    \frac{K'_b}{K_b}(x) = \pi b \cot(\pi b x) - \pi \cot(\pi x)
    .
  \end{equation}
  Moreover, 
  if $b$ is odd then $K_b$ is $1$-periodic and we have
  \begin{equation}\label{eq:Kb-for-b-odd}
    \forall x\in \R,\
    K_b(x) 
    = \frac{1}{b} + \frac{2}{b}
    \sum_{\ell = 1}^{(b-1)/2} \cos(2\pi \ell x)
    ,
  \end{equation}
  and if $b$ is even then $K_b$ is $2$-periodic with
  $K_b(x+1)=-K_b(x)$ for any $x\in\R$,
  and we have
  \begin{equation}\label{eq:Kb-for-b-even}
    \forall x\in \R,\
    K_b(x) 
    =
    \frac{2}{b}
    \sum_{0\leq \ell<b/2}
    \cos\left(2\pi \left(\ell+\frac12\right) x \right)
    .
  \end{equation}
\end{lemma}
\begin{proof}
  The function $K_b$ belongs to $C^\infty(\R,\R)$
  as the restriction on $\R$ of the entire function
  \begin{displaymath}
    z \mapsto \frac{\sin(\pi b z)}{b \sin(\pi z)}.
  \end{displaymath}
  For $b$ odd and $x\in\R\setminus\Z$ we have
  \begin{align*}
    K_b(x)
    &
      =
      \frac{\e\left(\frac{b x}{2}\right)-\e\left(-\frac{b x}{2}\right)}{
      b \left(\e\left(\frac{x}{2}\right)-\e\left(-\frac{x}{2}\right)\right)}
      =
      \frac{1}{b}
      \sum_{\ell = -(b-1)/2}^{(b-1)/2} \e( \ell x)
    \\
    &
      =
      \frac{1}{b}
      +
      \frac{2}{b} \sum_{\ell = 1}^{(b-1)/2} \cos(2\pi \ell x)
      ,
  \end{align*}
  which, by continuity, gives \eqref{eq:Kb-for-b-odd}.
  For $x\in [0,(2b-2)^{-1}]$,
  so that $0\leq 2 \ell x \leq \frac{b-1}{2b-2} \leq \frac12$,
  this implies that $K_b$ is a sum of decreasing and concave functions
  on $[0,(2b-2)^{-1}]$.   
  Hence, for $b$ odd,
  $K_b$ is decreasing and concave on $[0,(2b-2)^{-1}]$.  

  For $b$ even, let $b'=b/2$.
  For $x\in\R\setminus\Z$ we have
  \begin{align*}
    K_b(x) = \frac{\cos(\pi b' x)\sin(\pi b' x)}{b' \sin(\pi x)}
    &
      =
    \Re\left(
      \frac{
        \e\left( \frac{b' x}{2}\right)
        \left(
          \e\left( \frac{b' x}{2}\right)
          - \e\left(- \frac{b' x}{2}\right) \right)}
      {b' \left(
          \e\left( \frac{x}{2}\right) - \e\left(- \frac{x}{2}\right)
        \right)}
    \right)
    \\
    &
      =
    \Re\left(
      \frac{
        \e\left( \frac{x}{2}\right)
        \left( \e\left( b' x\right) - 1  \right)}
      {b' \left(
          \e\left( x\right) - 1 \right)
        }
    \right)
  \end{align*}
  so that
  \begin{displaymath}
    K_b(x) 
    =
    \Re\left(
      \frac{\e\left( \frac{x}{2}\right)}{b'}
      \sum_{0\leq \ell<b'} \e\left( \ell x\right)
    \right)
    =
    \frac{2}{b}
    \sum_{0\leq \ell<b/2}
    \cos\left(2\pi \left(\ell+\frac12\right) x \right)
    ,
  \end{displaymath}
  which, by continuity, gives \eqref{eq:Kb-for-b-even}.
  For $x\in [0,(2b-2)^{-1}]$,
  so that
  \begin{displaymath}
    0\leq 2 \left( \ell +\frac12 \right) x
    \leq 2 \left( \frac{b}{2} -1 +\frac12 \right) \frac{1}{2b-2}
    \leq \frac{b-1}{2b-2} \leq \frac12,
  \end{displaymath}
  this implies that $K_b$ is a sum of decreasing and concave functions
  on $[0,(2b-2)^{-1}]$.
  Hence, for $b$ even,
  $K_b$ is decreasing and concave on $[0,(2b-2)^{-1}]$.
\end{proof}

It might be useful to notice that for $x\in\R$ we have
\begin{equation}\label{eq:Kb-geometric-sum}
  \abs{K_b(x)} = \abs{\frac1b \sum_{u=0}^{b-1} \e(ux)}
\end{equation}
and that $K_b^2$ is Féjer's kernel:
\begin{equation}\label{eq:Fejer-kernel}
  K_b^2(x)
  =
  \frac1b
  \sum_{\abs{h} < b} \left(1-\frac{\abs{h}}{b}\right) \e(h x) 
  =
  \frac1b
  +
  \frac2b
  \sum_{h=1}^{b-1} \left(1-\frac{h}{b}\right) \cos(2\pi h x)
  .
\end{equation}

\section{Fourier transform}\label{section:fourier-transform}

We define for any integer $\lambda\geq 0$ and any
$(\alpha,\vartheta)\in\R^2$, the exponential sum
\begin{equation}\label{eq:definition-F}
  \mathcal{F}_\lambda(\alpha,\vartheta)
  =
  \frac{1}{b^\lambda}
  \sum_{0\leq n < b^\lambda}
  \e\left(\alpha R_\lambda(n)-\vartheta n\right)
  .
\end{equation}
We make the trivial observation that
\begin{equation}
  \label{symetry_F}
  \conjugate{\mathcal{F}_\lambda(\alpha,\vartheta)}
  =
  \mathcal{F}_\lambda(-\alpha,-\vartheta)
  =
  \mathcal{F}_\lambda(\vartheta,\alpha)
\end{equation}
and
\begin{equation}
  \label{trivial_bound_F}
  \abs{\mathcal{F}_\lambda(\alpha,\vartheta)} \leq 1.
\end{equation}

\begin{remark}
  For $b=2$ and $\lambda\geq 2$, the exponential sum
  $\mathcal{F}_\lambda(\alpha,\vartheta)$
  defined by~\eqref{eq:definition-F} is connected to
  $F_\lambda(\alpha,\vartheta)$ defined
  in~\cite{DMRSS_reversible_primes} via:
  \begin{equation*}
    F_\lambda(\alpha,\vartheta)
    =
    \e\left((\alpha-\vartheta)(2^{\lambda-1}+1)\right)
    \mathcal{F}_{\lambda-2}(2\alpha,2\vartheta).
  \end{equation*}
\end{remark}

\subsection{Product formula}

\begin{lemma}[Product formula]\label{lemma:FT-product}
  For any integer $\lambda\geq 0$
  and $(\alpha, \vartheta) \in \R^2$, we have
  \begin{equation}\label{eq:FT-product-formula}
    \abs{\mathcal{F}_\lambda(\alpha,\vartheta)}
    =
    \prod_{j=0}^{\lambda-1}
    \abs{K_b\left( \alpha b^{\lambda-1-j} - \vartheta  b^j\right)}
    ,
  \end{equation}
  with the convention that an empty product is equal to $1$.
\end{lemma}
\begin{proof}
  For $\lambda = 0$, we have
  $\abs{\mathcal{F}_\lambda(\alpha,\vartheta)}=1$ while the right hand
  side of \eqref{eq:FT-product-formula} is an empty product equal to
  $1$ by convention.
  For $\lambda\geq 1$, by writing
  \begin{equation*}
    n = \sum_{j=0}^{\lambda-1} \varepsilon_j \, b^j
    \quad
    \text{ and }
    \quad
    R_\lambda(n) = \sum_{j=0}^{\lambda-1} \varepsilon_j  \, b^{\lambda-1-j}
  \end{equation*}
  we obtain
  \begin{equation*}
    \mathcal{F}_\lambda(\alpha,\vartheta)
    =
    \frac{1}{b^\lambda}
    \prod_{j=0}^{\lambda-1}
    \sum_{\varepsilon_j=0}^{b-1}
    \e\left(
      \varepsilon_j
      \left(\alpha b^{\lambda-1-j} - \vartheta b^j\right)
    \right)
    , 
  \end{equation*}
  which, taking the modulus and using \eqref{eq:Kb-geometric-sum},
  leads to \eqref{eq:FT-product-formula}.
\end{proof}
\begin{lemma}[Splitting product formula]\label{FT-splitting-product}
  For integers $0\leq \lambda' \leq \lambda$
  and $(\alpha, \vartheta) \in \R^2$, we have 
  \begin{equation}\label{eq:FT-splitting-product}
    \abs{\mathcal{F}_\lambda(\alpha,\vartheta)}
    =
    \abs{\mathcal{F}_{\lambda'}(\alpha b^{\lambda-\lambda'},\vartheta)}
    \cdot
    \abs{\mathcal{F}_{\lambda-\lambda'}(\alpha,\vartheta b^{\lambda'})}
    .
  \end{equation}
\end{lemma}
\begin{proof}
  By \eqref{eq:FT-product-formula} we have
  \begin{align*}
    &
      \abs{\mathcal{F}_\lambda(\alpha,\vartheta)}
    \\
    &
      =
      \prod_{j=0}^{\lambda'-1}
      \abs{K_b\left( \alpha b^{\lambda-1-j} - \vartheta  b^j\right)}
      \prod_{j=\lambda'}^{\lambda-1}
      \abs{K_b\left( \alpha b^{\lambda-1-j} - \vartheta  b^j\right)}
    \\
    &
      =
      \prod_{j=0}^{\lambda'-1}
      \abs{K_b\left(
      \alpha b^{\lambda-\lambda'} b^{\lambda'-1-j} - \vartheta b^j
      \right)}
      \prod_{j'=0}^{\lambda-\lambda'-1}
      \abs{K_b\left(
      \alpha b^{\lambda-1-\lambda'-j'} - \vartheta b^{\lambda'} b^{j'}
      \right)}
      ,
  \end{align*}
  which, again by \eqref{eq:FT-product-formula},
  gives \eqref{eq:FT-splitting-product}.
\end{proof}~

\begin{lemma}\label{lemma:F2-difference}
  For any $(\alpha,\vartheta,\vartheta')\in\R^3$
  and any integer $\lambda\geq 0$,
  we have
  \begin{equation}
    \label{eq:F2-difference}
    \abs{
      \abs{\mathcal{F}_\lambda(\alpha,\vartheta)}^2
      -
      \abs{\mathcal{F}_\lambda(\alpha,\vartheta')}^2
    }
    \leq
    \frac{2 \pi }{3} b^\lambda
    \norm{\vartheta-\vartheta'}
    .
  \end{equation}
\end{lemma}
\begin{proof}
  By periodicity we may assume that
  \begin{math}
    \vartheta-\frac12 \leq \vartheta' < \vartheta+\frac12
    ,
  \end{math}
  so that
  \begin{math}
    \norm{\vartheta-\vartheta'}
    =
    \abs{\vartheta-\vartheta'}
    .
  \end{math}
  
  By \eqref{eq:definition-F}
  we have
  \begin{displaymath}
    \abs{\mathcal{F}_\lambda(\alpha,\vartheta)}^2
    =
    \frac{1}{b^{2\lambda}}
    \sum_{0\leq n_1 < b^\lambda}
    \sum_{0\leq n_2 < b^\lambda}
    \e\left(
      \alpha \left(R_\lambda(n_1)-R_\lambda(n_2)\right)
      -\vartheta (n_1-n_2)
    \right)
    ,
  \end{displaymath}
  hence
  \begin{multline*}
    \abs{\mathcal{F}_\lambda(\alpha,\vartheta)}^2
    -
    \abs{\mathcal{F}_\lambda(\alpha,\vartheta')}^2
    \\
    =
    \frac{1}{b^{2\lambda}}
    \sum_{0\leq n_1 < b^\lambda}
    \sum_{0\leq n_2 < b^\lambda}
    \e\left(
      \alpha \left(R_\lambda(n_1)-R_\lambda(n_2)\right)
    \right)
    \\
    \left(
      \e\left( - \vartheta (n_1-n_2) \right)
      -
      \e\left( - \vartheta' (n_1-n_2) \right)
    \right)
    ,
  \end{multline*}
  thus
  \begin{align*}
    \abs{
      \abs{\mathcal{F}_\lambda(\alpha,\vartheta)}^2
      -
      \abs{\mathcal{F}_\lambda(\alpha,\vartheta')}^2
    }
    &\leq
    \frac{1}{b^{2\lambda}}
    \sum_{0\leq n_1 < b^\lambda}
    \sum_{0\leq n_2 < b^\lambda}
    \abs{
      -2i\pi
      \int_{\vartheta' (n_1-n_2)}^{\vartheta (n_1-n_2)} \e(-t) \, dt
    }
    \\
    &\leq
      \frac{2 \pi \abs{\vartheta-\vartheta'}}{b^{2\lambda}}
      \sum_{0\leq n_1 < b^\lambda}
      \sum_{0\leq n_2 < b^\lambda}
      \abs{n_1-n_2}
    \\
    & \quad =
      \frac{4 \pi \abs{\vartheta-\vartheta'}}{b^{2\lambda}}
      \sum_{0\leq n_1 < b^\lambda}
      \frac{n_1(n_1+1)}{2}
    \\
    & \quad =
      \frac{4 \pi \abs{\vartheta-\vartheta'}}{b^{2\lambda}}
      \frac{(b^\lambda-1)b^\lambda(b^\lambda+1)}{6}
    \\
    & \quad 
      \leq
      \frac{2 \pi }{3} b^\lambda
      \abs{\vartheta-\vartheta'} 
      ,
  \end{align*}
  which, remembering that
  \begin{math}
    \norm{\vartheta-\vartheta'}
    =
    \abs{\vartheta-\vartheta'}
    ,
  \end{math}
  gives \eqref{eq:F2-difference}.
\end{proof}

\bigskip

\subsection{Uniform bounds}~

\begin{lemma}\label{lemma:majoration-U}
  For any integer $b\geq 2$ and any real number $t$
  such that $\norm{t} \leq \frac 1b$, we have
  \begin{equation}\label{eq:majoration-U}
    \abs{K_b(t)}
    \leq
    \exp\left(-\,\frac{\pi^2}{6}(b^2-1)\norm{t}^2\right)
    .
  \end{equation}
\end{lemma}
\begin{proof}
  This improvement of \cite[Lemme 3]{mauduit-rivat-2009} can be found
  in \cite[Lemme 11]{barat-martin-mauduit-rivat} and was noticed
  independently by Tuxanidy and Panario
  in \cite[Lemma 5.6]{tuxanidy-panario-2024}.
\end{proof}

\begin{lemma}\label{lemma:max-cross-norms}
  For any integer $b\geq 2$ and $(\alpha,\vartheta)\in\R^2$ we have
  \begin{equation}
    \label{eq:max-cross-norms}
    \max\left(
      \norm{\alpha b - \vartheta},
      \norm{\alpha - \vartheta b}
    \right)
    \geq
    \frac{
      \max\left(
        \norm{\alpha \left( b^2-1 \right)},
        \norm{\vartheta \left( b^2-1 \right)}
      \right)
    }
    {b+1}
    .
  \end{equation}
\end{lemma}
\begin{proof}
  We proceed as in \cite[Lemme 6]{mauduit-rivat-2009}.
  Let
  \begin{displaymath}
    \delta = \frac{\norm{\vartheta \left(b^2-1 \right)}}{b+1}
    ,\quad
    u = \alpha b - \vartheta
    ,\quad
    v = \alpha - \vartheta b
    .
  \end{displaymath}
  Since
  \begin{math}
    u = v b + \vartheta \left(b^2-1 \right),
  \end{math}
  if $\norm{u} < \delta$, then
  \begin{align*}
    b \norm{v} \geq \norm{v b}
    =
    \norm{u - \vartheta\left(b^2-1 \right)}
    \geq
    \norm{\vartheta\left(b^2-1 \right)}
    -
    \norm{u}
    \geq
    \delta (b+1) - \delta = \delta b
    .
  \end{align*}
  Whether $\norm{u} < \delta$
  or $\norm{u} \geq \delta$,
  it follows that
  \begin{displaymath}
    \max\left(
      \norm{\alpha b-\vartheta},
      \norm{\alpha-\vartheta b}
    \right)
    =
    \max\left(
      \norm{u},
      \norm{v}
    \right)
    \geq \delta = \frac{\norm{\vartheta \left( b^2-1 \right)}}{b+1}.
  \end{displaymath}
  Exchanging $\alpha$ and $\vartheta$,
  we get \eqref{eq:max-cross-norms}.
\end{proof}

\begin{lemma}\label{lemme:produit-U}
  For any integer $b\geq 2$ and any $x \in \R$, we have
  \begin{equation}\label{eq:produit-U}
    \max_{t\in \R}
    \abs{ K_b\left(x-t\right) K_b\left(x-bt\right) }
    \leq
    K_b\left(\frac{\norm{(b-1)x}}{b+1}\right)
    .
  \end{equation}
\end{lemma}
\begin{proof}
  This is \cite[Lemme 6]{mauduit-rivat-2009}.
\end{proof}

For $\lambda\geq 0$ and $\alpha \in \R$, we define
\begin{equation}\label{eq:definition-G}
  G_\lambda(\alpha)
  =
  \prod_{j=0}^{\lambda-1}
  K_b \left( \frac{\norm{\alpha b^j}}{b+1}\right)
  > 0
  .
\end{equation}

\begin{lemma}\label{lemma:norm-lowerbound}
  For any integer $b\geq 2$ and any $\alpha \in \R\setminus\Z$, the
  integer
  \begin{displaymath}
    j_0(\alpha) = 
    \floor{\frac{\log\frac{b}{(b+1)\norm{\alpha}}}{\log b}}
  \end{displaymath}
  satisfies
  \begin{displaymath}
    \norm{\alpha b^{j_0(\alpha)}} \geq \frac{1}{b+1}.
  \end{displaymath}
\end{lemma} 
\begin{proof}
  Since $\norm{\alpha}\le 1/2$, we have $j_0 = j_0(\alpha) \geq 0$.
  Therefore $b^{j_0} \in \Z$ and we have
  \begin{math}
    \norm{\alpha b^{j_0}} = \norm{\norm{\alpha} b^{j_0}}
  \end{math}
  by parity and periodicity.  By definition of $j_0$, we have
  \begin{displaymath}
    \frac{1}{b+1} < \norm{\alpha} b^{j_0} 
    \leq
    \frac{b}{b+1} = 1-\frac{1}{b+1}, 
  \end{displaymath}
  which gives the result.
\end{proof}

\begin{lemma}\label{lemma:G-lambda-upperbound}
  For any integers $b\geq 2$, $\lambda \geq 1$
  and any $\alpha \in \R$,
  if
  \begin{displaymath}
    \alpha_0 = \min_{0\leq j \leq \lambda-1} \norm{\alpha b^{j}} > 0 
    ,\qquad
    J = 
    1 + \floor{\frac{\log\frac{b}{(b+1)\alpha_0}}{\log b}}
    ,
  \end{displaymath}
  then
  \begin{equation}\label{eq:G-lambda-upperbound}
    G_\lambda(\alpha)
    \leq
    K_b\left( (b+1)^{-2} \right)^{\floor{\frac{\lambda}{J}}}
    .
  \end{equation}
  In particular, for any integers
  $b\geq 2$, $\lambda\geq 1$, $d\geq 2$, $h\in\Z$
  such that
  \begin{math}
    b^{\lambda-1} h \not\equiv 0 \bmod d,
  \end{math}
  we have
  \begin{equation}
    \label{eq:uniform-upperbound-of-G-for-denominator-d}
    \abs{G_\lambda\left(\frac{h}{d}\right)}
    \leq
    K_b\left( (b+1)^{-2} \right)^{\frac{\lambda \log 2}{2\log d} - 1}
    \leq
    \frac{\pi}{2} \,
    b^{- \frac{2 \Upsilon_b}{\log d}\lambda}
    ,
  \end{equation}
  with
  \begin{equation}\label{eq:definition-Upsilon_b}
    \Upsilon_b
    =
    \frac{- (\log 2) \log K_b((b+1)^{-2})}{4 \log b}
    > 0
    .
  \end{equation}
\end{lemma}
\begin{proof}
  For
  \begin{math}
    0\leq k < K = \floor{\frac{\lambda}{J}}
    ,
  \end{math}
  using Lemma \ref{lemma:norm-lowerbound}
  we have $0\leq j_0(\alpha b^{kJ}) < J $
  and
  \begin{displaymath}
    \norm{\alpha b^{kJ+j_0(\alpha b^{k J})}}
    \geq \frac{1}{b+1}
    .
  \end{displaymath}
  Writing
  \begin{displaymath}
    G_\lambda(\alpha)
    \leq
    \prod_{0\leq k < K}
    \prod_{0\leq j < J}
    K_b \left( \frac{\norm{\alpha b^{k J+j}}}{b+1}\right)
    \leq
    \prod_{0\leq k < K}
    K_b\left( \frac{\norm{\alpha b^{kJ+j_0(\alpha b^{k J})}}}{b+1} \right)
    ,  
  \end{displaymath}
   since $K_b$ is decreasing on $[0,\frac{1}{2(b+1)}]$, it follows that
  \begin{displaymath}
    G_\lambda(\alpha)
    \leq
    \prod_{0\leq k < K}
    K_b\left( (b+1)^{-2} \right)
    =
    K_b\left( (b+1)^{-2} \right)^K
    ,
  \end{displaymath}
  which gives \eqref{eq:G-lambda-upperbound}.

  Since
  \begin{math}
    b^{\lambda-1} h \not\equiv 0 \bmod d,
  \end{math}
  for $0\leq j \leq \lambda-1$
  we have
  \begin{math}
    b^j h \not\equiv 0 \bmod d,
  \end{math}
  hence
  \begin{displaymath}
    \alpha_0
    =
    \min_{0\leq j \leq \lambda-1}
    \norm{\frac{h}{d}b^j}
    \geq
    \frac1d
    .
  \end{displaymath}
  By \eqref{eq:G-lambda-upperbound} it follows that
  \begin{displaymath}
    \abs{G_\lambda\left(\frac{h}{d}\right)}
    \leq
    K_b\left( (b+1)^{-2} \right)^{\floor{\frac{\lambda}{J}}}    
    ,
  \end{displaymath}
  where, remembering that $b\geq 2$ and $d\geq 2$,
  \begin{align*}
    J
    =
    1 + \floor{\frac{\log\frac{b}{(b+1)\alpha_0}}{\log b}}
    &
      \leq
    1 + \frac{\log\frac{b d}{b+1}}{\log b}
    \\
    &
      \leq
    1 + \frac{\log d}{\log b}
    =
    \left(\frac{1}{\log d}+\frac{1}{\log b}\right)
    \log d
    \leq
    \frac{2\log d}{\log 2}
    ,
  \end{align*}
  so that
  \begin{displaymath}
    \floor{\frac{\lambda}{J}}
    >
    \frac{\lambda}{J} - 1
    \geq
    \frac{\lambda \log 2}{2\log d} - 1
    ,
  \end{displaymath}
  which proves the left hand side inequality of
  \eqref{eq:uniform-upperbound-of-G-for-denominator-d}.
  Observing that
  \begin{equation}\label{eq:minoration-Kb-(b+1)^-2}
    K_b\left((b+1)^{-2}\right)^{-1}
    =
    \frac{b\sin \left(\pi(b+1)^{-2}\right)}{\sin \left(\pi b(b+1)^{-2}\right)}
    \leq
    \frac{\pi b(b+1)^{-2}}{2  b(b+1)^{-2}}
    =\frac{\pi}{2}
    ,
  \end{equation}
  we get the right hand side inequality of
  \eqref{eq:uniform-upperbound-of-G-for-denominator-d}.
\end{proof}

\begin{lemma}\label{lemma:pointwise-upperbound-of-F}
  For any integer $\lambda\geq 1$ and any $(\alpha,\vartheta)\in\R^2$,
  we have
  \begin{equation}\label{eq:uniform-upperbound-of-F}
    \abs{\mathcal{F}_\lambda(\alpha,\vartheta)}
    \leq
    \min\left(
      G_{\lambda-1}^{1/2}(\alpha (b^2-1))
      ,
      G_{\lambda-1}^{1/2}(\vartheta (b^2-1))
    \right)
    .
  \end{equation}
  For $(u,v)\in [0,1]^2$ with $u+v=1$, we have
  \begin{equation}\label{eq:combined-upperbound-of-F}
    \abs{\mathcal{F}_\lambda(\alpha,\vartheta)}
    \leq
    G_{\lambda-1}^{u/2}(\alpha (b^2-1))
    \
    G_{\lambda-1}^{v/2}(\vartheta (b^2-1))
    .
  \end{equation}
\end{lemma}
\begin{proof}
  By Lemma~\ref{lemma:FT-product},
  \begin{equation*}
    \abs{\mathcal{F}_\lambda(\alpha,\vartheta)}
    =
    \prod_{j=0}^{\lambda-1}
    \abs{K_b \left( \alpha b^{\lambda-1-j} - \vartheta b^j\right)}
    ,
  \end{equation*}
  hence
  \begin{equation*}
    \abs{\mathcal{F}_\lambda(\alpha,\vartheta)}^2
    \leq
    \prod_{j=0}^{\lambda-2}
    \abs{
      K_b \left( \alpha b^{\lambda-1-j}  - \vartheta  b^j\right)
      K_b \left( \alpha b^{\lambda-2-j}  - \vartheta  b^{j+1}\right)
      }
  \end{equation*}
  and by Lemma~\ref{lemme:produit-U}
  with $x = \alpha b^{\lambda-2-j}(b+1)$ and
  $t=\alpha b^{\lambda-2-j} + \vartheta b^j$,
  \begin{equation*}
    \abs{\mathcal{F}_\lambda(\alpha,\vartheta)}^2
    \leq
    \prod_{j=0}^{\lambda-2}
    K_b \left( \frac{\norm{\alpha b^{\lambda-2-j} (b^2-1)}}{b+1}\right)
    = G_{\lambda-1}(\alpha(b^2-1)).
  \end{equation*}
  By \eqref{symetry_F} we may exchange $\alpha$ and $\vartheta$,
  and we obtain \eqref{eq:uniform-upperbound-of-F}.
  
  For $(u,v)\in [0,1]^2$ with $u+v=1$,
  by \eqref{eq:uniform-upperbound-of-F} we have
  \begin{align*}
    \abs{\mathcal{F}_\lambda(\alpha,\vartheta)}^u
    &\leq
    G_{\lambda-1}^{u/2}(\alpha (b^2-1))
    \\
    \abs{\mathcal{F}_\lambda(\alpha,\vartheta)}^v
    &\leq
    G_{\lambda-1}^{v/2}(\vartheta (b^2-1))
    ,
  \end{align*}
   and the product gives \eqref{eq:combined-upperbound-of-F}.
\end{proof}

\begin{lemma}\label{lemma:uniform-upperbound-of-F-for-denominator-d}
  For any integers $\lambda\geq 2$, $d\geq 2$ and $h\in\Z$
  such that
  \begin{displaymath}
    b^{\lambda-2} \left(b^2-1\right) h \not\equiv 0 \bmod d,
  \end{displaymath}
  and any $\vartheta\in\R$,
  we have
  \begin{equation}
    \label{eq:uniform-upperbound-of-F-for-denominator-d}
    \abs{\mathcal{F}_\lambda\left(\frac{h}{d},\vartheta\right)}
    \leq
    K_b\left( (b+1)^{-2} \right)^{\frac{\lambda \log 2}{4\log d} - \frac34}    
    \leq
    \left( \frac{\pi}{2} \right)^{3/4}\,
    b^{-\Upsilon_b\frac{\lambda}{\log d}}
    ,
  \end{equation}
  where $\Upsilon_b$ is
  defined by \eqref{eq:definition-Upsilon_b}.
\end{lemma}
\begin{proof}
  By Lemma \ref{lemma:pointwise-upperbound-of-F},
  for any $\vartheta\in\R$,
  we have
  \begin{displaymath}
    \abs{\mathcal{F}_\lambda\left(\frac{h}{d},\vartheta\right)}
    \leq
    G_{\lambda-1}^{1/2}\left(\frac{h}{d}(b^2-1)\right)
    .
  \end{displaymath}
  Since
  \begin{math}
    b^{\lambda-2} \left(b^2-1\right) h \not\equiv 0 \bmod d,
  \end{math}
  by \eqref{eq:uniform-upperbound-of-G-for-denominator-d},
  remembering that $d\geq 2$,
  it follows that
  \begin{displaymath}
    \abs{\mathcal{F}_\lambda\left(\frac{h}{d},\vartheta\right)}
    \leq
    K_b\left( (b+1)^{-2} \right)^{\frac{(\lambda-1) \log 2}{4\log d}-\frac12}
    \leq
    K_b\left( (b+1)^{-2} \right)^{\frac{\lambda \log 2}{4\log d}-\frac34}
    ,
  \end{displaymath}
  which proves the left hand side inequality
  of~\eqref{eq:uniform-upperbound-of-F-for-denominator-d}.
  Remembering \eqref{eq:minoration-Kb-(b+1)^-2},
  we get the right hand side inequality of
  \eqref{eq:uniform-upperbound-of-F-for-denominator-d}.
\end{proof}

\subsection{Incomplete sums}

\begin{lemma}\label{lemma:reverse-split}
  For integers $0\leq \lambda'\leq \lambda$,
  $0\leq n_1 < b^{\lambda'}$ and $n_2 \geq 0$
  we have
  \begin{equation} \label{eq:reverse-split}
    R_\lambda(n_1 + b^{\lambda'} n_2)
    =
    b^{\lambda-\lambda'} R_{\lambda'}(n_1)
    +
    R_{\lambda-\lambda'}(n_2)
    .
  \end{equation}
\end{lemma}
\begin{proof}
  It is enough to observe that
  \begin{displaymath}
    R_\lambda(n_1 + b^{\lambda'} n_2)
    =
    \sum_{j=0}^{\lambda'-1} \varepsilon_j(n_1) b^{\lambda-1-j}
    +
    \sum_{j=0}^{\lambda-\lambda'-1} \varepsilon_{j}(n_2) b^{\lambda-\lambda'-1-j}
    ,
  \end{displaymath}
  and we get \eqref{eq:reverse-split}.
\end{proof}

The following lemma permits to handle an incomplete sum:
\begin{lemma}\label{lemma:b-adic-splitting-of-type-I-sums}
  For any integers $b\geq 2$, $\lambda\geq 0$,
  and any $(\alpha,\vartheta)\in\R^2$, 
  $x\in\R$ with $1 \leq x \leq b^\lambda$,
  \begin{equation}\label{eq:b-adic-splitting-of-type-I-sums}
    \abs{
      \sum_{0\leq \ell < x}
      \e\left(\alpha R_\lambda(\ell) - \vartheta \ell\right)
    }
    \leq
    (b-1)
    \sum_{0\leq \lambda' \leq \frac{\log x}{\log b}}
    b^{\lambda'}
    \abs{
      \mathcal{F}_{\lambda'}\left(\alpha b^{\lambda-\lambda'},\vartheta\right)
    }
    .
  \end{equation}
\end{lemma}
\begin{proof}
  The upper bound~\eqref{eq:b-adic-splitting-of-type-I-sums} is
  clearly true for $x = b^{\lambda}$. So we can assume that
  $x<b^{\lambda}$.  Writing
  \begin{math}
    \lambda' = \floor{\frac{\log x}{\log b}},
  \end{math}
  we have $b^{\lambda'} \leq x < b^\lambda$, hence $\lambda' < \lambda$.
  We write $x = y b^{\lambda'}+x'$ with
  $y \in \{0,\ldots,b-1\}$ and $0\leq x'<b^{\lambda'}$.
  The left hand-side of \eqref{eq:b-adic-splitting-of-type-I-sums} is
  at most
  \begin{multline*}
    \abs{
      \sum_{0\leq k < y}
      \sum_{0\leq \ell < b^{\lambda'}}
      \e\left(
        \alpha R_\lambda(k b^{\lambda'}+ \ell)
        - \vartheta (k b^{\lambda'}+\ell)\right)
    }
    \\
    +
    \abs{
      \sum_{0\leq \ell < x'}
      \e\left(
        \alpha R_\lambda(y b^{\lambda'} +\ell)
        - \vartheta (yb^{\lambda'}+\ell)\right)
    }
    .
  \end{multline*}
  By Lemma~\ref{lemma:reverse-split},
  for $k\in \{0,\ldots,b-1\}$ and $\ell \in \{0,\ldots,b^{\lambda'}-1\}$ 
  we have
  \begin{displaymath}
    R_\lambda(k b^{\lambda'} + \ell)
    = R_{\lambda-\lambda'}(k) + R_{\lambda'}(\ell) b^{\lambda-\lambda'}
    ,
  \end{displaymath}
  which permits to get for the left hand-side of
  \eqref{eq:b-adic-splitting-of-type-I-sums}
  the following upper bound
  \begin{displaymath}
    y
    \abs{
      \sum_{0\leq \ell < b^{\lambda'}}
      \e\left(\alpha R_{\lambda'}(\ell) b^{\lambda-\lambda'} - \vartheta \ell\right)
    }
    +
    \abs{
      \sum_{0\leq \ell < x'}
      \e\left(\alpha R_{\lambda'}(\ell) b^{\lambda-\lambda'} - \vartheta \ell\right)
    }
    .
  \end{displaymath}
  Remembering that $y\leq b-1$ and applying this process again to
  $x'$ and so on, we obtain
  \begin{displaymath}
    \abs{
      \sum_{0\leq \ell < x}
      \e\left(\alpha R_\lambda(\ell) - \vartheta \ell\right)
    }
    \leq
    (b-1)
    \sum_{0\leq \lambda' \leq \frac{\log x}{\log b}}
    \abs{
      \sum_{0\leq \ell < b^{\lambda'}}
      \e\left(
        \alpha b^{\lambda-\lambda'} R_{\lambda'}(\ell) - \vartheta \ell\right)
    }    
    ,    
  \end{displaymath}
  which is \eqref{eq:b-adic-splitting-of-type-I-sums}.
\end{proof}

\bigskip

\subsection{Mean values}~

\begin{lemma}\label{lemma:L2-mean-value}
  Let $b\geq 2$ and $\lambda \geq 0$ be integers.
  For $(\alpha,t)\in\R^2$ we have
  \begin{equation}
    \label{eq:L2-mean-value}
    \sum_{0\leq h < b^\lambda}
    \abs{
      \mathcal{F}_{\lambda}\left(\alpha,t+\frac{h}{b^{\lambda}}\right)
    }^2
    = 1
    .
  \end{equation}
\end{lemma}
\begin{proof}
  By \eqref{eq:definition-F},
  the left hand side of \eqref{eq:L2-mean-value} is equal to
  \begin{displaymath}
    \frac{1}{b^{2\lambda}}
    \sum_{n_1=0}^{b^\lambda-1}
    \sum_{n_2=0}^{b^\lambda-1}
    \sum_{h=0}^{b^\lambda-1} 
    \e\left(
      \alpha \left(R_\lambda(n_1)-R_\lambda(n_2)\right)
      - t (n_1-n_2)
      - \frac{h}{b^{\lambda}} (n_1-n_2)
    \right),
  \end{displaymath}
  and by orthogonality we have $n_1=n_2$
  and we obtain~\eqref{eq:L2-mean-value}.
\end{proof}

\begin{lemma}\label{lemma:L2-mean-value-in-alpha}
  For any integers $b\geq 2$ and $\lambda \geq 1$,
  for any sequence $\vartheta_1,\ldots,\vartheta_N$
  of real numbers
  which is $b^{-\lambda}$ well spaced modulo~$1$,
  for any $\alpha\in\R$,
  we have
  \begin{equation}\label{eq:L2-mean-value-in-alpha}
    \sum_{n=1}^N
    \abs{\mathcal{F}_{\lambda}\left(\alpha,\vartheta_n\right)}^2
    \leq
    2
    .
  \end{equation}
\end{lemma}
\begin{proof}
  By the large sieve inequality (Lemma \ref{lemma:large-sieve}), we have
  \begin{displaymath}
    \sum_{n=1}^N
    \abs{\mathcal{F}_{\lambda}\left(\alpha,\vartheta_n\right)}^2
    \leq
    \left(b^\lambda - 1 + b^\lambda\right) b^\lambda 
    b^{-2\lambda},
  \end{displaymath}
  and \eqref{eq:L2-mean-value-in-alpha} follows.
\end{proof}

The following two lemmas will permit us to introduce the exponent
$\eta_b$ which will be used to bound $L^1$ norms.

\begin{lemma}\label{lemma:psi-b}
  For any integer $b\geq 2$ and $x\in\R$
  we have
  \begin{equation}
    \label{eq:majoration-psi-b}
    \psi_b(x)
    :=
    \sum_{r=1}^b \abs{K_b\left( x + \frac{r}{b} \right)}
    \leq
    \psi_b\left(\frac1{2b}\right)
    \leq
    \frac{2}{\pi} \log \left( c_\psi b \right)
    ,
  \end{equation}
  with
  \begin{displaymath}
    4.86
    <
    c_\psi
    =
    \frac{2}{\pi} 
    \exp\left(   \frac{2 \frac{\pi}{10} }{\sin \frac{\pi}{10}} \right)
    <
    4.87,
  \end{displaymath}
  and we have the exact values
  \begin{equation}\label{eq:psi_b-exact-values}
    \psi_2\left(\frac14\right) = \sqrt2,\
    \psi_3\left(\frac16\right) = \frac53,\
    \psi_4\left(\frac18\right) = \left(2+\sqrt2\right)^{1/2}
    .
  \end{equation}
\end{lemma}
\begin{proof}
  By \cite[Lemme 14]{mauduit-rivat-2010},
  for any integer $b\geq 2$ and $x\in\R$
  we have
  \begin{displaymath}
    \psi_b(x)
    \leq
    \psi_b\left(\frac1{2b}\right)
    \leq
    \frac{2}{b \sin \frac{\pi}{2 b}}
    +
    \frac{2}{\pi} \log \frac{2 b}{\pi}
    =
    \frac{2}{\pi}
    \left(
      \frac{2 \frac{\pi}{2b} }{\sin \frac{\pi}{2 b}}
      +
      \log \frac{2 b}{\pi}
      \right)
    .
  \end{displaymath}
  Since $t \mapsto \frac{\sin t}{t}$ is decreasing on $[0,\pi/10]$,
  it follows that \eqref{eq:majoration-psi-b} holds for $b\geq 5$.

  The exact values \eqref{eq:psi_b-exact-values} are given by
  \cite[Lemme 14]{mauduit-rivat-2010}.
  They permit us to check by elementary numerical computations
  that \eqref{eq:majoration-psi-b}
  also holds for $2\leq b\leq 4$,
  and this completes the proof of \eqref{eq:majoration-psi-b}.
\end{proof}

\begin{lemma}\label{lemma:eta_b-properties}
  For any integer $b\geq 2$,
  let
  \begin{equation}\label{eq:definition-eta_b}
    \eta_2 = \frac{\log\left(2+\sqrt2\right)}{4\log 2}
    \quad
    \text{ and }
    \quad
    \forall b\geq 3,\
    \eta_b = \frac{\log \psi_b\left(\frac{1}{2 b} \right)}{\log b}
    ,
  \end{equation}
  so that
  \begin{displaymath}
    2^{\eta_2}= \left(2+\sqrt2\right)^{1/4}
    \quad
    \text{ and }
    \quad
    \forall b\geq 3,\
    \psi_b\left(\frac{1}{2 b} \right)
    =
    b^{\eta_b}
    .
  \end{displaymath}
  For any integer $b\geq 5$ we have
  \begin{equation}\label{eq:upperbound-eta-b}
    0 < \eta_b
    <
    \eta_2 = \eta_4 = \frac{\log \left(2+\sqrt2\right)}{4\log 2}
    < \eta_3 = \frac{\log 5}{\log 3} - 1 < 0.465
    .
  \end{equation}
\end{lemma}
\begin{proof}
  The definition of $\eta_b$ is taken from
  \cite[Lemme 14 and Lemme 18]{mauduit-rivat-2010}.
  The upper bounds \eqref{eq:upperbound-eta-b}
  follow from \eqref{eq:majoration-psi-b},
  \eqref{eq:psi_b-exact-values}
  and the fact that
  \begin{math}
    b\mapsto
    \log\left( \frac{2}{\pi} \log \left( c_\psi b \right) \right)
    / \log b
  \end{math}
  is decreasing for $b\geq 5$,
  or may be deduced from the proof
  of (84) in \cite[Lemme~14]{mauduit-rivat-2010}.
\end{proof}

\begin{lemma}\label{lemma:L1-norm-in-alpha-of-F}
  For any integers $b\geq 2$ and $\lambda\geq 0$,
  and $\eta_b$ defined by \eqref{eq:definition-eta_b},
  we have
  \begin{equation}
    \label{eq:L1-norm-in-alpha-of-F}
    \int_0^1
    \abs{\mathcal{F}_{\lambda}\left(\alpha,\vartheta\right)}
    \, d\vartheta
    \leq
    b^{(\eta_b-1)\lambda}
    .
  \end{equation}
\end{lemma}
\begin{proof}
  For $\lambda=0$ both sides of \eqref{eq:L1-norm-in-alpha-of-F} are
  equal to $1$.
  Let us assume first that $b\geq 3$.
  For $\lambda\geq 1$,
  by \eqref{eq:FT-splitting-product} and the fact that
  \begin{math}
    \abs{\mathcal{F}_{\lambda-1}(\alpha,\cdot)}
  \end{math}
  is $1$-periodic, we have
  \begin{align*}
    \int_0^1
    \abs{\mathcal{F}_{\lambda}\left(\alpha,\vartheta\right)}
    \, d\vartheta
    &=
    \int_0^1
    \abs{\mathcal{F}_{\lambda-1}(\alpha,\vartheta b)}
    \cdot
    \abs{\mathcal{F}_{1}(\alpha b^{\lambda-1},\vartheta)}
    \, d\vartheta
    \\
    &
      =
      \sum_{r=0}^{b-1}
      \int_0^{\frac1b}
      \abs{\mathcal{F}_{\lambda-1}\left(
      \alpha,\left(\frac{r}{b} + \vartheta\right) b \right)
      }
      \cdot
      \abs{\mathcal{F}_{1}\left(
      \alpha b^{\lambda-1},\frac{r}{b}+ \vartheta\right)
      }
      \, d\vartheta
    \\
    &
      =
      \int_0^{1}
      \abs{\mathcal{F}_{\lambda-1}\left(
      \alpha, \vartheta' \right)
      }
      \frac1b
      \sum_{r=0}^{b-1}
      \abs{\mathcal{F}_{1}\left(
      \alpha b^{\lambda-1},\frac{r+\vartheta'}{b} \right)
      }
      \, d\vartheta'
      .
  \end{align*}
  By \eqref{eq:FT-product-formula}, \eqref{eq:majoration-psi-b}
  and \eqref{eq:definition-eta_b},
  \begin{align*}
    \frac1b
    \sum_{r=0}^{b-1}
    \abs{\mathcal{F}_{1}\left(
        \alpha b^{\lambda-1},\frac{r+\vartheta'}{b} \right)
    }
    &=
    \frac1b
    \sum_{r=0}^{b-1}
    \abs{K_b\left( \alpha b^{\lambda-1} - \frac{r+\vartheta'}{b} \right)}
    \\
    &=
    \frac1b
    \psi_b\left( \alpha b^{\lambda-1} - \frac{\vartheta'}{b} \right)
    \leq
    \frac{b^{\eta_b}}{b}
    ,
  \end{align*}
  hence
  \begin{displaymath}
    \int_0^1
    \abs{\mathcal{F}_{\lambda}\left(\alpha,\vartheta\right)}
    \, d\vartheta
    \leq
    b^{\eta_b-1}
    \int_0^1
    \abs{\mathcal{F}_{\lambda-1}\left(\alpha,\vartheta\right)}
    \, d\vartheta
    ,
  \end{displaymath}
  and by induction
  \begin{displaymath}
    \int_0^1
    \abs{\mathcal{F}_{\lambda}\left(\alpha,\vartheta\right)}
    \, d\vartheta
    \leq
    b^{(\eta_b-1)\lambda}
    \int_0^1
    \abs{\mathcal{F}_{0}\left(\alpha,\vartheta\right)}
    \, d\vartheta
    =
    b^{(\eta_b-1)\lambda}
    ,
  \end{displaymath}
  which gives \eqref{eq:L1-norm-in-alpha-of-F} for $b\geq 3$.

  Let us assume now that $b=2$.
  For $\lambda\geq 2$,
  by \eqref{eq:FT-splitting-product} and the fact that
  \begin{math}
    \abs{\mathcal{F}_{\lambda-2}(\alpha,\cdot)}
  \end{math}
  is $1$-periodic, we have
  \begin{align*}
    \int_0^1
    &
      \abs{\mathcal{F}_{\lambda}\left(\alpha,\vartheta\right)}
    \, d\vartheta
    \\
    &=
    \int_0^1
    \abs{\mathcal{F}_{\lambda-2}(\alpha,\vartheta 2^2)}
    \cdot
    \abs{\mathcal{F}_{2}(\alpha 2^{\lambda-2},\vartheta)}
    \, d\vartheta
    \\
    &
      =
      \sum_{r=0}^{2^2-1}
      \int_0^{1/2^2}
      \abs{\mathcal{F}_{\lambda-2}\left(
      \alpha,\left(\frac{r}{2^2} + \vartheta\right) 2^2 \right)
      }
      \cdot
      \abs{\mathcal{F}_{2}\left(
      \alpha 2^{\lambda-2},\frac{r}{2^2}+ \vartheta\right)
      }
      \, d\vartheta
      ,
  \end{align*}
  hence
  \begin{multline}\label{eq:integral-induction-formula-base-2}
    \int_0^1
    \abs{\mathcal{F}_{\lambda}\left(\alpha,\vartheta\right)}
    \, d\vartheta
    \\
    =
    \int_0^{1}
    \abs{\mathcal{F}_{\lambda-2}\left(
        \alpha, \vartheta' \right)
    }
    \frac1{2^2}
    \sum_{r=0}^{2^2-1}
    \abs{\mathcal{F}_{2}\left(
        \alpha 2^{\lambda-2},\frac{r+\vartheta'}{2^2} \right)
    }
    \, d\vartheta'
    .
  \end{multline}
  We have
  \begin{align*}
    &
    \sum_{r=0}^{2^2-1}
      \abs{\mathcal{F}_{2}\left(
    \alpha 2^{\lambda-2},\frac{r+\vartheta'}{2^2} \right)
    }
    \\
    &=
    \left(
      \abs{\mathcal{F}_{1}\left(
          \alpha 2^{\lambda-1},\frac{\vartheta'}{2^2} \right)
      }
      +
      \abs{\mathcal{F}_{1}\left(
          \alpha 2^{\lambda-1},\frac{2+\vartheta'}{2^2} \right)
      }
    \right)
    \abs{\mathcal{F}_{1}\left(
        \alpha 2^{\lambda-2},\frac{\vartheta'}{2} \right)
    }
    \\
    &+
    \left(
      \abs{\mathcal{F}_{1}\left(
          \alpha 2^{\lambda-1},\frac{1+\vartheta'}{2^2} \right)
      }
      +
      \abs{\mathcal{F}_{1}\left(
          \alpha 2^{\lambda-1},\frac{3+\vartheta'}{2^2} \right)
      }
    \right)
    \abs{\mathcal{F}_{1}\left(
        \alpha 2^{\lambda-2},\frac{1+\vartheta'}{2} \right)
    }
    .
  \end{align*}
  By \eqref{eq:L2-mean-value} we have
  \begin{displaymath}
    \abs{\mathcal{F}_{1}\left(
        \alpha 2^{\lambda-2},\frac{\vartheta'}{2} \right)
    }^2
    +
    \abs{\mathcal{F}_{1}\left(
        \alpha 2^{\lambda-2},\frac{1+\vartheta'}{2} \right)
    }^2
    = 1,
  \end{displaymath}
  hence by Cauchy-Schwarz 
  \begin{multline*}
    \left(
      \sum_{r=0}^{2^2-1}
      \abs{\mathcal{F}_{2}\left(
          \alpha 2^{\lambda-2},\frac{r+\vartheta'}{2^2} \right)
      }
    \right)^2
    \\
    \leq
    \left(
      \abs{\mathcal{F}_{1}\left(
          \alpha 2^{\lambda-1},\frac{\vartheta'}{2^2} \right)
      }
      +
      \abs{\mathcal{F}_{1}\left(
          \alpha 2^{\lambda-1},\frac{2+\vartheta'}{2^2} \right)
      }
    \right)^2
    \\
    +
    \left(
      \abs{\mathcal{F}_{1}\left(
          \alpha 2^{\lambda-1},\frac{1+\vartheta'}{2^2} \right)
      }
      +
      \abs{\mathcal{F}_{1}\left(
          \alpha 2^{\lambda-1},\frac{3+\vartheta'}{2^2} \right)
      }
    \right)^2
    ,
  \end{multline*}
  which,
  expanding the square, using \eqref{eq:L2-mean-value} again
  and \eqref{eq:FT-product-formula},
  gives
  \begin{multline*}
    \left(
      \sum_{r=0}^{2^2-1}
      \abs{\mathcal{F}_{2}\left(
          \alpha 2^{\lambda-2},\frac{r+\vartheta'}{2^2} \right)
      }
    \right)^2
    \\
    \leq
    1
    +
    2\abs{\cos\pi\left(
        \alpha 2^{\lambda-1}-\frac{\vartheta'}{2^2} \right)
    }
    \abs{\sin\pi\left(
        \alpha 2^{\lambda-1}-\frac{\vartheta'}{2^2} \right)
    }
    \\
    +
    1
    +
    2\abs{\cos\pi\left(
        \alpha 2^{\lambda-1}-\frac{1+\vartheta'}{2^2} \right)
    }
    \abs{\sin\pi\left(
        \alpha 2^{\lambda-1}-\frac{1+\vartheta'}{2^2} \right)
    }
    ,
  \end{multline*}
  hence
  \begin{multline*}
    \left(
      \sum_{r=0}^{2^2-1}
      \abs{\mathcal{F}_{2}\left(
          \alpha 2^{\lambda-2},\frac{r+\vartheta'}{2^2} \right)
      }
    \right)^2
    \\
    \leq
    2
    +
    \abs{\sin\pi\left(
        \alpha 2^{\lambda}-\frac{\vartheta'}{2} \right)
    }
    +
    \abs{\cos\pi\left(
        \alpha 2^{\lambda}-\frac{\vartheta'}{2} \right)
    }
    \leq 2+\sqrt2
    .
  \end{multline*}
  Inserting this in \eqref{eq:integral-induction-formula-base-2}
  and using \eqref{eq:definition-eta_b}, we deduce
  \begin{multline*}
    \int_0^1
    \abs{\mathcal{F}_{\lambda}\left(\alpha,\vartheta\right)}
    \, d\vartheta
    \leq
    \frac{\left(2+\sqrt2\right)^{1/2}}{4}
    \int_0^1
    \abs{\mathcal{F}_{\lambda-2}\left(\alpha,\vartheta\right)}
    \, d\vartheta
    \\
    =
    2^{2(\eta_2-1)}
    \int_0^1
    \abs{\mathcal{F}_{\lambda-2}\left(\alpha,\vartheta\right)}
    \, d\vartheta
    ,
  \end{multline*}
  and by induction
  \begin{displaymath}
    \int_0^1
    \abs{\mathcal{F}_{\lambda}\left(\alpha,\vartheta\right)}
    \, d\vartheta
    \leq
    2^{2\floor{\frac{\lambda}{2}}(\eta_2-1)}
    \int_0^1
    \abs{\mathcal{F}_{\lambda-2\floor{\frac{\lambda}{2}}}\left(
        \alpha,\vartheta\right)}
    \, d\vartheta
    .
  \end{displaymath}
  If $\lambda$ is even
  then $\floor{\frac{\lambda}{2}}=\frac{\lambda}{2}$,
  $\lambda-2\floor{\frac{\lambda}{2}}=0$,
  and we get \eqref{eq:L1-norm-in-alpha-of-F}.
  \\
  If $\lambda$ is odd
  then $\floor{\frac{\lambda}{2}}=\frac{\lambda-1}{2}$,
  $\lambda-2\floor{\frac{\lambda}{2}}=1$ and,
  using \eqref{eq:definition-eta_b},
  \begin{displaymath}
    \int_0^1
    \abs{\mathcal{F}_{1}\left(\alpha,\vartheta\right)}
    \, d\vartheta
    =
    \int_0^1
    \abs{\cos\pi\left(\alpha-\vartheta\right)}
    \, d\vartheta
    =
    \frac{2}{\pi}
    <
    \frac{\left(2+\sqrt2\right)^{1/4}}{2}
    =
    2^{\eta_2-1},
  \end{displaymath}
  and since
  \begin{math}
    \left(2\floor{\frac{\lambda}{2}} + 1\right)(\eta_2-1)
    =
    \left(\lambda-1 + 1\right)(\eta_2-1)
    =
    \lambda(\eta_2-1),
  \end{math}
  we get \eqref{eq:L1-norm-in-alpha-of-F}.
\end{proof}

\begin{lemma}\label{lemma:L1-mean-value-in-alpha}
  For any integers $b\geq 2$, $\lambda \geq 0$, 
  and any sequence $(\vartheta_1,\ldots,\vartheta_N) \in \R^N$
  which is $\delta>0$ well spaced modulo~$1$,
  we have, for $\alpha\in\R$,
  \begin{equation}\label{eq:L1-mean-value-in-alpha}
    \sum_{n=1}^N
    \abs{\mathcal{F}_{\lambda}\left(\alpha,\vartheta_n\right)}
    \leq
    \left( \frac{1}{\delta} + \frac{\pi}{2} b^{\lambda} \right)
    \int_0^1
    \abs{\mathcal{F}_{\lambda}\left(\alpha,\vartheta\right)}
    \, d\vartheta
    \leq
    \left( \frac{1}{\delta b^{\lambda}} + \frac{\pi}{2}  \right)
    b^{\eta_b\lambda}
    .
  \end{equation}
\end{lemma}
\begin{proof}
  If $\lambda=0$, it is enough to observe that
  since $(\vartheta_1,\ldots,\vartheta_N) \in \R^N$
  is $\delta$ well spaced modulo~$1$,
  we have
  \begin{math}
    N \leq \frac1{\delta}
    .
  \end{math}
  We assume from now that $\lambda\geq 1$.
  By \eqref{eq:definition-F} we have
  \begin{math}
    \abs{\mathcal{F}_\lambda(\alpha,\vartheta)}
    =
    \abs{P_{\alpha,\lambda}(\vartheta)}
  \end{math}
  where
  \begin{displaymath}
    P_{\alpha,\lambda}(\vartheta)
    =
    \frac{1}{b^\lambda}
    \sum_{- \frac{b^\lambda}{2} \leq \ell < \frac{b^\lambda}{2}}
    \e\left(
      \alpha R_\lambda\left(\ell + \floor{\frac{b^\lambda}{2}} \right)
      -\vartheta \ell \right)    
  \end{displaymath}
  is a trigonometric polynomial of degree
  $\floor{\frac{b^\lambda}{2}}$.
  
  By Lemma~\ref{lemma:sobolev-gallagher} we have
  \begin{displaymath}
    \sum_{n=1}^N
    \abs{\mathcal{F}_{\lambda}\left(\alpha,\vartheta_n\right)}
    =
    \sum_{n=1}^N
    \abs{P_{\alpha,\lambda}\left(\vartheta_n\right)}
    \leq
    \frac{1}{\delta} 
    \int_0^1
    \abs{P_{\alpha,\lambda}(\vartheta)}
    \, d\vartheta
    +
    \frac12 
    V_{P_{\alpha,\lambda}}
    .
  \end{displaymath}
  and by Lemma~\ref{lemma:bernstein-zygmund} we have
  \begin{displaymath}
    V_{P_{\alpha,\lambda}}
    =
    \int_0^1
    \abs{P'_{\alpha,\lambda}(\vartheta)}
    \, d\vartheta
    \leq
    2\pi
    \floor{\frac{b^\lambda}{2}}
    \int_0^1
    \abs{P_{\alpha,\lambda}(\vartheta)}
    \, d\vartheta
    ,
  \end{displaymath}
  hence
  \begin{displaymath}
    \sum_{n=1}^N
    \abs{\mathcal{F}_{\lambda}\left(\alpha,\vartheta_n\right)}
    \leq
    \left(
      \frac{1}{\delta}
      +
      \frac{2\pi}{2}
      \floor{\frac{b^\lambda}{2}}
    \right)
    \int_0^1
    \abs{\mathcal{F}_\lambda(\alpha,\vartheta)}
    \, d\vartheta
    ,
  \end{displaymath}
  which, combined with Lemma \ref{lemma:L1-norm-in-alpha-of-F},
  leads to \eqref{eq:L1-mean-value-in-alpha}.
\end{proof}

\begin{lemma}\label{lemma:L1-norm-modular-via-sobolev-gallagher}
  Let $b\geq 2$ and $0\leq \delta \leq \lambda$ be integers.
  Let $d \dv b^{\lambda-\delta}$ with $\gcd(d,b)<b$
  and $\delta_1\geq 1$
  such that $b^{\delta_1-1} \leq d < b^{\delta_1}$.
  For $a\in\Z$ and $\alpha\in\R$,
  we have
  \begin{multline}
    \label{eq:L1-norm-modular-via-sobolev-gallagher}
    \sum_{\substack{0\leq h< b^{\lambda}\\ h \equiv a \bmod d b^\delta}}
    \abs{
      \mathcal{F}_{\lambda}\left(\alpha,\frac{h}{b^{\lambda}}\right)
    }
    \\
    \leq
    \left( b + \frac{\pi}{2} \right)
    \frac{b^{\eta_b(\lambda-\delta)}}{d^{\eta_b}}
    \abs{
      \mathcal{F}_{\delta}\left(\alpha,\frac{a}{b^{\delta}}\right)
    }
    G_{\delta_1-1}^{1/2}\left(\alpha b^{\delta}\left(b^2-1\right)\right)
    .
  \end{multline}  
\end{lemma}
\begin{proof}
  If $\delta=\lambda$ then $d=1$ and $\delta_1=1$,
  hence \eqref{eq:L1-norm-modular-via-sobolev-gallagher} holds.
  We assume from now that $0\leq \delta < \lambda$.
  Since $d \dv b^{\lambda-\delta}$ with $\gcd(d,b)<b$
  we have $b^{\delta_1-1} \leq d < b^{\lambda-\delta}$,
  hence $\delta_1 \leq \lambda-\delta$.
  
  We may assume without loss of generality
  that $0\leq a < d b^\delta$.
  
  If $h \equiv a \bmod d b^\delta$ then $h \equiv a \bmod b^\delta$,
  hence by~\eqref{eq:FT-splitting-product}
  with $\lambda'=\lambda-\delta$
  we have
  \begin{align*}
    \abs{\mathcal{F}_\lambda\left(\alpha,\frac{h}{b^{\lambda}}\right)}
    &=
    \abs{\mathcal{F}_{\lambda-\delta}\left(
        \alpha b^{\delta},\frac{h}{b^{\lambda}}
      \right)}
    \cdot
    \abs{\mathcal{F}_{\delta}\left(\alpha,\frac{h}{b^{\delta}} \right)}
    \\
    &=
    \abs{\mathcal{F}_{\lambda-\delta}\left(
        \alpha b^{\delta},\frac{h}{b^{\lambda}}
      \right)}
    \cdot
    \abs{\mathcal{F}_{\delta}\left(\alpha,\frac{a}{b^{\delta}} \right)}
    ,
  \end{align*}
  hence
  \begin{displaymath}
    \sum_{\substack{0\leq h< b^{\lambda}\\ h \equiv a \bmod d b^\delta}}
    \abs{\mathcal{F}_\lambda\left(\alpha,\frac{h}{b^{\lambda}}\right)}
    =
    \abs{\mathcal{F}_{\delta}\left(\alpha,\frac{a}{b^{\delta}} \right)}
    \sum_{\substack{0\leq h< b^{\lambda}\\ h \equiv a \bmod d b^\delta}}
    \abs{\mathcal{F}_{\lambda-\delta}\left(
        \alpha b^{\delta},\frac{h}{b^{\lambda}}
      \right)}
    .
  \end{displaymath}
  Writing
  $h = a+k d b^\delta$,
  remembering that $d\dv b^{\lambda-\delta}$
  we have
  \begin{displaymath}
    \abs{
      \mathcal{F}_{\lambda-\delta}\left(
        \alpha b^{\delta}, \frac{h}{b^{\lambda}}\right)
    }
    =
    \abs{
      \mathcal{F}_{\lambda-\delta}\left(
        \alpha b^{\delta}, \frac{a+k d b^\delta}{b^{\lambda}}\right)
    }
    =
    \abs{
      \mathcal{F}_{\lambda-\delta}\left(
        \alpha b^{\delta},
        \frac{a}{b^{\lambda}}+\frac{k}{b^{\lambda-\delta}/d}\right)
    }
    ,
  \end{displaymath}  
  hence
  \begin{displaymath}
    \sum_{\substack{0\leq h< b^{\lambda}\\ h \equiv a \bmod d b^\delta}}
    \abs{
      \mathcal{F}_{\lambda-\delta}\left(
        \alpha b^{\delta},\frac{h}{b^{\lambda}}\right)
    }
    =
    \sum_{0\leq k < \frac{b^{\lambda-\delta}}{d}}
    \abs{
      \mathcal{F}_{\lambda-\delta}\left(
        \alpha b^{\delta},
        \frac{a}{b^{\lambda}}+\frac{k}{b^{\lambda-\delta}/d}\right)
    }
    .
  \end{displaymath}  
  By~\eqref{eq:FT-splitting-product} we have
  \begin{multline*}
    \abs{
      \mathcal{F}_{\lambda-\delta}\left(
        \alpha b^{\delta},
        \frac{a}{b^{\lambda}}+\frac{k}{b^{\lambda-\delta}/d}\right)
    }
    \\
    =
    \abs{
      \mathcal{F}_{\lambda-\delta-\delta_1}\left(
        \alpha b^{\delta+\delta_1},
        \frac{a}{b^{\lambda}} + \frac{k}{b^{\lambda-\delta}/d}
      \right)
    }
    \abs{
      \mathcal{F}_{\delta_1}\left(
        \alpha b^{\delta},
        \frac{a}{b^{\delta+\delta_1}} + \frac{k d}{b^{\delta_1}}
      \right)
    }
    ,
  \end{multline*}
  and by \eqref{eq:uniform-upperbound-of-F}
  \begin{displaymath}
    \abs{
      \mathcal{F}_{\delta_1}\left(
        \alpha b^{\delta},
        \frac{a}{b^{\delta+\delta_1}}
        +\frac{k d}{b^{\delta_1}}\right)
    }
    \leq
    G_{\delta_1-1}^{1/2}\left(\alpha b^{\delta} \left(b^2-1\right) \right)
    ,
  \end{displaymath}
  which is independent of $k$.
  For $0\leq k < b^{\lambda-\delta}/d$,
  the points
  \begin{math}
      \frac{k}{b^{\lambda-\delta}/d}
  \end{math}
  are
  \begin{math}
    \frac{d}{b^{\lambda-\delta}}
  \end{math}
  well spaced modulo $1$. 
  By Lemma~\ref{lemma:L1-mean-value-in-alpha} we have
  \begin{multline*}
    \sum_{0\leq k < \frac{b^{\lambda-\delta}}{d}}
    \abs{
      \mathcal{F}_{\lambda-\delta-\delta_1}\left(
        \alpha b^{\delta+\delta_1},
        \frac{a}{b^{\lambda}} + \frac{k}{b^{\lambda-\delta}/d}
      \right)
    }
    \\
    \leq
    \left(\frac{b^{\delta_1}}{d} + \frac{\pi}{2} \right)
    b^{\eta_b\left(\lambda-\delta-\delta_1\right)}
    \leq
    \left( b + \frac{\pi}{2} \right)
    \frac{b^{\eta_b\left(\lambda-\delta\right)}}{d^{\eta_b}}
    ,
  \end{multline*}
  and gathering the previous estimates we get
  \eqref{eq:L1-norm-modular-via-sobolev-gallagher}.
\end{proof}

\bigskip

\subsection{$L^\kappa$ norms}~

We recall that the Gauss error function $\erf(x)$ 
and Gauss error complementary function $\erfc(x)$ are defined by
\begin{equation}
  \label{eq:definition-gauss-error-function}
  \erf(x) = \frac{2}{\sqrt\pi} \int_0^x e^{-u^2} \, du
  = 1 - \erfc(x).
\end{equation}
\begin{lemma}\label{lemma:maximum-Tb}
  For any integer $b\geq 2$ and any real number $\kappa>0$
  the function $T_{b,\kappa}$ defined on $\R$ by
  \begin{equation}\label{eq:definition-Tb}
    T_{b,\kappa}\left(\alpha\right)
    =
    \frac{1}{b} \sum_{\ell = 0}^{b-1}
    K_b^{\kappa}\left(\frac{1}{b+1} \norm{\frac{\alpha+\ell}{b}}
    \right)
    \in \left] 0,1 \right]
  \end{equation}
  is $1$-periodic, even,
  and satisfies, taking
  \begin{math}
        B = \frac{\pi^2}{6} \frac{b-1}{b+1},
  \end{math}
  \begin{equation}
    \label{eq:max-Tb}
    \frac1b
    \leq
    \max_{\alpha \in \R} T_{b,\kappa}\left(\alpha\right)
    \leq
    \frac1b
    +
    \sqrt{\frac{\pi}{B \kappa}}
    -
    \sqrt{\frac{\pi}{B \kappa}}
    \erfc\left(\frac12 \sqrt{B \kappa}\right)
    -
    \frac1b K_b^{\kappa}\left(\frac{1}{2b+2}\right)
    .
  \end{equation}
  In particular
  \begin{equation}
    \label{eq:max-Tb-simple}
    \frac1b
    \leq
    \max_{\alpha \in \R} T_{b,\kappa}\left(\alpha\right)
    <
    \frac1b
    +
    \sqrt{\frac{6 (b+1)}{\pi (b-1) \, \kappa}}
    .
  \end{equation}
  Moreover, for $0 < \kappa \leq 1$, we have
  \begin{equation}\label{eq:max-Tb-small-kappa}
    \max_{\alpha \in \R}
    T_{b,\kappa}\left(\alpha\right)
    =
    T_{b,\kappa}\left(\frac{b+1}{2}\right)
    < 1.
  \end{equation}
\end{lemma}
\begin{proof}
  Since $t\mapsto \norm{t}$ is $1$-periodic and even,
  it is easy to check that the function $T_{b,\kappa}$
  is also $1$-periodic and even.
  To prove the lower bound given by \eqref{eq:max-Tb}, it is enough
  to observe that:
  \begin{displaymath}
    \frac1b =
    \frac1b K_b^{\kappa}\left(0 \right)
    \leq
    T_{b,\kappa}\left(0\right)
    \leq
    \max_{\alpha \in \R} T_{b,\kappa}\left(\alpha\right)
    .
  \end{displaymath}
  Let us prove the upper bound given by \eqref{eq:max-Tb}.
  The function
  \begin{math}
    t \mapsto K_b^{\kappa}\left(\frac{\norm{t}}{b+1}\right)
  \end{math}
  is a $1$-periodic continuous function of bounded
  variation on $\left[-\frac12,\frac12\right]$,
  increasing on $[-\frac12,0]$ and decreasing on $[0,\frac12]$
  (by Lemma~\ref{lemma:concavity-of-Dirichlet-kernel}),
  with total variation
  \begin{displaymath}
    V
    =
    2
    \left(
      K_b^{\kappa}\left(0\right)
      -
      K_b^{\kappa}\left(\frac{1}{2b+2}\right)
    \right)
    =
    2
    \left(
      1
      -
      K_b^{\kappa}\left(\frac{1}{2b+2}\right)
    \right)
    .
  \end{displaymath}
  Let $\alpha\in\R$.
  The sequence $\left(\frac{\alpha+\ell}{b}\right)_{0\leq \ell<b}$
  is $\frac1b$ well spaced modulo $1$, hence
  by Lemma~\ref{lemma:sobolev-gallagher} it follows that
  \begin{displaymath}
   b \, T_{b,\kappa}\left(\alpha\right)
   \leq
   b \int_{-1/2}^{1/2} K_b^{\kappa}\left(\frac{t}{b+1}\right) \, dt
   +
   \frac{V}{2}
   .
  \end{displaymath}
  By \eqref{eq:majoration-U} we deduce
  \begin{displaymath}
   b \, T_{b,\kappa}\left(\alpha\right)
   \leq
   b \int_{-1/2}^{1/2}
   \exp\left(
     -\,\frac{\pi^2}{6}(b^2-1)
     \kappa
     \left(\frac{t}{b+1}\right)^2
   \right)
   \, dt
   +
   1 - K_b^{\kappa}\left(\frac{1}{2b+2}\right)
   .
  \end{displaymath}
  Taking
  \begin{math}
    u = t \sqrt{B \kappa},
  \end{math}
  we obtain
  \begin{displaymath}
   b \, T_{b,\kappa}\left(\alpha\right)
   \leq
   \frac{b}{\sqrt{B \kappa}}
   \int_{-\frac12 \sqrt{B \kappa}}^{\frac12 \sqrt{B \kappa}}
   \exp\left(- \, u^2 \right)
   \, du
   +
   1 - K_b^{\kappa}\left(\frac{1}{2b+2}\right)
   ,
  \end{displaymath}
  and, using \eqref{eq:definition-gauss-error-function},
  we get the upper bound of \eqref{eq:max-Tb}.
  
  The upper bound of \eqref{eq:max-Tb-simple} follows from the fact
  that $\erfc>0$ and
  \begin{math}
    K_b^{\kappa}\left(\frac{1}{2b+2}\right) > 0
    .
  \end{math}
  
  Let us now prove \eqref{eq:max-Tb-small-kappa}.
  By symmetry and periodicity
  we can assume that $0\leq \alpha \leq 1/2$.
  By Lemma~\ref{lemma:concavity-of-Dirichlet-kernel},
  $K_b$ is concave on $[0,(2b-2)^{-1}]$,
  hence for $0 < \kappa \leq 1$, we have on $\left[0,(2b-2)^{-1}\right[$:
  \begin{displaymath}
    \left(K_b^\kappa\right)''
    =
    \kappa(\kappa-1) K_b^{\kappa-2} \left(K'_b\right)^2
    +
    \kappa K_b^{\kappa-1} K''_b
    \leq 0
    ,
  \end{displaymath}
  thus $K_b^\kappa$ is concave on $[0,(2b-2)^{-1}]$.
  
  Let us first assume that $b$ is odd.
  We have
  \begin{align*}
    &T_{b,\kappa}\left(\alpha\right)
      =
      \frac{1}{b}
      K_b^\kappa\left(\frac{1}{b+1} \norm{\frac{\alpha}{b}} \right)
    \\
    &\qquad\qquad
      +
      \frac{1}{b} \sum_{\ell = 1}^{(b-1)/2}
      \left(
      K_b^\kappa\left(
      \frac{1}{b+1} \norm{\frac{\alpha+\ell}{b}}
      \right)
      +
      K_b^\kappa\left(
      \frac{1}{b+1} \norm{\frac{\alpha-\ell}{b}}
      \right)
      \right)
    \\
    &=
      \frac{1}{b}
      K_b^\kappa\left(\frac{\alpha}{b(b+1)} \right)
      +
      \frac{1}{b}
      \sum_{\ell = 1}^{(b-1)/2}
    \left(
      K_b^\kappa\left( \frac{\ell+\alpha}{b(b+1)} \right)
      +
      K_b^\kappa\left( \frac{\ell-\alpha}{b(b+1)} \right)
    \right)
    .
  \end{align*}
  Since
  \begin{displaymath}
    0 \leq
    \frac{\ell-\alpha}{b(b+1)}
    \leq
    \frac{\ell+\alpha}{b(b+1)}
    \leq
    \frac{1}{2(b+1)}
    \leq
    \frac{1}{2(b-1)}
    ,
  \end{displaymath}
  and $K_b^\kappa$ is concave on $[0,(2b-2)^{-1}]$,
  we have
  \begin{displaymath}
    \frac12 K_b^\kappa\left( \frac{\ell-\alpha}{b(b+1)} \right)
    +
    \frac12 K_b^\kappa\left( \frac{\ell+\alpha}{b(b+1)} \right)
    \leq
    K_b^\kappa\left( \frac{\ell}{b(b+1)} \right)
    .
  \end{displaymath}
  Moreover,
  by Lemma~\ref{lemma:concavity-of-Dirichlet-kernel},
  $K_b^\kappa$ is decreasing on $[0,(2b-2)^{-1}]$.
  We deduce, for $b$ odd, that
  \begin{displaymath}
    T_{b,\kappa}\left(\alpha\right)
    \leq
    \frac{1}{b}
    K_b^\kappa\left( 0 \right)
    +
    \frac{2}{b} \sum_{\ell = 1}^{(b-1)/2}
    K_b^\kappa\left( \frac{\ell}{b(b+1)} \right)
    =
    T_{b,\kappa}\left(0\right)
    =
    T_{b,\kappa}\left(\frac{b+1}{2}\right)
    < 1
    .
  \end{displaymath}
  
  Let us now assume that $b$ is even.
  We have
  \begin{align*}
    T_{b,\kappa}\left(\alpha\right)
    &=
      \frac{1}{b} \sum_{\ell = 1}^{b/2}
      \left(
      K_b^\kappa\left(
      \frac{1}{b+1} \norm{\frac{\alpha+\ell-1}{b}}
      \right)
      +
      K_b^\kappa\left(
      \frac{1}{b+1} \norm{\frac{\alpha-\ell}{b}}
      \right)
      \right)
    \\
    &=
      \frac{1}{b}
      \sum_{\ell = 1}^{b/2}
    \left(
      K_b^\kappa\left( \frac{\ell-1+\alpha}{b(b+1)} \right)
      +
      K_b^\kappa\left( \frac{\ell-\alpha}{b(b+1)} \right)
    \right)
    \\
    &=
      \frac{1}{b}
      \sum_{\ell = 1}^{b/2}
    \left(
      K_b^\kappa\left( \frac{\ell-\frac12+\left(\alpha-\frac12\right)}{b(b+1)} \right)
      +
      K_b^\kappa\left( \frac{\ell-\frac12-\left(\alpha-\frac12\right)}{b(b+1)} \right)
    \right)
    .
  \end{align*}
  Since $K_b^\kappa$ is concave on $[0,(2b-2)^{-1}]$,
  we deduce, for $b$ even, that
  \begin{displaymath}
    T_{b,\kappa}\left(\alpha\right)
    \leq
    \frac{2}{b}
    \sum_{\ell = 1}^{b/2}
    K_b^\kappa\left( \frac{\ell-\frac12}{b(b+1)} \right)
    =
    T_{b,\kappa}\left(\frac12\right)
    =
    T_{b,\kappa}\left(\frac{b+1}{2}\right)
    < 1
    .
    \qedhere
  \end{displaymath}
\end{proof}

\begin{lemma}\label{lemma:approximation-Tb}
  For any integer $b\geq 2$ and any real number $\kappa>0$
  we have for any $\alpha\in\R$
  \begin{multline}\label{eq:approximation-Tb}
    \abs{
      T_{b,\kappa}\left(\alpha\right)
      -
      \int_{-1/2}^{1/2}
      K_b^{\kappa}\left( \frac{t}{b+1} \right)
      \, dt
    }
    \\
    \leq
    \frac{1}{b}
    \left(1-\left(\frac1b\cot\frac{\pi}{2b+2}\right)^\kappa\right)
    \leq
    \frac{1}{b}
    \left(1-\frac{2^\kappa}{\pi^\kappa}\right)
    .
  \end{multline}
  For $\kappa=1$, the integral above satisfies
  \begin{equation}\label{eq:integral-approximation-Tb}
    \abs{
      \int_{-1/2}^{1/2}
      K_b\left( \frac{t}{b+1} \right)
      \, dt
      -
      \int_{-1/2}^{1/2}
      \frac{\sin \pi t}{\pi t}
      \, dt
    }
    \leq
    \frac{\pi^2}{24 b} + \frac{\pi^2}{3} \frac{b+1}{b(2b+1)}
    .
  \end{equation}
\end{lemma}
\begin{proof}
  Taking
  \begin{math}
    f(t) =
    K_b^{\kappa}\left( \frac{1}{b+1} \norm{\frac{\alpha}{b}+t} \right)
    ,
  \end{math}
  by Lemma~\ref{lemma:Euler-MacLaurin-order-1}
  we have
  \begin{displaymath}
    \abs{
      T_{b,\kappa}\left(\alpha\right)
      -
      \int_{-1/2}^{1/2}
      K_b^{\kappa}\left( \frac{t}{b+1} \right)
      \, dt
    }
    \leq
    \frac{V}{2b}
    ,
  \end{displaymath}
  where $V$ is the total variation of the function $f$.
  By periodicity,
  $V$ is also the total variation of the function
  \begin{math}
    t \mapsto K_b^{\kappa}\left(\frac{\norm{t}}{b+1}\right)
    .
  \end{math}
  By the proof of Lemma~\ref{lemma:maximum-Tb},
  \begin{displaymath}
    V
    =
    2
    \left(
      K_b^{\kappa}\left(0\right)
      -
      K_b^{\kappa}\left(\frac{1}{2b+2}\right)
    \right)
    =
    2
    \left(
      1
      -
      K_b^{\kappa}\left(\frac{1}{2b+2}\right)
    \right)
    .
  \end{displaymath}
  We have
  \begin{displaymath}
    K_b\left(\frac{1}{2b+2}\right)
    =
    \frac{\sin\left(\frac\pi{2}-\frac{\pi}{2b+2}\right)}{b\sin\frac{\pi}{2b+2}}
    =
    \frac1b \cot\frac{\pi}{2b+2}
    ,
  \end{displaymath}
  which gives the first inequality of \eqref{eq:approximation-Tb}.
  Using $\sin u \leq u$ and $\cos u \geq 1-\frac{u^2}{2}$,
  it follows that
  \begin{displaymath}
    K_b\left(\frac{1}{2b+2}\right)
    \geq
    \frac{2b+2}{\pi b} \left( 1 - \frac{\pi^2}{2(2b+2)^2} \right)
    =
    \frac{2}{\pi}
    + \frac{8(b+1)-\pi^2}{2\pi b(2b+2)}
    \geq
    \frac{2}{\pi}
    ,
  \end{displaymath}
  which gives the second inequality of~\eqref{eq:approximation-Tb}.

  Let us now prove \eqref{eq:integral-approximation-Tb}.
  For $t\in[-1/2,1/2]$ we have
  \begin{displaymath}
    K_b\left(\frac{t}{b+1}\right)
    =
    \frac{\sin\left(\pi t -\frac{\pi t}{b+1}\right)}{b\sin\frac{\pi t}{b+1}}
    =
    \frac1b \sin(\pi t) \cot\frac{\pi t}{b+1}
    -
    \frac1b \cos(\pi t)
  \end{displaymath}
  and, using the power series expansion
  (see e.g. \cite[Proposition 9.1.4]{cohen-NT-II-2007})
  \begin{displaymath}
    \pi u \cot \pi u
    =
    1
    -
    \sum_{m=1}^\infty
    \abs{B_{2m}} \frac{(2\pi)^{2m}}{(2m)!} u^{2m}
    \qquad
    (\abs{u}<1)
  \end{displaymath}
  where $(B_n)$ denotes the sequence of Bernoulli numbers,
  which satisfy
  \begin{displaymath}
    \abs{B_{2m}}
    =
    \frac{2 (2m)! \zeta(2m)}{(2\pi)^{2m}}
  \end{displaymath}
  (see e.g. \cite[Corollary 9.1.21]{cohen-NT-II-2007}) we get
  \begin{multline*}
    K_b\left(\frac{t}{b+1}\right)
    \\
    =
    \frac{b+1}b 
    \frac{\sin(\pi t)}{\pi t}
    -
    \frac{b+1}b 
    \frac{\sin(\pi t)}{\pi t}
    \sum_{m=1}^\infty
    2 \zeta(2m)
    \left(\frac{t}{b+1}\right)^{2m}
    -
    \frac1b \cos(\pi t)
    .
  \end{multline*}
  It follows that
  \begin{multline*}
    \abs{
      \int_{-1/2}^{1/2} K_b\left(\frac{t}{b+1}\right) \, dt
      -
      \int_{-1/2}^{1/2} \frac{\sin(\pi t)}{\pi t} \, dt
    }
    \\
    \leq
    \abs{
      \frac{1}b
      \int_{-1/2}^{1/2}
      \left( \frac{\sin(\pi t)}{\pi t} - \cos (\pi t) \right)
      \, dt
    }
    +
    \frac{b+1}b
    2 \zeta(2)
    \int_{-1/2}^{1/2}
    \sum_{m=1}^\infty
    \left(\frac{t}{b+1}\right)^{2m}
    \, dt
    .
  \end{multline*}
  Since for $u\in\R$ (by continuity for $u=0$) we have
  \begin{math}
    1 - \frac{u^2}{6} \leq \frac{\sin u}{u} \leq 1
  \end{math}
  and
  \begin{math}
    1 - \frac{u^2}{2} \leq \cos u \leq 1
  \end{math}
  which gives
  \begin{displaymath}
    - \frac{u^2}{6} \leq \frac{\sin u}{u} - \cos u \leq \frac{u^2}{2}
  \end{displaymath}
  we obtain
  \begin{displaymath}
    \abs{
      \frac{1}b
      \int_{-1/2}^{1/2}
      \left( \frac{\sin(\pi t)}{\pi t} - \cos (\pi t) \right)
      \, dt
    }
    \leq
    \frac{1}b
    \int_{-1/2}^{1/2}
    \frac{(\pi t)^2}{2} 
    \, dt
    =
    \frac{\pi^2}{24 b}
    ,
  \end{displaymath}
  and
  \begin{align*}
    \int_{-1/2}^{1/2}
    \sum_{m=1}^\infty
    \left(\frac{t}{b+1}\right)^{2m}
    \, dt
    &=
    2 (b+1)
    \int_{0}^{1/(2b+2)}
    \sum_{m=1}^\infty
    u^{2m}
    \, du
    \\
    &=
    2 (b+1)
    \left( \artanh\left(\frac{1}{2b+2}\right) - \frac{1}{2b+2} \right)
    \\
    &=
    2 (b+1)
    \left(
      \frac12
      \log \frac{1+\frac{1}{2b+2}}{1-\frac{1}{2b+2}}
      - \frac{1}{2b+2}
    \right)
    \\
    &=
    2 (b+1)
    \left(
      \frac12
      \log \left( 1 + \frac{2}{2b+1} \right)
      - \frac{1}{2b+2}
    \right)
    \\
    &\leq
    2 (b+1)
    \left(
      \frac{1}{2b+1} - \frac{1}{2b+2}
    \right)
    =
    \frac{1}{2b+1}
    ,
  \end{align*}
  and since $\zeta(2)=\pi^2/6$ we get \eqref{eq:integral-approximation-Tb}.
\end{proof}

\begin{lemma}\label{lemma:Lkappa-norm-of-G}
  For any integers $b\geq 2$ and $0\leq\lambda\leq \lambda'$
  and any real number $\kappa>0$,
  and for any $1$-periodic function $\Phi:\R\to\R^+$ integrable on $[0,1]$,
  we have
  \begin{equation}\label{eq:Lkappa-norm-of-G}
    \int_0^1
    G_\lambda^{\kappa}(\alpha) \,
    \Phi\left(\alpha b^{\lambda'}\right)
    d\alpha
    \leq
    \left(\max_{\alpha\in\R} T_{b,\kappa}\left(\alpha\right) \right)^\lambda
    \norm{\Phi}_1
    =
    b^{-\zeta_{b,\kappa} \lambda}
    \norm{\Phi}_1
    ,
  \end{equation}
  where $T_{b,\kappa}$ is defined by \eqref{eq:definition-Tb}
  and
  \begin{equation}\label{eq:definition-zeta_b}
    \zeta_{b,\kappa}
    =
    \frac{
      - \log \left(
        \max_{\alpha\in\R} T_{b,\kappa}\left(\alpha\right)
      \right)}{\log b}
    \geq 0
    .
  \end{equation}
  We notice that for $0<\kappa\leq 1$,
  by \eqref{eq:max-Tb-small-kappa} we have $\zeta_{b,\kappa} > 0$,
  and 
  for
  \begin{math}
    \kappa \geq \frac{6 b^2(b+1)}{\pi (b-1)^3}
    ,
  \end{math}
  by \eqref{eq:max-Tb-simple}
  we also have $\zeta_{b,\kappa} > 0$.
\end{lemma}
\begin{proof}
  For integers $1\leq \lambda\leq \lambda'$ and $\alpha\in\R$, we have
  \begin{equation*}
    G_\lambda^{\kappa}(\alpha)
    =
    K_b^{\kappa}\left(\frac{\norm{\alpha}}{b+1}\right)
    G_{\lambda-1}^{\kappa}(\alpha b).
  \end{equation*}
  Since $\alpha \mapsto G_{\lambda-1}^{\kappa}(\alpha b)$ has period
  $\frac{1}{b}$, we obtain
  \begin{align*}
    \int_0^1
    &
      G_\lambda^{\kappa}(\alpha)
    \Phi\left(\alpha b^{\lambda'}\right)
    d\alpha
    \\
    &=
    \sum_{\ell = 0}^{b-1} \int_0^{1/b}
    K_b^{\kappa}\left(\frac{1}{b+1} \norm{\alpha+\frac{\ell}{b}} \right)
    G_{\lambda-1}^{\kappa}(\alpha b)
    \Phi\left(\alpha b^{\lambda'}\right)
    d\alpha
    \\
    &= 
    \int_0^1
    T_{b,\kappa}\left(\alpha\right)
    G_{\lambda-1}^{\kappa}(\alpha)
    \Phi\left(\alpha b^{\lambda'-1}\right)
    d\alpha
    ,
  \end{align*}
  hence
  \begin{displaymath}
    \int_0^1 G_\lambda^{\kappa}(\alpha)
    \Phi\left(\alpha b^{\lambda'}\right)
    d\alpha
    \leq
    \left(\max_{\alpha\in\R} T_{b,\kappa}\left(\alpha\right) \right)
    \int_0^1
    G_{\lambda-1}^{\kappa}(\alpha)
    \Phi\left(\alpha b^{\lambda'-1}\right)
    d\alpha
  \end{displaymath}
  and by induction we get
  \begin{displaymath}
    \int_0^1 G_\lambda^{\kappa}(\alpha)
    \Phi\left(\alpha b^{\lambda'}\right)
    d\alpha
    \leq
    \left(\max_{\alpha\in\R} T_{b,\kappa}\left(\alpha\right) \right)^\lambda
    \int_0^1
    \Phi\left(\alpha b^{\lambda'-\lambda}\right)
    d\alpha
    .
  \end{displaymath}
  Since $b\geq 2$ is an integer, we have
  \begin{displaymath}
    \int_0^1
    \Phi\left(\alpha b^{\lambda'-\lambda}\right)
    d\alpha
    =
    b^{\lambda-\lambda'}
    \int_0^{b^{\lambda'-\lambda}}
    \Phi\left(u\right)
    du
    =
    \norm{\Phi}_1,
  \end{displaymath}
  and the result follows.
\end{proof}

\begin{lemma}\label{lemma:L_1_norm_of-G'}
  For any integers $b\geq 2$ and  $\lambda\geq 1$
  and any real number $\kappa >0$,
  $G_{\lambda}^{\kappa}$ admits almost everywhere a derivative
  $\left(G_{\lambda}^{\kappa} \right)'$ which is integrable
  and satisfies
  \begin{displaymath}
    \norm{\left(G_{\lambda}^{\kappa} \right)'}_1
    \leq
    \kappa \,
    b^\lambda \,
    \norm{G_{\lambda}^{\kappa}}_1
    .
  \end{displaymath}
\end{lemma}
\begin{proof}
  By \eqref{eq:definition-G}, for $\lambda\geq 1$ and $\alpha\in\R$,
  we have
  \begin{displaymath}
    G_{\lambda}(\alpha)
    =
    \prod_{j=0}^{\lambda-1}
    K_b \left( \frac{\norm{\alpha b^j}}{b+1}\right)
    ,
  \end{displaymath}
  where $K_b$ is defined by \eqref{eq:definition-Kb}.
  Hence, for almost all $\alpha \in \R$, we have
  \begin{displaymath}
    \abs{\left(G_{\lambda}^{\kappa} \right)'(\alpha)}
    \leq
    \abs{G_{\lambda}^{\kappa}(\alpha)}
    \sum_{j=0}^{\lambda-1}
    \kappa \frac{b^j}{b+1}
    \abs{\frac{K'_b}{K_b}\left(\frac{\norm{\alpha b^j}}{b+1}\right)}
    ,
  \end{displaymath}
  and by Lemma~\ref{lemma:concavity-of-Dirichlet-kernel},
  $K_b$ is positive and decreasing on $[0,(2b-2)^{-1}]$,
  $K'_b$ is negative and decreasing on $[0,(2b-2)^{-1}]$, so that
  $\abs{K'_b}$ is increasing on $[0,(2b-2)^{-1}]$.
  It follows that $\abs{K'_b/K_b}$ is increasing
  on $\left[0,(2b-2)^{-1}\right[$.
  Since
  \begin{math}
    \norm{\alpha b^j} \leq \frac{1}{2}
    ,
  \end{math}
  we deduce that
  \begin{displaymath}
    \abs{\frac{K'_b}{K_b}\left(\frac{\norm{\alpha b^j}}{b+1}\right)}
    \leq
    \abs{\frac{K'_b}{K_b}\left(\frac{1}{2b+2}\right)}
  \end{displaymath}
  and
  \begin{align*}
    \abs{\left(G_{\lambda}^{\kappa} \right)'(\alpha)}
    \leq &
    \abs{G_{\lambda}^{\kappa}(\alpha)}
    \abs{\frac{K'_b}{K_b}\left(\frac{1}{2b+2}\right)}
    \frac{\kappa}{b+1}
    \sum_{j=0}^{\lambda-1} b^j
    \\
    &=
    \kappa 
    \abs{G_{\lambda}^{\kappa}(\alpha)}
    \abs{\frac{K'_b}{K_b}\left(\frac{1}{2b+2}\right)}
    \cdot
    \frac{b^{\lambda} -1}{b^2-1}
    .
  \end{align*}
  For $b\geq 3$ we have by \eqref{eq:Kb-logarithmic-derivative}
  \begin{align*}
    \abs{\frac{K'_b}{K_b}\left(\frac{1}{2b+2}\right)}
    &=
    - \frac{K'_b}{K_b}\left( \frac1{2b+2}\right)
    =
    \pi \cot\frac{\pi}{2b+2}
    -
    \pi b \cot\frac{\pi b}{2b+2}
    \\
    &\leq
    \pi \cot\frac{\pi}{2b+2}
    \leq 2b+2
    \leq b^2-1
    ,
  \end{align*}
  while for $b=2$ we have
  \begin{displaymath}
    \abs{\frac{K'_b}{K_b}\left(\frac{1}{2b+2}\right)}
    =
    \abs{\frac{K'_2}{K_2}\left(\frac{1}{6}\right)}
    =
    \frac{\pi}{\sqrt{3}} < b^2-1 = 3
    .
  \end{displaymath}
  Hence the result holds for any integer $b\geq 2$.
\end{proof}

\begin{lemma}\label{lemma:sobolev-gallagher-for-G}
  For any integers $b\geq 2$ and $\delta \geq 1$,
  any real number $\kappa>0$ 
  and any sequence $(x_1,\ldots,x_N) \in \R^N$
  which is $b^{-\delta}$ well spaced modulo~$1$, we have
  \begin{displaymath}
    \sum_{n=1}^N G_\delta^\kappa(x_n)
    \leq
    \left( \frac{\kappa}{2}+1 \right)
    b^\delta \norm{G_\delta^\kappa}_1
    \leq
    \left( \frac{\kappa}{2}+1 \right)
    b^{(1-\zeta_{b,\kappa})\delta}
    ,
  \end{displaymath}
  where $\zeta_{b,\kappa}$ is defined by \eqref{eq:definition-zeta_b}.
\end{lemma}
\begin{proof}
  By Lemma~\ref{lemma:L_1_norm_of-G'},
  $G_\delta^\kappa$ is of bounded variation on $[0,1]$,
  and by Lemma~\ref{lemma:sobolev-gallagher} we have
  \begin{displaymath}
    \sum_{n=1}^N \abs{G_\delta^\kappa(x_n)}
    \leq
    b^{\delta} \int_0^1 \abs{G_\delta^\kappa(u)} \, du
    + \frac12 \int_0^1 \abs{\left(G_\delta^\kappa\right)'(u)} \, du
    ,
  \end{displaymath}
  so that by Lemma~\ref{lemma:L_1_norm_of-G'}, 
  we get the first inequality
  \begin{displaymath}
    \sum_{n=1}^N G_\delta^\kappa(x_n)
    \leq
    \left( \frac{\kappa}{2}+1 \right)
    b^\delta \norm{G_\delta^\kappa}_1
    ,
  \end{displaymath}
  and the second inequality follows by Lemma~\ref{lemma:Lkappa-norm-of-G}
  with $\Phi=1$.
\end{proof}

\begin{lemma}\label{lemma:sobolev-gallagher-combined-with-holder}
  Let $b\geq 2$, $\delta\geq 1$, $\delta'\geq 1$ and $\kappa\geq 1$
  be integers such that $(\kappa+1)\delta \leq\delta'$,
  and let $(x_1,\ldots,x_N) \in \R^N$ be such that for any integer
  $\ell\geq 0$, the sequence
  \begin{math}
    (x_1 b^{\ell\delta},\ldots, x_N b^{\ell\delta})
  \end{math}
  is $b^{-\delta}$ well spaced modulo~$1$.
  Then
  \begin{multline}\label{eq:sobolev-gallagher-combined-with-holder}
    \sum_{n=1}^N G_{\delta '}^{1/2}(x_n)
    \leq
    \frac{\sqrt{3}}{2}
    b^{(1-\zeta_{b,1})\delta/2} \,
    (\kappa+2)^{1/2}
    b^{(1-\zeta_{b,\kappa})\delta/2} \,
    \\
    \leq
    \frac{\sqrt{3}}{2}
    b^{(1-\zeta_{b,1})\delta/2} \,
    (\kappa+2)^{1/2}
    \left(
      1
      +
      \sqrt{\frac{6\, b^2(b+1)}{\pi (b-1) \, \kappa}}
    \right)^{\delta/2}
    .
  \end{multline}
\end{lemma}
\begin{proof}
  Since $0\leq K_b \leq 1$ on $\left[0,(2b+2)^{-1}\right]$,
  we may drop some terms
  in the product \eqref{eq:definition-G},
  and we get for any $\alpha\in\R$:
  \begin{displaymath}
    G_{\delta'}(\alpha)
    =
    \prod_{j=0}^{\delta'-1}
    K_b\left ( \frac{\norm{\alpha b^j}}{b+1}\right )
    \leq
    \prod_{\ell=0}^\kappa
    \prod_{j=\ell \delta}^{(\ell +1)\delta -1}
    K_b\left( \frac{\norm{\alpha b^j}}{b+1} \right)
    =
    \prod_{\ell=0}^\kappa
    G_{\delta}\left( \alpha b^{\ell\delta} \right)
    .
  \end{displaymath}
  Using Hölder's inequality with exponents
  \begin{math}
    \left(\frac12,\frac{1}{2\kappa}, \ldots ,\frac{1}{2\kappa}\right)
    \in \R^{\kappa+1}
    ,
  \end{math}
  this permits to write
  \begin{align*}
    \sum_{n=1}^N G_{\delta '}^{1/2}(x_n)
    &\leq
    \sum_{n=1}^N 
    \prod_{\ell=0}^\kappa
    G_{\delta}^{1/2}\left( x_n b^{\ell\delta} \right)
    \\
    &\leq
    \left( \sum_{n=1}^N G_{\delta}(x_n) \right)^{1/2}
    \prod_{\ell=1}^\kappa
    \left(
      \sum_{n=1}^N G_{\delta}^{\kappa}\left(x_n b^{\ell\delta}\right)
    \right)^{1/(2\kappa)}
    .
  \end{align*}
  Since the sequence $(x_1,\ldots,x_N)$
  is $b^{-\delta}$ well spaced modulo $1$,
  by Lemma~\ref{lemma:sobolev-gallagher-for-G} we have
  \begin{displaymath}
    \sum_{n=1}^N G_{\delta}(x_n)
    \leq
    \frac32 \,
    b^{\delta} \norm{ G_{\delta}}_1
    \leq
    \frac32 \,
    b^{(1-\zeta_{b,1})\delta}
    ,
  \end{displaymath}
  and since the sequence $(x_1 b^{\ell\delta},\ldots,x_N b^{\ell\delta})$
  is $b^{-\delta}$ well spaced modulo $1$,
  by Lemma~\ref{lemma:sobolev-gallagher-for-G} we have
  \begin{displaymath}
    \sum_{n=1}^N G_{\delta}^\kappa(x_n b^{\ell\delta})
    \leq
    \left( \frac{\kappa}{2}+1\right)
    b^{\delta}\norm{ G_{\delta}^\kappa}_1
    \leq
    \left( \frac{\kappa}{2}+1\right)
    b^{(1-\zeta_{b,\kappa})\delta}
    .
  \end{displaymath}
  It follows that
    \begin{displaymath}
    \sum_{n=1}^N G_{\delta '}^{1/2}(x_n)
    \leq
    \frac{\sqrt{3}}{2}
    b^{(1-\zeta_{b,1})\delta/2} \,
    (\kappa+2)^{1/2}
    b^{(1-\zeta_{b,\kappa})\delta/2} \,
    ,
  \end{displaymath}
  giving the first inequality
  of~\eqref{eq:sobolev-gallagher-combined-with-holder}.
  
  By \eqref {eq:definition-zeta_b} and \eqref{eq:max-Tb-simple},
  we have
  \begin{displaymath}
    b^{-\zeta_{b,\kappa}}
    \leq
    \frac1b
    +
    \sqrt{\frac{6 (b+1)}{\pi (b-1) \, \kappa}}
    ,
  \end{displaymath}
  which leads to the second inequality
  of~\eqref{eq:sobolev-gallagher-combined-with-holder}.
\end{proof}

\section{Palindromic interlude}\label{section:palindromes}

Proposition 7.1 of Tuxanidy and Panario \cite{tuxanidy-panario-2024},
after renormalisation and taking $\lambda=N-1$, $\kappa=K$,
states that for integers $b\geq 2$, $\lambda\geq 1$ and $\kappa\geq 2$
we have
\begin{displaymath}
  \int_0^1
  \prod_{j=0}^{\lambda-1}
  K_b^{2\kappa}\left( \alpha b^{j+1} + \alpha b^{2\lambda+1-j} \right) \, d\alpha
  \leq
  b^{2\kappa-2\lambda}
  \left(1+O\left(\frac1{\sqrt{\kappa}}+\frac{b^2}{\kappa}\right)\right)^{2\lambda-2}
  .
\end{displaymath}
We can save the factor $b^{2\kappa}$ by proving the following
explicit upper bound:
\begin{proposition}\label{prop:maj_int_prod_K_b}
  For any integers $b\geq 2$, $\lambda\geq 1$,
  and any real number $\kappa>0$ we have
  \begin{align*}
    \int_0^1
    \prod_{j=0}^{\lambda-1}
    \abs{
    K_b\left( \alpha b^{j+1} + \alpha b^{2\lambda+1-j} \right)
    }^{2\kappa}
    \, d\alpha
    &\leq
    \left(
      \max_{\alpha\in\R} T_{b,\frac{\kappa}{2}}\left(\alpha\right)
    \right)^{2\lambda-2}
    =
    b^{-\zeta_{b,\frac{\kappa}{2}}(2\lambda-2)}
    \\
    &\leq
    \left(
      \frac1b
      +
      \sqrt{\frac{12 (b+1)}{\pi (b-1) \, \kappa}}
    \right)^{2\lambda-2}
    .
  \end{align*}
\end{proposition}
\begin{proof}
  By \eqref{eq:FT-product-formula}
  we have
  \begin{displaymath}
    \prod_{j=0}^{\lambda-1}
    \abs{K_b\left( \alpha b^{j+1} + \alpha b^{2\lambda+1-j} \right)}^{2\kappa}
    =
    \abs{\mathcal{F}_\lambda\left(\alpha b^{\lambda+2},-\alpha b\right)}^{2\kappa}
    ,
  \end{displaymath}
  and by \eqref{eq:combined-upperbound-of-F} 
  \begin{displaymath}
    \abs{\mathcal{F}_\lambda\left(\alpha b^{\lambda+2},-\alpha b\right)}^{2\kappa}
    \leq
    G_{\lambda-1}^{\kappa/2}\left(\alpha b^{\lambda+2} (b^2-1)\right)
    \
    G_{\lambda-1}^{\kappa/2}\left(\alpha b (b^2-1)\right)
    .
  \end{displaymath}
  Since $b\geq 2$ is an integer and $G_{\lambda-1}$ is $1$-periodic,
  we have
  \begin{align*}
    \int_0^1
    &
      G_{\lambda-1}^{\kappa/2}\left(\alpha b^{\lambda+2} (b^2-1)\right)
    \,
    G_{\lambda-1}^{\kappa/2}\left(\alpha b (b^2-1)\right)
    d\alpha
    \\
    &=
    \frac{1}{b (b^2-1)}
    \int_0^{b (b^2-1)}
    G_{\lambda-1}^{\kappa/2}\left(u b^{\lambda+1} \right)
    \,
    G_{\lambda-1}^{\kappa/2}\left(u \right)
    du
    \\
    &=
    \int_0^{1}
    G_{\lambda-1}^{\kappa/2}\left(u b^{\lambda+1} \right)
    \,
    G_{\lambda-1}^{\kappa/2}\left(u \right)
    du
    .
  \end{align*}
  By  Lemma~\ref{lemma:Lkappa-norm-of-G}
  with $\Phi = G_{\lambda-1}^{\kappa/2}$
  we have
  \begin{displaymath}
    \int_0^{1}
    G_{\lambda-1}^{\kappa/2}\left(u b^{\lambda+1} \right)
    \,
    G_{\lambda-1}^{\kappa/2}\left(u \right)
    du
    \leq
    \left(
      \max_{\alpha\in\R} T_{b,\frac{\kappa}{2}}\left(\alpha\right)
    \right)^{\lambda-1}
    \norm{G_{\lambda-1}^{\kappa/2}}_1
    ,
  \end{displaymath}
  hence again by  Lemma~\ref{lemma:Lkappa-norm-of-G}
  with $\Phi = 1$
  we deduce
  \begin{displaymath}
    \int_0^{1}
    G_{\lambda-1}^{\kappa/2}\left(u b^{\lambda+1} \right)
    \,
    G_{\lambda-1}^{\kappa/2}\left(u \right)
    du
    \leq
    \left(
      \max_{\alpha\in\R} T_{b,\frac{\kappa}{2}}\left(\alpha\right)
    \right)^{2\lambda-2}
    ,
  \end{displaymath}
  and the result follows by \eqref{eq:max-Tb-simple}
  and \eqref{eq:definition-zeta_b}.
\end{proof}

\section{Type II sums}\label{section:type-II-sums}

Let $D\geq 2$ be a real number, 
\begin{equation}\label{eq:definition-A-d}
  \mathcal{A}_D
  =
  \left\{\frac{h}{d}:
    d\leq D,\,
    \gcd(d,b(b^2-1))=1,\,
    1\leq h< d,\, \gcd(h,d)=1\right\},
\end{equation}
and for $\alpha = \frac{h}{d} \in \mathcal{A}_D$, let
\begin{equation}\label{eq:def-W}
  W(\alpha) = \frac{1}{d}.
\end{equation}

Our goal is to estimate the sums arising
from~\eqref{eq:maj_gen_type_II} with $\mathcal{A}$ replaced by
$\mathcal{A}_D$, $W(\alpha)$ defined by \eqref{eq:def-W} and
$f(\alpha,n)$ replaced by $\e(\alpha R_{\lambda}(n))$.

Let $\lambda\geq 2$, $\mu \geq 1$ and $\nu\geq 1$ be integers such that
\begin{equation}
  \label{eq:definition-lambda}
  \lambda = \mu + \nu
  .
\end{equation}
For any integer $\kappa \geq 1$, we define
\begin{equation}\label{eq:definition-I-lambda}
  I_\kappa = \{b^{\kappa-1},\ldots,b^{\kappa}-1\}.
\end{equation}
For $I_\mu$ and $I_\lambda$ defined by \eqref{eq:definition-I-lambda},
a non empty subset $J_\lambda$ of $I_\lambda$ and a sequence
$\mathbf{z}=(z_n)$ of complex numbers of modulus at most $1$, we
consider
\begin{displaymath}
  S_{II}(\alpha,\lambda,\mu,J_{\lambda},\mathbf{z})
  =
  \sum_{m\in I_\mu}
  \abs{
    \sum_{\substack{n\\ mn\in J_\lambda}} z_n \e(\alpha R_{\lambda}(mn))
  }
\end{displaymath}
and
\begin{equation}\label{S_II_alpha_mu_lambda}
  S_{II}(\alpha,\lambda,\mu)
  =
  \sup_{J_\lambda}
  \sup_{\mathbf{z}}
  S_{II}(\alpha,\lambda,\mu,J_{\lambda},\mathbf{z}).
\end{equation}

We will obtain in Lemma~\ref{lemma:conclusion-S_II-individual} an
individual upper bound of $S_{II}(\alpha,\lambda,\mu)$ for specific
values of $\alpha$ and in
Lemma~\ref{lemma:conclusion-S_II-sum-on-alpha}, an upper bound of the
sum
\begin{equation}\label{eq:initial-sum-to-estimate-from-typeII-sums}
  \sum_{\alpha\in\A_D} W(\alpha)  S_{II}(\alpha,\lambda,\mu)
\end{equation}
under the condition that $D=b^{\xi\lambda}$ for some $\xi>0$ and
\begin{equation}\label{eq:condition-mu-nu-typeII}
  \beta_1 \lambda \leq \mu \leq \beta_2 \lambda
  ,
\end{equation}
for some positive real numbers $\beta_1$ and $\beta_2$. In order to
apply Lemma~\ref{lemma:maj_sum_alpha_vaughan}, in
Section~\ref{section:conclusion_expo_sums}, we will choose $\beta_1$
and $\beta_2$ such that
\begin{equation}
  \label{eq:conditions-beta1-beta2-typeII}
  0< \beta_1 < \frac13,\qquad \frac12 < \beta_2 < 1.
\end{equation}

\subsection{Removing the upper digits}
By the Cauchy-Schwarz inequality we have
\begin{equation}\label{eq:majoration-SII-after-Cauchy-Schwarz}
  S_{II}(\alpha,\lambda,\mu,J_{\lambda},\mathbf{z})^2
  \leq
  b^\mu
  \sum_{m\in I_\mu} 
  \abs{\sum_{\substack{n\\ mn\in J_\lambda}} z_n \e(\alpha R_{\lambda}(mn))}^2
  .
\end{equation}
If we expand the square directly, a difference of $R_\lambda(\cdot)$
will appear without control on the variables.
Using van der Corput's inequality (Lemma~\ref{lemma:van-der-corput})
will enable us to take advantage of the relative size of the variables.

For $m\in I_\mu$ and $mn\in J_\lambda$,
we have $n\in I_\nu \cup I_{\nu+1}$:
\begin{displaymath}
  b^{\nu-1} =  b^{\lambda-\mu-1} \leq \frac{b^{\lambda-1}}{m} 
  \leq
  n < \frac{b^\lambda}{m} \leq b^{\lambda-\mu+1} = b^{\nu+1}
  .
\end{displaymath}
Therefore the number of $n$'s in \eqref{eq:majoration-SII-after-Cauchy-Schwarz}
is at most $b^{\nu+1}-b^{\nu-1}$.
Let $\rho$ be an integer such that $0 \leq \rho \leq \nu-1$.
Using Lemma~\ref{lemma:van-der-corput} with $R = b^{\rho}$, we obtain
\begin{multline*}
  S_{II}(\alpha,\lambda,\mu,J_{\lambda},\mathbf{z})^2
  \leq
  b^\mu
  \sum_{m\in I_\mu}
  \frac{b^{\nu+1}}{b^{\rho}}\,
  \Re
  \Bigg(
    b^{\nu+1}
    +
    2   \sum_{r=1}^{b^{\rho}-1}  \left(1 - \frac{r}{b^{\rho}}\right)
    \\
    \sum_{\substack{n\\ mn\in J_\lambda \\ m(n+r) \in J_\lambda}}
    z_{n+r} \conjugate{z_n}
    \e(\alpha (R_{\lambda}(m(n+r))- R_{\lambda}(mn)))
  \Bigg)
  ,
\end{multline*}
hence
\begin{equation}\label{eq:maj_SIIcarre_T}
  S_{II}(\alpha,\lambda,\mu,J_{\lambda},\mathbf{z})^2
  \leq
  \frac{b^{2\mu+2\nu+2}}{b^{\rho}}
  +
  \frac{2 \, b^{\mu+\nu+1}}{b^{\rho}}
  \sum_{r=1}^{b^{\rho}-1} 
  T_1(\alpha,\lambda,\mu,J_\lambda,r)
\end{equation}
where
\begin{equation}
  \label{eq:def_T_sum}
  T_1(\alpha,\lambda,\mu,J_\lambda,r)
  =
  \sum_{b^{\nu-1}\leq n < b^{\nu+1}}
  \abs{
    \sum_{\substack{m\in I_\mu\\mn \in J_\lambda\\m(n+r)\in J_\lambda}}
    \e(\alpha (R_{\lambda}(m(n+r))- R_{\lambda}(mn)))
  }
  .
\end{equation}

\begin{lemma}[Carry property]\label{lemma:carry-property}
  For any integers $\mu\geq 1$, $\nu\geq 1$,
  $\rho\geq 1$, $\rho'\geq 1$,
  such that $\rho+\rho'<\nu$,
  the set
  $\mathcal{E}_{\mu,\nu,\rho,\rho'}$ of pairs
  \begin{math}
    (m,n) \in I_\mu\times \{b^{\nu-1},\ldots,b^{\nu+1}-1\}
  \end{math}
  such that
  \begin{displaymath}
    \exists (j,r) \in
    \{ \mu+\rho+\rho',\ldots, \mu+\nu-1 \} \times \{ 1,\ldots,b^{\rho}-1 \},
    \
    \varepsilon_j(mn+mr) \neq \varepsilon_j(mn)
  \end{displaymath}
  is of size
  \begin{math}
    \abs{\mathcal{E}_{\mu,\nu,\rho,\rho'}} = O(b^{\mu+\nu-\rho'}).  
  \end{math}
\end{lemma}
\begin{proof}
  We proceed as in \cite[section~4]{mauduit-rivat-RS-primes}.
  If $(m,n)\in \mathcal{E}_{\mu,\nu,\rho,\rho'}$
  then $\varepsilon_j (mn)=b-1$
  for all $j\in\{ \mu +\rho,\ldots ,\mu+\rho +\rho'-1\}$.
  Thus we can find integers $a,a'$ such that $mn=a+b^{\mu+\rho} a'$ with
  $0\le a <b^{\mu+\rho}$, $a'\in\mathcal{B}$ where $\mathcal{B}$ is
  the set of the integers $a'$ with $0\leq a'< b^{\nu -\rho +1}$ such that
  $\varepsilon_j (a')=b-1$ for all $0\le j <\rho'$.
  Since the elements of $\mathcal{B}$ have $\rho'$ fixed digits,
  $|\mathcal{B}| = b^{\nu-\rho-\rho'+1}$.
  We can thus bound $\abs{\mathcal{E}_{\mu,\nu,\rho,\rho'}}$ as follows
  \begin{align*}
    \abs{\mathcal{E}_{\mu,\nu,\rho,\rho'}}
    &\leq
    \sum_{b^{\mu -1}\le m < b^\mu}\sum_{a'\in \mathcal{B}}
    \sum_{\substack{0\leq a < b^{\mu +\rho}\\ a+b^{\mu+\rho}a'\equiv 0\bmod m}} 1
    \\
    &\ll 
    \sum_{b^{\mu -1}\le m < b^\mu}
    \frac{b^{\mu +\rho}\abs{\mathcal{B}}}{m}\ll b^{\mu+\nu -\rho'}
    ,
  \end{align*}
  as expected.
\end{proof}

Let us first bound $T_1(\alpha,\lambda,\mu,J_\lambda,r)$.
Let $\rho'$ be a (small) positive integer
such that
\begin{equation}
  \label{eq:condition-rho-rho'}
  \rho + \rho' < \nu
  ,
\end{equation}
and let
\begin{equation}\label{eq:definition-mu2}
  \mu_2=\mu + \rho + \rho'.
\end{equation}
Using
\eqref{eq:reverse-split} with $\lambda'=\mu_2$
and $\mathcal{E}_{\mu,\nu,\rho,\rho'}$ as defined in
Lemma~\ref{lemma:carry-property},
it follows that for
\begin{math}
  (m,n) \in (I_\mu\times \{b^{\nu-1},\ldots,b^{\nu+1}-1\})
  \setminus\mathcal{E}_{\mu,\nu,\rho,\rho'}
\end{math}
we have
\begin{displaymath}
  R_\lambda(mn)
  =
  b^{\lambda-\mu_2}  R_{\mu_2}(mn)
  +
  \sum_{j=\mu_2}^{\lambda-1} \varepsilon_j(mn) b^{\lambda-1-j}
  ,
\end{displaymath}
and
\begin{displaymath}
  R_\lambda(mn+mr)
  =
  b^{\lambda-\mu_2}  R_{\mu_2}(mn+mr)
  +
  \sum_{j=\mu_2}^{\lambda-1} \varepsilon_j(mn) b^{\lambda-1-j}
  ,
\end{displaymath}
hence
\begin{displaymath}
  R_\lambda(mn+mr) -   R_\lambda(mn)
  =
  b^{\lambda-\mu_2}
  \left( R_{\mu_2}(mn+mr) - R_{\mu_2}(mn) \right)
  .
\end{displaymath}
We deduce that
\begin{equation}\label{eq:link_T1_T2}
  T_1(\alpha,\lambda,\mu,J_\lambda,r)
  =
  T_2(\alpha,\lambda,\mu,\mu_2,J_\lambda,r)
  + O\left( b^{\lambda-\rho'} \right)
  ,
\end{equation}
with
\begin{multline}\label{eq:definition-T2}
  T_2(\alpha,\lambda,\mu,\mu_2,J_\lambda,r)
  \\
  =
  \sum_{b^{\nu-1}\leq n < b^{\nu+1}}
  \abs{
    \sum_{\substack{m\in I_\mu\\mn \in J_\lambda\\m(n+r)\in J_\lambda}}
    \e\left(\alpha b^{\lambda-\mu_2} (R_{\mu_2}(m(n+r))- R_{\mu_2}(mn))\right)
  }.
\end{multline}
It follows from \eqref{eq:maj_SIIcarre_T} that
\begin{equation}\label{eq:maj_SIIcarre_T2}
  S_{II}(\alpha,\lambda,\mu,J_{\lambda},\mathbf{z})^2
  \ll
  \frac{b^{2\lambda}}{b^{\rho}}
  +
  \frac{b^{2\lambda}}{b^{\rho'}}
  +
  b^{\lambda-\rho}
  \sum_{r=1}^{b^{\rho}-1} 
  T_2(\alpha,\lambda,\mu,\mu_2,J_\lambda,r)
  .
\end{equation}

\begin{remark}
  The use of \eqref{eq:maj_SIIcarre_T2} in order to estimate
  \eqref{eq:initial-sum-to-estimate-from-typeII-sums}
  will produce an error term
  \begin{displaymath}
    \left(
      \frac{b^{\lambda}}{b^{\rho/2}}
      +
      \frac{b^{\lambda}}{b^{\rho'/2}}
    \right)
    \sum_{\alpha\in\A_D} W(\alpha)
    \gg
    D
    \left(
      \frac{b^{\lambda}}{b^{\rho/2}}
      +
      \frac{b^{\lambda}}{b^{\rho'/2}}
    \right)
    ,
  \end{displaymath}
  which limits the size of $D$ to
  \begin{displaymath}
    D \ll b^{\frac12\min(\rho,\rho')}
    .
  \end{displaymath}
\end{remark}

\subsection{Toward exponential sums}
Let us focus on the inner term of~\eqref{eq:definition-T2}.
Using the periodicity of $R_{\mu_2}(\cdot)$, we have
\begin{multline*}
  \e\left(\alpha   b^{\lambda-\mu_2}
  \left( R_{\mu_2}(m(n+r)) - R_{\mu_2}(mn) \right)\right)
  \\
  =
  \frac{1}{b^{2\mu_2}}
  \sum_{\substack{0\leq k_1< b^{\mu_2}\\ 0\leq h_1< b^{\mu_2}}}
  \sum_{\substack{0\leq k_2< b^{\mu_2}\\ 0\leq h_2< b^{\mu_2}}}
  \e(\alpha b^{\lambda-\mu_2} (R_{\mu_2}(k_2)- R_{\mu_2}(k_1)))
  \\
  \e\left(\frac{-h_2(k_2-m(n+r)) + h_1(k_1-mn)}{b^{\mu_2}}\right)
  .
\end{multline*}
By \eqref{eq:definition-F} this rewrites as
\begin{multline*}
  \e\left(
  \alpha b^{\lambda-\mu_2} (R_{\mu_2}(m(n+r))- R_{\mu_2}(mn))
  \right)
  =
  \\
  \sum_{\substack{0\leq h_1< b^{\mu_2}\\ 0\leq h_2< b^{\mu_2}}}
  \!
  \mathcal{F}_{\mu_2}\left(
    \alpha b^{\lambda-\mu_2},\frac{h_2}{b^{\mu_2}}\right)
  \conjugate{\mathcal{F}_{\mu_2}\left(
      \alpha b^{\lambda-\mu_2},\frac{h_1}{b^{\mu_2}}\right)}
  \e\left(\frac{h_2m(n+r) - h_1mn}{b^{\mu_2}}\right),  
\end{multline*}
which leads to
\begin{multline}\label{eq:Fourier-analysis}
  \e\left(
  \alpha b^{\lambda-\mu_2} (R_{\mu_2}(m(n+r))- R_{\mu_2}(mn))
  \right)
  =
  \\
  \sum_{\substack{0\leq h_3< b^{\mu_2}\\0\leq h_2< b^{\mu_2}}}
  \mathcal{F}_{\mu_2}\left(
    \alpha b^{\lambda-\mu_2},\frac{h_2}{b^{\mu_2}}\right)
  \conjugate{\mathcal{F}_{\mu_2}\left(
      \alpha b^{\lambda-\mu_2},\frac{h_2-h_3}{b^{\mu_2}}\right)}
  \e\left(\frac{h_3 m n + h_2 m r}{b^{\mu_2}}\right)
  .
\end{multline}
Summing over $m$, we have to bound above a geometric sum over $m$
and we get
\begin{multline*}
  \abs{
    \sum_{\substack{m\in I_\mu\\mn \in J_\lambda,\ m(n+r)\in J_\lambda}}
    \e\left(
    \alpha b^{\lambda-\mu_2} (R_{\mu_2}(m(n+r))- R_{\mu_2}(mn))
    \right)
    }
  \\
  \leq
  \sum_{0\leq h_3< b^{\mu_2}}
  \sum_{0\leq h_2< b^{\mu_2}}
  \abs{
    \mathcal{F}_{\mu_2}\left(
      \alpha b^{\lambda-\mu_2},\frac{h_2}{b^{\mu_2}}\right)
    \mathcal{F}_{\mu_2}\left(
      \alpha b^{\lambda-\mu_2},\frac{h_2-h_3}{b^{\mu_2}}\right)
  }
  \\
  \min\left(
    b^{\mu},
    \abs{\sin\pi\left(\frac{h_3 n + h_2 r}{b^{\mu_2}}\right)}^{-1}
  \right).
\end{multline*}
By Lemma~\ref{lemma:sum_inverse_sinus}, writing
$\gcd(h_3,b^{\mu_2}) = d b^\delta$ with $0\leq \delta \leq \mu_2$ and
$d\dv b^{\mu_2-\delta}$, $\gcd(d,b)<b$ (notice that if $b$ is a prime
number then $d=1$), we have
\begin{multline*}
  \sum_{b^{\nu-1}\leq n < b^{\nu+1}}
  \min\left(
    b^{\mu},
    \abs{\sin\pi\left(\frac{h_3 n + h_2 r}{b^{\mu_2}}\right)}^{-1}
  \right)
  \ll
  \\
  \left(d^{-1}b^{-\delta}+b^{\nu-\mu_2}\right)
  \left(
    d b^\delta
    \min\left(
    b^{\mu},
    \abs{
      \sin\left(
        \frac{\pi d b^\delta}{b^{\mu_2}}
        \norm{\frac{h_2 r}{d b^\delta}}\right)}^{-1}
  \right)
    + b^{\mu_2} \log b^{\mu_2}\right).
\end{multline*}
It follows from \eqref{eq:definition-T2} that
\begin{multline*}
  T_2(\alpha,\lambda,\mu,\mu_2,J_\lambda,r)
  \\
  \ll
  \sum_{0\leq h_2< b^{\mu_2}}
  \abs{
    \mathcal{F}_{\mu_2}\left(\alpha b^{\lambda-\mu_2},\frac{h_2}{b^{\mu_2}}\right)
  }
  \sum_{0\leq \delta \leq \mu_2}
  \sum_{\substack{d \dv  b^{\mu_2-\delta}\\ \gcd(d,b)<b}}
  \left(d^{-1}b^{-\delta}+b^{\nu-\mu_2}\right)
  \\
  \left(
    d b^\delta
    \min\left(
      b^{\mu},
      \abs{
        \sin\left(
          \frac{\pi d}{b^{\mu_2-\delta}}\norm{\frac{h_2 r}{d b^\delta}}
        \right)}^{-1}
    \right)
    + b^{\mu_2} \log b^{\mu_2}
  \right)
  \\
  \sum_{\substack{0\leq h_3< b^{\mu_2}\\ \gcd(h_3,b^{\mu_2})=d b^\delta}}
  \abs{
    \mathcal{F}_{\mu_2}\left(
      \alpha b^{\lambda-\mu_2},\frac{h_2-h_3}{b^{\mu_2}}
    \right)
  }.
\end{multline*}
If $h_2 \equiv a \bmod d b^\delta$, by adding some terms we have
\begin{displaymath}
  \sum_{\substack{0\leq h_3< b^{\mu_2}\\ \gcd(h_3,b^{\mu_2})=d b^\delta}}
  \abs{
    \mathcal{F}_{\mu_2}\left(
      \alpha b^{\lambda-\mu_2},\frac{h_2-h_3}{b^{\mu_2}}
    \right)
  }
  \leq
  \sum_{\substack{0\leq h< b^{\mu_2}\\ h \equiv a \bmod d b^\delta}}
  \abs{
    \mathcal{F}_{\mu_2}\left(
      \alpha b^{\lambda-\mu_2},\frac{h}{b^{\mu_2}}
    \right)
  }
  .
\end{displaymath}
Filtering modulo $d b^\delta$ we get
\begin{multline*}
  T_2(\alpha,\lambda,\mu,\mu_2,J_\lambda,r)
  \ll
  \sum_{0\leq \delta \leq \mu_2}
  \sum_{\substack{d \dv  b^{\mu_2-\delta}\\ \gcd(d,b)<b}}
  \left(d^{-1}b^{-\delta}+b^{\nu-\mu_2}\right)
  \\
  \sum_{0\leq a < d b^\delta}
  \left(
    d b^\delta
    \min\left(
      b^{\mu},
      \abs{
        \sin\left(
          \frac{\pi d}{b^{\mu_2-\delta}}\norm{\frac{a r}{d b^\delta}}
        \right)}^{-1}
    \right)
    + b^{\mu_2} \log b^{\mu_2}
  \right)
  \\
  \left(
    \sum_{\substack{0\leq h< b^{\mu_2}\\ h \equiv a \bmod d b^\delta}}
    \abs{
      \mathcal{F}_{\mu_2}\left(
        \alpha b^{\lambda-\mu_2},\frac{h}{b^{\mu_2}}
      \right)
    }
  \right)^2.
\end{multline*}
By Lemma~\ref{lemma:L1-norm-modular-via-sobolev-gallagher}, this leads to
\begin{multline*}
  T_2(\alpha,\lambda,\mu,\mu_2,J_\lambda,r)
  \ll
  \sum_{0\leq \delta \leq \mu_2}
  \sum_{\substack{d \dv  b^{\mu_2-\delta}\\ \gcd(d,b)<b}}
  \left(d^{-1}b^{-\delta}+b^{\nu-\mu_2}\right)
  \\
  \sum_{0\leq a < d b^\delta}
  \left(
    d b^\delta
    \min\left(
      b^{\mu},
      \abs{
        \sin\left(
          \frac{\pi d}{b^{\mu_2-\delta}}\norm{\frac{a r}{d b^\delta}}
        \right)}^{-1}
    \right)
    + b^{\mu_2} \log b^{\mu_2}
  \right)
  \\
  \frac{b^{2\eta_b(\mu_2-\delta)}}{d^{2\eta_b}}
  \abs{
    \mathcal{F}_{\delta}\left(
      \alpha b^{\lambda-\mu_2},\frac{a}{b^{\delta}}
    \right)
  }^2.
\end{multline*}
Writing $a = a' + k b^\delta$ we get
\begin{multline*}
  T_2(\alpha,\lambda,\mu,\mu_2,J_\lambda,r)
  \ll
  \sum_{0\leq \delta \leq \mu_2}
  \sum_{\substack{d \dv  b^{\mu_2-\delta}\\ \gcd(d,b)<b}}
  \left(d^{-1}b^{-\delta}+b^{\nu-\mu_2}\right)
  \\
  \sum_{0\leq a' < b^\delta}
  \sum_{0\leq k < d}
  \left(
    d b^\delta
    \min\left(
      b^{\mu},
      \abs{
        \sin\left(
          \frac{\pi d}{b^{\mu_2-\delta}}
          \norm{\frac{a' r}{d b^\delta}+\frac{k r}{d}}
        \right)}^{-1}
    \right)
    + b^{\mu_2} \log b^{\mu_2}
  \right)
  \\
  \frac{b^{2\eta_b(\mu_2-\delta)}}{d^{2\eta_b}}
  \abs{
    \mathcal{F}_{\delta}\left(
      \alpha b^{\lambda-\mu_2},\frac{a'}{b^{\delta}}
    \right)
  }^2.
\end{multline*}
The function sinus is concave over $[0,\pi]$, hence for $t\in\R$
\begin{displaymath}
  \sin\left(
    \frac{\pi d}{b^{\mu_2-\delta}}
    \norm{\frac{k r + t}{d}}
  \right)
  \geq
  \frac{d}{b^{\mu_2-\delta}}
  \sin\left(
    \pi \norm{\frac{k r + t}{d}}
  \right)
  =
  \frac{d}{b^{\mu_2-\delta}}
  \abs{
    \sin \left( \pi \frac{k r + t}{d} \right)
  }
  ,
\end{displaymath}
and
\begin{multline*}
  \sum_{0\leq k < d}
  \left(
    d b^\delta
    \min\left(
      b^{\mu},
      \abs{
        \sin\left(
          \frac{\pi d}{b^{\mu_2-\delta}}
          \norm{\frac{k r + t}{d}}
        \right)}^{-1}
    \right)
    + b^{\mu_2} \log b^{\mu_2}
  \right)
  \\
  \leq
  \sum_{0\leq k < d}
    d b^\delta
    \min\left(
      b^{\mu},
      \frac{b^{\mu_2-\delta}}{d}
      \abs{
        \sin\pi\left( \frac{k r + t}{d} \right)}^{-1}
    \right)
  + d b^{\mu_2} \log b^{\mu_2}
  \\
  \leq
  b^{\mu_2}
  \sum_{0\leq k < d}
  \min\left(
    \frac{d}{b^{\mu_2-\delta}} b^{\mu},
    \abs{
      \sin\pi\left( \frac{k r + t}{d} \right)}^{-1}
  \right)
  + d b^{\mu_2} \log b^{\mu_2}
  .
\end{multline*}
By Lemma~\ref{lemma:sum_inverse_sinus}, this is
\begin{multline*}
  \ll
  b^{\mu_2}
  \gcd(r,d)
  \min\left(
    \frac{d}{b^{\mu_2-\delta}} b^{\mu},
    \abs{
      \sin\pi\left( \frac{\gcd(r,d)}{d} \norm{\frac{t}{\gcd(r,d)}} \right)
    }^{-1}
  \right)
  \\
  +
  b^{\mu_2}
  d \log d
  + d b^{\mu_2} \log b^{\mu_2}
\end{multline*}
which is
\begin{displaymath}
  \ll
  b^{\mu_2}
  d
  \min\left(
    \frac{\gcd(r,d) b^{\mu}}{b^{\mu_2-\delta}} ,
    \abs{
      \sin\left( \pi \norm{\frac{t}{\gcd(r,d)}} \right)
    }^{-1}
  \right)
  +
  b^{\mu_2}
  d \log d
  + d b^{\mu_2} \log b^{\mu_2}.
\end{displaymath}
It follows, taking $t=a'r/b^\delta$,
\begin{multline*}
  T_2(\alpha,\lambda,\mu,\mu_2,J_\lambda,r)
  \\
  \ll
  \sum_{0\leq \delta \leq \mu_2} b^{\mu_2+2\eta_b(\mu_2-\delta)}
  \sum_{\substack{d \dv  b^{\mu_2-\delta}\\ \gcd(d,b)<b}} d^{1-2\eta_b}
  \left(d^{-1}b^{-\delta}+b^{\nu-\mu_2}\right)
  \\
  \sum_{0\leq a' < b^\delta}
  \left(
    \min\left(
      \frac{\gcd(r,d) b^{\mu}}{b^{\mu_2-\delta}} ,
      \abs{
        \sin\left( \pi \norm{\frac{a'r}{\gcd(r,d) b^\delta}} \right)
      }^{-1}
    \right)
    +  \log b^{\mu_2}
  \right)
  \\
  \abs{
    \mathcal{F}_{\delta}\left(
      \alpha b^{\lambda-\mu_2},\frac{a'}{b^{\delta}}
    \right)
  }^2
  .
\end{multline*}
Since
\begin{displaymath}
  \sum_{0\leq a' < b^\delta}
  \abs{
    \mathcal{F}_{\delta}\left(
      \alpha b^{\lambda-\mu_2},\frac{a'}{b^{\delta}}
    \right)
  }^2
  = 1,
\end{displaymath}
isolating the error term produced by $\log b^{\mu_2}$
and splitting $  \left(d^{-1}b^{-\delta}+b^{\nu-\mu_2}\right)$,
we get
\begin{multline}\label{eq:estimate-T2-T3-initial}
  T_2(\alpha,\lambda,\mu,\mu_2,J_\lambda,r)
  \ll
  T_3(\alpha,\lambda,\mu,\mu_2,r)
  +
  T_4(\alpha,\lambda,\mu,\mu_2,r,0,\mu_2)\\
  +
  \left( \log b^{\mu_2} \right)
  \sum_{0\leq \delta \leq \mu_2} b^{\mu_2+2\eta_b(\mu_2-\delta)}
  \sum_{\substack{d \dv  b^{\mu_2-\delta}\\ \gcd(d,b)<b}} d^{1-2\eta_b}
  \left(d^{-1}b^{-\delta}+b^{\nu-\mu_2}\right)
\end{multline}
with
\begin{multline}\label{eq:definition-T3}
  T_3(\alpha,\lambda,\mu,\mu_2,r)
  =
  \sum_{0\leq \delta \leq \mu_2}
  b^{\mu_2+2\eta_b(\mu_2-\delta)-\delta}
  \sum_{\substack{d \dv  b^{\mu_2-\delta}\\ \gcd(d,b)<b}}
  d^{-2\eta_b}
  \\
  \sum_{0\leq a' < b^\delta}
  \min\left(
    \gcd(r,d) b^{\mu+\delta-\mu_2} ,
    \abs{
      \sin\left( \pi \norm{\frac{a'r}{\gcd(r,d) b^\delta}} \right)
    }^{-1}
  \right)
  \\
  \abs{
    \mathcal{F}_{\delta}\left(
      \alpha b^{\lambda-\mu_2},\frac{a'}{b^{\delta}}
    \right)
  }^2
  .
\end{multline}
and
\begin{multline}\label{eq:definition-T4}
  T_4(\alpha,\lambda,\mu,\mu_2,r,\delta_{-},\delta_{+})
  =
  \sum_{\delta_{-}\leq \delta \leq \delta_{+}}
  b^{\nu+2\eta_b(\mu_2-\delta)}  
  \sum_{\substack{d \dv  b^{\mu_2-\delta}\\ \gcd(d,b)<b}} d^{1-2\eta_b}
  \\
  \sum_{0\leq a' < b^\delta}
  \min\left(
    \gcd(r,d) b^{\mu+\delta-\mu_2} ,
    \abs{
      \sin\left( \pi \norm{\frac{a'r}{\gcd(r,d) b^\delta}} \right)
    }^{-1}
  \right)
  \\
  \abs{
    \mathcal{F}_{\delta}\left(
      \alpha b^{\lambda-\mu_2},\frac{a'}{b^{\delta}}
    \right)
  }^2
  .
\end{multline}

\subsection{The contribution of the error term in \eqref{eq:estimate-T2-T3-initial}}

Since $-2\eta_b <0$, by \eqref{eq:sigma-z<0-b-lambda} we have
the following upper bound, independent of $\mu_2-\delta$:
\begin{equation}\label{eq:maj_sum_d^-2eta}
  \sum_{\substack{d \dv  b^{\mu_2-\delta}\\ \gcd(d,b)<b}}
  d^{-2\eta_b} 
  \leq \prod_{p \dv b} \frac{1}{1-p^{-2\eta_b}}
  .
\end{equation}
It follows that the contribution of the term involving $d^{-1}b^{-\delta}$
in \eqref{eq:estimate-T2-T3-initial}
is 
\begin{displaymath}
  \sum_{0\leq \delta \leq \mu_2} b^{\mu_2+2\eta_b(\mu_2-\delta)}
  \sum_{\substack{d \dv  b^{\mu_2-\delta}\\ \gcd(d,b)<b}} d^{1-2\eta_b}
  d^{-1}b^{-\delta}
  \ll_b
  b^{\mu_2(1+2\eta_b)}
  .
\end{displaymath}

If $d \dv b^{\mu_2-\delta}$ and $\gcd(d,b)<b$ then for any prime
number $p$ we have
\begin{align*}
  v_p(\gcd(d,b^{\mu_2-\delta}))
  &=
  \min\left( v_p(d), (\mu_2-\delta) v_p(b)\right)
  \\
  &\leq
  (\mu_2-\delta)
  \min\left( v_p(d), v_p(b)\right)
  =
  v_p(\gcd(d,b)^{\mu_2-\delta})
\end{align*}
hence
\begin{displaymath}
  d = \gcd(d,b^{\mu_2-\delta})
  \leq
  \gcd(d,b)^{\mu_2-\delta}
  \leq \left(\frac{b}{P^-(b)}\right)^{\mu_2-\delta}
\end{displaymath}
and
\begin{align}\label{eq:maj_sum_d^1-2eta}
  \sum_{\substack{d \dv  b^{\mu_2-\delta}\\ \gcd(d,b)<b}}
  d^{1-2\eta_b} 
  &\leq
  \left(\frac{b}{P^-(b)}\right)^{(1-2\eta_b)(\mu_2-\delta)}
  \sum_{\substack{d \dv  b^{\mu_2-\delta}\\ \gcd(d,b)<b}} 1
  \\
  \nonumber
  &\leq
  \left(\frac{b}{P^-(b)}\right)^{(1-2\eta_b)(\mu_2-\delta)}
  \tau(b^{\mu_2-\delta})
  .
\end{align}
The contribution of the term involving $b^{\nu-\mu_2}$
in \eqref{eq:estimate-T2-T3-initial}
is bounded by
\begin{multline*}
  b^{\nu-\mu_2}
  \sum_{0\leq \delta \leq \mu_2}
  b^{\mu_2+2\eta_b(\mu_2-\delta)}
  \sum_{\substack{d \dv  b^{\mu_2-\delta}\\ \gcd(d,b)<b}} d^{1-2\eta_b}
  \\
  \ll
  b^{\nu-\mu_2}
  \sum_{0\leq \delta \leq \mu_2}
  b^{\mu_2+2\eta_b(\mu_2-\delta)}
  \left(\frac{b}{P^-(b)}\right)^{(1-2\eta_b)(\mu_2-\delta)}
  \tau(b^{\mu_2-\delta}) 
  \\
  \ll
  b^{\nu}
  \tau(b^{\mu_2})
  \sum_{0\leq \delta \leq \mu_2}
  \left(\frac{b}{P^-(b)^{(1-2\eta_b)}}\right)^{\mu_2-\delta}
  \ll
  \frac{b^{\nu+\mu_2}}{P^-(b)^{(1-2\eta_b)\mu_2}}
  \tau(b^{\mu_2}) 
  .
\end{multline*}

Bounding by $\gcd(r,d) b^{\mu+\delta-\mu_2}$
the minimum in the definition of $T_3$
we get
\begin{multline*}
  T_3(\alpha,\lambda,\mu,\mu_2,r)
  \ll
  \sum_{0\leq \delta \leq \mu_2} b^{\mu_2+2\eta_b(\mu_2-\delta)-\delta}
  \sum_{\substack{d \dv b^{\mu_2-\delta}\\ \gcd(d,b)<b}}
  d^{-2\eta_b}
  \\
  \sum_{0\leq a' < b^\delta}
  \gcd(r,d) b^{\mu+\delta-\mu_2}
  \abs{ \mathcal{F}_{\delta}\left( \alpha
      b^{\lambda-\mu_2},\frac{a'}{b^{\delta}} \right) }^2
  ,
\end{multline*}
which, by Lemma~\ref{lemma:L2-mean-value}
and Lemma~\ref{lemma:gcd-sum}
gives
\begin{displaymath}
  \frac{1}{b^\rho} \sum_{r=1}^{b^{\rho}-1}
  T_3(\alpha,\lambda,\mu,\mu_2,r)
  \ll
  \sum_{0\leq \delta \leq \mu_2}
  b^{\mu+2\eta_b(\mu_2-\delta)}
  \sum_{\substack{d \dv b^{\mu_2-\delta}\\ \gcd(d,b)<b}}
  d^{-2\eta_b}
  \tau(d)
  .
\end{displaymath}
Using \eqref{eq:tau-sigma-z-b-lambda} we get
\begin{displaymath}
  \frac{1}{b^\rho} \sum_{r=1}^{b^{\rho}-1}
  T_3(\alpha,\lambda,\mu,\mu_2,r)
  \ll_b
  b^{\mu+2\eta_b\mu_2}
  .
\end{displaymath}

Since $\mu+2\eta_b\mu_2 \leq (1+2\eta_b)\mu_2$, we deduce that
\begin{multline}\label{eq:estimate-T2-T3-final}
  \frac{1}{b^\rho} \sum_{r=1}^{b^{\rho}-1}
  T_2(\alpha,\lambda,\mu,\mu_2,J_\lambda,r)
  \\
  \ll
  \frac{1}{b^\rho} \sum_{r=1}^{b^{\rho}-1}
  T_4(\alpha,\lambda,\mu,\mu_2, r,0,\mu_2)
  \qquad
  \\
  +
  b^{\mu_2(1+2\eta_b)} \log b^{\mu_2}
  +
  \frac{b^{\nu+\mu_2}}{P^-(b)^{(1-2\eta_b)\mu_2}}
  \tau(b^{\mu_2}) \log b^{\mu_2}
  .
\end{multline}

\subsection{The contribution of  $T_4$ for small $\delta$}
Let $\delta_0$ be an integer to be chosen later such that
\begin{equation}
  \label{eq:condition-delta0}
  \rho+\rho' \leq \delta_0 < \mu_2
  .
\end{equation}
Bounding by $\gcd(r,d) b^{\mu+\delta-\mu_2}$
the minimum in the definition of $T_4$
given by \eqref{eq:definition-T4}
we get
\begin{multline*}
  T_4(\alpha,\lambda,\mu,\mu_2, r,0,\delta_0)
  \ll
  \sum_{0\leq \delta \leq \delta_0}
  b^{\nu+2\eta_b(\mu_2-\delta)}
  \sum_{\substack{d \dv b^{\mu_2-\delta}\\ \gcd(d,b)<b}}
  d^{1-2\eta_b}
  \\
  \sum_{0\leq a' < b^\delta}
  \gcd(r,d) b^{\mu+\delta-\mu_2}
  \abs{ \mathcal{F}_{\delta}\left( \alpha
      b^{\lambda-\mu_2},\frac{a'}{b^{\delta}} \right) }^2
  ,
\end{multline*}
which, by Lemma~\ref{lemma:L2-mean-value}
and Lemma~\ref{lemma:gcd-sum}
gives
\begin{displaymath}
  \frac{1}{b^\rho} \sum_{r=1}^{b^{\rho}-1}
  T_4(\alpha,\lambda,\mu,\mu_2, r,0,\delta_0)
  \ll
  \sum_{0\leq \delta \leq \delta_0}
  b^{\lambda+2\eta_b(\mu_2-\delta)+\delta-\mu_2}
  \sum_{\substack{d \dv b^{\mu_2-\delta}\\ \gcd(d,b)<b}}
  d^{1-2\eta_b}
  \tau(d)
  .
\end{displaymath}
Since $d \dv b^{\mu_2}$ we have $\tau(d) \leq \tau(b^{\mu_2})$, hence
by \eqref{eq:maj_sum_d^1-2eta}
we deduce
\begin{multline*}
  \frac{1}{b^\rho} \sum_{r=1}^{b^{\rho}-1}
  T_4(\alpha,\lambda,\mu,\mu_2, r,0,\delta_0)
  \\
  \ll
  \tau\left(b^{\mu_2}\right)
  \sum_{0\leq \delta \leq \delta_0}
  b^{\lambda+2\eta_b(\mu_2-\delta)+\delta-\mu_2}
  \left(\frac{b}{P^-(b)}\right)^{(1-2\eta_b)(\mu_2-\delta)}
  \tau(b^{\mu_2-\delta})
  .
\end{multline*}
which, using $\tau(b^{\mu_2-\delta}) \leq \tau(b^{\mu_2})$,
gives
\begin{multline*}
  \frac{1}{b^\rho} \sum_{r=1}^{b^{\rho}-1}
  T_4(\alpha,\lambda,\mu,\mu_2, r,0,\delta_0)
  \\
  \ll
  \left(\tau\left(b^{\mu_2}\right)\right)^2
  \sum_{0\leq \delta \leq \delta_0}
  \frac{b^{\lambda}}{P^-(b)^{(1-2\eta_b)(\mu_2-\delta)}}
  \ll
  \left(\tau\left(b^{\mu_2}\right)\right)^2
  \frac{b^{\lambda}}{P^-(b)^{(1-2\eta_b)(\mu_2-\delta_0)}}
  .
\end{multline*}

Remembering \eqref{eq:estimate-T2-T3-final}, we obtain
\begin{align*}
  \frac{1}{b^\rho}
  &
    \sum_{r=1}^{b^{\rho}-1}
  T_2(\alpha,\lambda,\mu,\mu_2,J_\lambda,r)
  \\
  &
    \ll
    b^{\mu_2(1+2\eta_b)} \log b^{\mu_2}
  +
  \frac{b^{\nu+\mu_2}}{P^-(b)^{(1-2\eta_b)\mu_2}}
  \tau(b^{\mu_2}) \log b^{\mu_2}
  \\
  & \quad +
  \left(\tau\left(b^{\mu_2}\right)\right)^2
  \frac{b^\lambda}{P^-(b)^{(1-2\eta_b)(\mu_2-\delta_0)}}
  +
  \frac{1}{b^\rho} \sum_{r=1}^{b^{\rho}-1}
  T_4(\alpha,\lambda,\mu,\mu_2, r,\delta_0+1,\mu_2)
  .
\end{align*}
Therefore by \eqref{eq:maj_SIIcarre_T2}, we get
\begin{multline}\label{eq:maj_SIIcarre_end}
  S_{II}(\alpha,\lambda,\mu)^2
  \\
  \ll
  b^{2\lambda-\rho}
  +
  b^{2\lambda-\rho'}
  +
  b^{\lambda}
  b^{\mu_2(1+2\eta_b)} \log b^{\mu_2}
  +
  \frac{b^{2\lambda+\rho+\rho'}}{P^-(b)^{(1-2\eta_b)\mu_2}}
  \tau(b^{\mu_2}) \log b^{\mu_2}
  \\
  +
  \left(\tau\left(b^{\mu_2}\right)\right)^2
  \frac{b^{2\lambda}}{P^-(b)^{(1-2\eta_b)(\mu_2-\delta_0)}}
  +
  \frac{b^\lambda}{b^\rho} \sum_{r=1}^{b^{\rho}-1}
  T_4(\alpha,\lambda,\mu,\mu_2, r,\delta_0+1,\mu_2)
  .
\end{multline}

\subsection{The contribution of  $T_4$ for large $\delta$}
This is the most difficult part of the study of the Type II sums.
It will be done in
Section~\ref{section:critical-contribution-of-Type-II-sums}.

\section{The critical contribution of the Type II sums}
\label{section:critical-contribution-of-Type-II-sums}
We first observe that by \eqref{eq:condition-delta0}
and \eqref{eq:definition-mu2},
for $\delta \geq \delta_0+1$
we have $\mu+\delta-\mu_2\geq 1$.

In order to control the size of the minimum appearing in the
definition of $T_4$ given by~\eqref{eq:definition-T4},
for any integers $\delta'\geq 0$ and $r'\geq 1$, we introduce
\begin{displaymath}
  \mathcal{E}(r',\delta,\delta')
  =
  \left\{
  a':\ 0\leq a' < b^\delta,\
  \norm{\frac{a'r'}{b^\delta}}
  \leq \frac{1}{b^{\delta'}}
  \right\}
\end{displaymath}
and
\begin{align*}
  \widetilde{\mathcal{E}}(r',\delta,\delta')
  &
    =
    \mathcal{E}(r',\delta,\delta')
    \setminus
    \mathcal{E}(r',\delta,\delta'+1)
  \\
  &
    =
    \left\{
    a':\ 0\leq a' < b^\delta,\
    \frac{1}{b^{\delta'+1}}
    <
    \norm{\frac{a'r'}{b^\delta}}
    \leq
    \frac{1}{b^{\delta'}}
    \right\}
    .
\end{align*}
For any integers $\delta'_1\geq 0$ and $r'\geq 1$ we have
\begin{displaymath}
  \left\{a':\ 0\leq a' < b^\delta \right\}
  =
  \left(
    \bigcup_{0\leq \delta' < \delta'_1}
    \widetilde{\mathcal{E}}(r',\delta,\delta')
  \right)
  \bigcup
  \mathcal{E}(r',\delta,\delta'_1)
  ,
\end{displaymath}
hence, splitting the summation over $a'$ accordingly with
$\delta'_1=\mu+\delta-\mu_2$, we get
\begin{equation}\label{eq:T4=T5+T6}
  T_4(\alpha,\lambda,\mu,\mu_2, r,\delta_0+1,\mu_2)
  =
  T_5(\alpha,\lambda,\mu,\mu_2, r,\delta_0)
  +
  T_6(\alpha,\lambda,\mu,\mu_2,r,\delta_0)
\end{equation}
where, with the notation $r_d = r/\gcd(r,d)$,
\begin{multline*}
  T_5(\alpha,\lambda,\mu,\mu_2, r,\delta_0)
  =
  \sum_{\delta_0+1\leq \delta \leq \mu_2}
  b^{\nu+2\eta_b(\mu_2-\delta)}  
  \sum_{\substack{d \dv  b^{\mu_2-\delta}\\ \gcd(d,b)<b}} d^{1-2\eta_b}
  \sum_{0\leq \delta' < \mu+\delta-\mu_2}
  \\
  \sum_{a' \in \widetilde{\mathcal{E}}(r_d,\delta,\delta')}
  \hspace{-1em}
  \min\left(
    \gcd(r,d) b^{\mu+\delta-\mu_2} ,
    \abs{
      \sin\left( \pi \norm{\frac{a'r_d}{b^\delta}} \right)
    }^{-1}
  \right)
  \abs{
    \mathcal{F}_{\delta}\left(
      \alpha b^{\lambda-\mu_2},\frac{a'}{b^{\delta}}
    \right)
  }^2
  ,
\end{multline*}
and
\begin{multline}\label{eq:definition-T6}
  T_6(\alpha,\lambda,\mu,\mu_2,r,\delta_0)
  =
  \sum_{\delta_0+1\leq \delta \leq \mu_2}
  b^{\nu+2\eta_b(\mu_2-\delta)}  
  \sum_{\substack{d \dv  b^{\mu_2-\delta}\\ \gcd(d,b)<b}} d^{1-2\eta_b}
  \\
  \sum_{a' \in \mathcal{E}(r_d,\delta,\mu+\delta-\mu_2)}
  \hspace{-2.5em}
  \min\left(
    \gcd(r,d) b^{\mu+\delta-\mu_2} ,
    \abs{
      \sin\left( \pi \norm{\frac{a'r_d}{b^\delta}} \right)
    }^{-1}
  \right)
  \hspace{-.2em}
  \abs{
    \mathcal{F}_{\delta}\left(
      \alpha b^{\lambda-\mu_2},\frac{a'}{b^{\delta}}
    \right)
  }^2
  \hspace{-.5em}
  .
\end{multline}

For
\begin{math}
  a'\in
  \widetilde{\mathcal{E}}(r_d,\delta,\delta')
\end{math}
we have
\begin{displaymath}
  \sin\left( \pi \norm{\frac{a'r_d}{b^\delta}} \right)
  \geq
  2 \norm{\frac{a'r_d}{b^\delta}}
  \geq
  \frac{2}{b^{\delta'+1}}
  ,
\end{displaymath}
hence
\begin{multline*}
  T_5(\alpha,\lambda,\mu,\mu_2, r,\delta_0)
  \\
  \ll
  \sum_{\delta_0+1\leq \delta \leq \mu_2}
  b^{\nu+2\eta_b(\mu_2-\delta)}  
  \sum_{\substack{d \dv  b^{\mu_2-\delta}\\ \gcd(d,b)<b}} d^{1-2\eta_b}
  \\
  \sum_{0\leq \delta' < \mu+\delta-\mu_2}
  b^{\delta'+1}
  \sum_{a' \in \widetilde{\mathcal{E}}(r_d,\delta,\delta')}
  \abs{
    \mathcal{F}_{\delta}\left(
      \alpha b^{\lambda-\mu_2},\frac{a'}{b^{\delta}}
    \right)
  }^2
  .
\end{multline*}
Let
\begin{equation}\label{eq:definition-rho2}
  \rho_2 = \min(\rho,\rho').
\end{equation}
Using Lemma~\ref{lemma:L2-mean-value},
the contribution to $T_5(\alpha,\lambda,\mu,\mu_2, r,\delta_0)$
of the terms with $0\leq \delta' < \mu+\delta-\mu_2-\rho_2$
is 
\begin{displaymath}
  \ll
  \sum_{\delta_0+1\leq \delta \leq \mu_2}
  b^{\nu+2\eta_b(\mu_2-\delta)} \
  b^{\mu+\delta-\mu_2-\rho_2} \
  \sigma_{1-2\eta_b}\left(b^{\mu_2-\delta}\right)
  ,
\end{displaymath}
which by \eqref{eq:sigma-z>0-b-lambda} is
\begin{displaymath}
  \ll_b
  \sum_{\delta_0+1\leq \delta \leq \mu_2}
  b^{\nu+2\eta_b(\mu_2-\delta)}  
  b^{\mu+\delta-\mu_2-\rho_2}
  b^{(1-2\eta_b)(\mu_2-\delta)}
  \ll
  \mu_2
  b^{\lambda-\rho_2}
  ,
\end{displaymath}
so that,
since
\begin{math}
  \widetilde{\mathcal{E}}(r_d,\delta,\delta')
  \subseteq
  \mathcal{E}(r_d,\delta,\delta')
  ,
\end{math}
extending the summation over $a'$ to
$\mathcal{E}(r_d,\delta,\delta')$ for the remaining $\delta'$,
it follows that:
\begin{multline}\label{eq:majoration-T5}
  T_5(\alpha,\lambda,\mu,\mu_2, r,\delta_0)
  \\
  \ll
  \sum_{\delta_0+1\leq \delta \leq \mu_2}
  b^{\nu+2\eta_b(\mu_2-\delta)}  
  \sum_{\substack{d \dv  b^{\mu_2-\delta}\\ \gcd(d,b)<b}} d^{1-2\eta_b}
  \qquad
  \\
  \sum_{\delta'= \mu+\delta-\mu_2-\rho_2} ^{\mu+\delta-\mu_2-1}
  b^{\delta'}
  \sum_{a' \in \mathcal{E}(r_d,\delta,\delta')}
  \abs{
    \mathcal{F}_{\delta}\left(
      \alpha b^{\lambda-\mu_2},\frac{a'}{b^{\delta}}
    \right)
  }^2
  +
  \mu_2 b^{\lambda-\rho_2}
  .
\end{multline}

Bounding by $\gcd(r,d) b^{\mu+\delta-\mu_2}$
the minimum in the definition of $T_6$
given by \eqref{eq:definition-T6}
we get
\begin{multline}\label{eq:majoration-T6}
  T_6(\alpha,\lambda,\mu,\mu_2,r,\delta_0)
  \\
  \ll
  \sum_{\delta_0+1\leq \delta \leq \mu_2}
  b^{\nu+2\eta_b(\mu_2-\delta)}  
  \sum_{\substack{d \dv  b^{\mu_2-\delta}\\ \gcd(d,b)<b}} d^{1-2\eta_b}
  \gcd(r,d)\ b^{\mu+\delta-\mu_2}
  \\
  \sum_{a' \in \mathcal{E}(r_d,\delta,\mu+\delta-\mu_2)}
  \abs{
    \mathcal{F}_{\delta}\left(
      \alpha b^{\lambda-\mu_2},\frac{a'}{b^{\delta}}
    \right)
  }^2
  .
\end{multline}
In order to gather the summations in \eqref{eq:majoration-T5}
and \eqref{eq:majoration-T6} into a single sum $T_7$,
we introduce a factor $\gcd(r,d)\geq 1$
in the summation of \eqref{eq:majoration-T5}.
By \eqref{eq:T4=T5+T6},
it follows that
\begin{equation}\label{eq:T4-T7}
  T_4(\alpha,\lambda,\mu,\mu_2, r,\delta_0+1,\mu_2)
  \ll
  T_7(\alpha,\lambda,\mu,\rho,\rho',r,\delta_0)
  +
  \mu_2 b^{\lambda-\rho_2}
\end{equation}
with, using~\eqref{eq:definition-mu2} and \eqref{eq:definition-rho2},
\begin{multline*}
  T_7(\alpha,\lambda,\mu,\rho,\rho',r,\delta_0)
  \\
  =
  \sum_{\delta_0+1\leq \delta \leq \mu_2}
  b^{\nu+2\eta_b(\mu_2-\delta)}  
  \sum_{\substack{d \dv  b^{\mu_2-\delta}\\ \gcd(d,b)<b}} d^{1-2\eta_b}
  \gcd(r,d)
  \\
  \sum_{\delta'=\mu+\delta-\mu_2-\rho_2}^{\mu+\delta-\mu_2}
  b^{\delta'}
  \sum_{a' \in \mathcal{E}(r_d,\delta,\delta')}
  \abs{
    \mathcal{F}_{\delta}\left(
      \alpha b^{\lambda-\mu_2},\frac{a'}{b^{\delta}}
    \right)
  }^2
  .
\end{multline*}

For any integers $r'\geq 1$ and $\delta'\geq 2$,
since $0<b^{-\delta'}<\frac12$,
we have
\begin{displaymath}
  \mathcal{E}(r',\delta,\delta')
  =
  \bigcup_{j=0}^{r'-1}
  \mathcal{E}_j(r',\delta,\delta')
\end{displaymath}
where, for $j \in\{1,\ldots,r'-1\}$,
\begin{displaymath}
  \mathcal{E}_j(r',\delta,\delta')
  =
  \left\{
    a':\ 0\leq a' < b^\delta,\
    \abs{\frac{a'r'}{b^\delta} -j}
    \leq
    \frac{1}{b^{\delta'}}
  \right\}
  ,
\end{displaymath}
and
\begin{multline*}
  \mathcal{E}_0(r',\delta,\delta')
  =
  \left\{
    a':\ 0\leq a' < b^\delta,\
    0\leq \frac{a'r'}{b^\delta}
    \leq
    \frac{1}{b^{\delta'}}
  \right\}
  \\
  \bigcup
  \left\{
    a':\ 0\leq a' < b^\delta,\
    r'-\frac{1}{b^{\delta'}}
    \leq
    \frac{a'r'}{b^\delta}
    < r'
  \right\}
  .
\end{multline*}
We have
\begin{displaymath}
  \forall j \in\{0,\ldots,r'-1\},\
  \forall a' \in \mathcal{E}_j(r',\delta,\delta'),\
  \norm{\frac{a'}{b^\delta} - \frac{j}{r'}}
  \leq
  \frac{1}{r' b^{\delta'}}
  ,
\end{displaymath}
hence, by Lemma~\ref{lemma:F2-difference},
for any integer $\delta'' \geq 0$ we have
\begin{multline*}
  \forall \alpha'\in\R,\
  \forall j \in\{0,\ldots,r'-1\},\
  \forall a' \in \mathcal{E}_j(r',\delta,\delta'),\
  \\
  \abs{
    \mathcal{F}_{\delta''}\left(
      \alpha',\frac{a'}{b^{\delta}}
    \right)
  }^2
  \leq
  \abs{
    \mathcal{F}_{\delta''}\left(
      \alpha',\frac{j}{r'}
    \right)
  }^2
  +
  \frac{2\pi}{3} \frac{b^{\delta''}}{r' b^{\delta'}}
  .
\end{multline*}
Applying this with $r'=r_d$ and $\delta''=\delta'-\rho_2$,
using Lemma~\ref{FT-splitting-product},
for $j \in\{0,\ldots,r_d-1\}$ and
$a' \in \mathcal{E}_j(r_d,\delta,\delta')$
we have
\begin{multline*}
  \abs{
    \mathcal{F}_{\delta}\left(
      \alpha b^{\lambda-\mu_2},\frac{a'}{b^{\delta}}
    \right)
  }^2
  \\
  =
    \abs{
    \mathcal{F}_{\delta-\delta'+\rho_2}\left(
      \alpha b^{\lambda-\mu_2},\frac{a'}{b^{\delta-\delta'+\rho_2}}
    \right)
  }^2
  \abs{
    \mathcal{F}_{\delta'-\rho_2}\left(
      \alpha b^{\lambda-\mu_2+\delta-\delta'+\rho_2},\frac{a'}{b^{\delta}}
    \right)
  }^2
  \\
  \ll
  \abs{
    \mathcal{F}_{\delta-\delta'+\rho_2}\left(
      \alpha b^{\lambda-\mu_2},\frac{a'}{b^{\delta-\delta'+\rho_2}}
    \right)
  }^2
  \\
  \left(
    \abs{
      \mathcal{F}_{\delta'-\rho_2}\left(
        \alpha b^{\lambda-\mu_2+\delta-\delta'+\rho_2},\frac{j}{r_d}
      \right)
    }^2
    +
    \frac{1}{r_d b^{\rho_2}}
  \right)
  .
\end{multline*}
This permits us to write
\begin{equation}\label{eq:T7-T8}
  T_7(\alpha,\lambda,\mu,\rho,\rho',r,\delta_0)
  \ll
  T_8(\alpha,\lambda,\mu,\rho,\rho', r,\delta_0)
  +
  E_8(\alpha,\lambda,\mu,\rho,\rho', r,\delta_0)
  ,
\end{equation}
with
\begin{multline*}
  T_8(\alpha,\lambda,\mu,\rho,\rho', r,\delta_0)
  \\
  =
  \sum_{\delta_0+1\leq \delta \leq \mu_2}
  b^{\nu+2\eta_b(\mu_2-\delta)}  
  \sum_{\substack{d \dv  b^{\mu_2-\delta}\\ \gcd(d,b)<b}} d^{1-2\eta_b}
  \gcd(r,d)
  \sum_{\delta'=\mu+\delta-\mu_2-\rho_2}^{\mu+\delta-\mu_2}
  b^{\delta'}
  \\
  \sum_{j=0}^{r_d-1}
  \abs{
    \mathcal{F}_{\delta'-\rho_2}\left(
      \alpha b^{\lambda-\mu_2+\delta-\delta'+\rho_2},\frac{j}{r_d}
    \right)
  }^2
  \\
  \sum_{a' \in \mathcal{E}_j(r_d,\delta,\delta')}
  \abs{
    \mathcal{F}_{\delta-\delta'+\rho_2}\left(
      \alpha b^{\lambda-\mu_2},\frac{a'}{b^{\delta-\delta'+\rho_2}}
    \right)
  }^2
\end{multline*}
and
\begin{multline*}
  E_8(\alpha,\lambda,\mu,\rho,\rho', r,\delta_0)
  \\
  =
  \sum_{\delta_0+1\leq \delta \leq \mu_2}
  b^{\nu+2\eta_b(\mu_2-\delta)}  
  \sum_{\substack{d \dv  b^{\mu_2-\delta}\\ \gcd(d,b)<b}} d^{1-2\eta_b}
  \gcd(r,d)
  \sum_{\delta'=\mu+\delta-\mu_2-\rho_2}^{\mu+\delta-\mu_2}
  b^{\delta'}
  \\
  \sum_{j=0}^{r_d-1}
  \frac{1}{r_d b^{\rho_2}}
  \sum_{a' \in \mathcal{E}_j(r_d,\delta,\delta')}
  \abs{
    \mathcal{F}_{\delta-\delta'+\rho_2}\left(
      \alpha b^{\lambda-\mu_2},\frac{a'}{b^{\delta-\delta'+\rho_2}}
    \right)
  }^2
  .  
\end{multline*}
Let us first consider $E_8$.
For $j \in\{1,\ldots,r_d-1\}$,
the set $\mathcal{E}_j(r_d,\delta,\delta')$ is a set of consecutive
integers of cardinal at most
\begin{math}
  2 b^{\delta-\delta'} + 2 \leq b^{\delta-\delta'+\rho_2}
  ,
\end{math}
and the set $\mathcal{E}_0(r_d,\delta,\delta')$ is the union of two
such sets. Extending the summation to the full set of residues
modulo $b^{\delta-\delta'+\rho_2}$,
by Lemma~\ref{lemma:L2-mean-value}
we have for $j \in\{0,\ldots,r_d-1\}$:
\begin{multline}\label{eq:majoration-sum-F-E_j}
  \sum_{a' \in \mathcal{E}_j(r_d,\delta,\delta')}
  \abs{
    \mathcal{F}_{\delta-\delta'+\rho_2}\left(
      \alpha b^{\lambda-\mu_2},\frac{a'}{b^{\delta-\delta'+\rho_2}}
    \right)
  }^2
  \\
  \ll
  \sum_{0\leq a'< b^{\delta-\delta'+\rho_2}}
  \abs{
    \mathcal{F}_{\delta-\delta'+\rho_2}\left(
      \alpha b^{\lambda-\mu_2},\frac{a'}{b^{\delta-\delta'+\rho_2}}
    \right)
  }^2
  = 1,
\end{multline}
hence
\begin{multline*}
  E_8(\alpha,\lambda,\mu,\rho,\rho', r,\delta_0)
  \\
  \ll
  \sum_{\delta_0+1\leq \delta \leq \mu_2}
  b^{\nu+2\eta_b(\mu_2-\delta)}  
  \sum_{\substack{d \dv  b^{\mu_2-\delta}\\ \gcd(d,b)<b}} d^{1-2\eta_b}
  \gcd(r,d)
  \sum_{\delta'=\mu+\delta-\mu_2-\rho_2}^{\mu+\delta-\mu_2}
  b^{\delta'}
  \frac{1}{b^{\rho_2}}
  .
\end{multline*}
By Lemma~\ref{lemma:gcd-sum},
using $\tau(d) \leq \tau\left(b^{\mu_2}\right)$
and \eqref{eq:sigma-z>0-b-lambda}, we have
\begin{multline}\label{eq:combined-gcd-sigma-upperbound}
  \frac{1}{b^\rho} \sum_{r=1}^{b^{\rho}-1}
  b^{2\eta_b(\mu_2-\delta)}  
  \sum_{\substack{d \dv  b^{\mu_2-\delta}\\ \gcd(d,b)<b}} d^{1-2\eta_b}
  \gcd(r,d)
  \ll
  b^{2\eta_b(\mu_2-\delta)}  
  \sum_{\substack{d \dv  b^{\mu_2-\delta}\\ \gcd(d,b)<b}} d^{1-2\eta_b}
  \tau(d)
  \\
  \ll
  \tau\left(b^{\mu_2}\right)
  b^{2\eta_b(\mu_2-\delta)}  
  b^{(1-2\eta_b)(\mu_2-\delta)}
  \ll
  \tau\left(b^{\mu_2}\right)
  b^{\mu_2-\delta}
  ,
\end{multline}
which gives
\begin{align}\label{eq:majoration-E8}
  \frac{1}{b^\rho} \sum_{r=1}^{b^{\rho}-1}
  E_8(\alpha,\lambda,\mu,\rho,\rho', r,\delta_0)
  &\ll
  \tau\left(b^{\mu_2}\right)
  \sum_{\delta_0+1\leq \delta \leq \mu_2}
  b^{\nu}  
  b^{\mu_2-\delta}
  b^{\mu+\delta-\mu_2-\rho_2}
  \\
  \nonumber
  &\ll
  \tau\left(b^{\mu_2}\right)
  \mu_2 b^{\lambda-\rho_2}
  .
\end{align}
Let us now consider $T_8$.
By \eqref{eq:majoration-sum-F-E_j} we have
\begin{multline*}
  T_8(\alpha,\lambda,\mu,\rho,\rho', r,\delta_0)
  \\
  \ll
  \sum_{\delta_0+1\leq \delta \leq \mu_2}
  b^{\nu+2\eta_b(\mu_2-\delta)}  
  \sum_{\substack{d \dv  b^{\mu_2-\delta}\\ \gcd(d,b)<b}} d^{1-2\eta_b}
  \gcd(r,d)
  \\
  \sum_{\delta'=\mu+\delta-\mu_2-\rho_2}^{\mu+\delta-\mu_2}
  b^{\delta'}
  \sum_{j=0}^{r_d-1}
  \abs{
    \mathcal{F}_{\delta'-\rho_2}\left(
      \alpha b^{\lambda-\mu_2+\delta-\delta'+\rho_2},\frac{j}{r_d}
    \right)
  }^2
  .
\end{multline*}
By Lemma~\ref{FT-splitting-product}
and Lemma~\ref{lemma:pointwise-upperbound-of-F}
we have
\begin{multline*}
  \abs{
    \mathcal{F}_{\delta'-\rho_2}\left(
      \alpha b^{\lambda-\mu_2+\delta-\delta'+\rho_2},\frac{j}{r_d}
    \right)
  }^2
  \leq
  \\
  \abs{
    \mathcal{F}_{\rho}\left(
      \alpha b^{\lambda-\mu_2+\delta-\delta'+\rho_2+\delta'-\rho_2-\rho},\frac{j}{r_d}
    \right)
  }^2
  \hspace{-.6em}
  G_{\delta'-\rho_2-\rho-1}\left(
      \alpha b^{\lambda-\mu_2+\delta-\delta'+\rho_2} (b^2-1)
    \right)
    .
\end{multline*}
The sequence $(j/r_d)_{j=0,\ldots,r_d-1}$ is
$r_d^{-1}$ well spaced modulo~$1$
and $1\leq r_d \leq r < b^\rho$ so this sequence is
$b^{-\rho }$ well spaced modulo~$1$.
By Lemma~\ref{lemma:L2-mean-value-in-alpha} we have
\begin{displaymath}
  \sum_{j=0}^{r_d-1}
  \abs{
    \mathcal{F}_{\rho}\left(
      \alpha
      b^{\lambda-\mu_2+\delta-\delta'+\rho_2+\delta'-\rho_2-\rho},
      \frac{j}{r_d}
    \right)
  }^2  
  \ll 1
  ,
\end{displaymath}
hence
\begin{multline*}
  T_8(\alpha,\lambda,\mu,\rho,\rho', r,\delta_0)
  \\
  \ll
  \sum_{\delta_0+1\leq \delta \leq \mu_2}
  b^{\nu+2\eta_b(\mu_2-\delta)}  
  \sum_{\substack{d \dv  b^{\mu_2-\delta}\\ \gcd(d,b)<b}} d^{1-2\eta_b}
  \gcd(r,d)
  \\
  \sum_{\delta'=\mu+\delta-\mu_2-\rho_2}^{\mu+\delta-\mu_2}
  b^{\delta'}
  G_{\delta'-\rho_2-\rho-1}\left(
      \alpha b^{\lambda-\mu_2+\delta-\delta'+\rho_2} (b^2-1)
    \right)
  ,
\end{multline*}
which, by \eqref{eq:combined-gcd-sigma-upperbound}
gives
\begin{equation}\label{eq:T8-upperbound-only-with-G}
  \frac{1}{b^\rho} \sum_{r=1}^{b^{\rho}-1}
  T_8(\alpha,\lambda,\mu,\rho,\rho', r,\delta_0)
  \ll
  T_9(\alpha,\lambda,\mu,\rho,\rho',\delta_0)
\end{equation}
with
\begin{multline}\label{eq:definition-T9}
  T_9(\alpha,\lambda,\mu,\rho,\rho',\delta_0)
  =
  \tau\left(b^{\mu_2}\right)
  \sum_{\delta_0+1\leq \delta \leq \mu_2}
  b^{\nu+\mu_2-\delta}  
  \\
  \sum_{\delta'=\mu+\delta-\mu_2-\rho_2}^{\mu+\delta-\mu_2}
  b^{\delta'}
  G_{\delta'-\rho_2-\rho-1}\left(
      \alpha b^{\lambda-\mu_2+\delta-\delta'+\rho_2} (b^2-1)
    \right)
  .
\end{multline}
By \eqref{eq:T4-T7}, \eqref{eq:T7-T8},
\eqref{eq:majoration-E8} we deduce
\begin{equation}\label{eq:T4-upperbound-only-with-G}
  \frac{1}{b^\rho} \sum_{r=1}^{b^{\rho}-1}
  T_4(\alpha,\lambda,\mu,\mu_2, r,\delta_0+1,\mu_2)
  \ll
  T_9(\alpha,\lambda,\mu,\rho,\rho',\delta_0)
  +
  \tau\left(b^{\mu_2}\right)
  \mu_2 b^{\lambda-\rho_2}
  .
\end{equation}
By \eqref{eq:maj_SIIcarre_end} and \eqref{eq:definition-rho2} we get
\begin{multline}\label{eq:S_II-T9}
  S_{II}(\alpha,\lambda,\mu)^2
  \\
  \ll
  \mu_2
  \tau\left(b^{\mu_2}\right)
  b^{2\lambda-\rho_2}
  +
  b^{\lambda}
  b^{\mu_2(1+2\eta_b)} \log b^{\mu_2}
  +
  \frac{b^{2\lambda+\rho+\rho'}}{P^-(b)^{(1-2\eta_b)\mu_2}}
  \tau(b^{\mu_2}) \log b^{\mu_2}
  \\
  +
  \left(\tau\left(b^{\mu_2}\right)\right)^2
  \frac{b^{2\lambda}}{P^-(b)^{(1-2\eta_b)(\mu_2-\delta_0)}}
  +
  b^\lambda
  T_9(\alpha,\lambda,\mu,\rho,\rho',\delta_0)  
  .
\end{multline}

\section{An individual bound for $S_{II}(\alpha,\lambda,\mu)$
  for $\alpha =h/d$ with $d$ small.}

In this part we assume that $\alpha =h/d$, with
$1 \leq h < d$,
\begin{math}
  (b^2-1) b^\lambda h \not\equiv 0 \bmod d
  .
\end{math}

Since
\begin{math}
  \delta'-\rho_2-\rho-2+\lambda-\mu_2+\delta-\delta'+\rho_2
  =
  \lambda-\mu_2+\delta-\rho-2
  \leq \lambda
\end{math}
we have
\begin{displaymath}
  b^{\delta'-\rho_2-\rho-2}
  b^{\lambda-\mu_2+\delta-\delta'+\rho_2}
  (b^2-1) h
  \not\equiv 0 \bmod d
  ,
\end{displaymath}
we are able to use the individual bound for
\begin{displaymath}
  G_{\delta'-\rho_2-\rho-1}\left(
    \alpha b^{\lambda-\mu_2+\delta-\delta'+\rho_2} (b^2-1) \right)
\end{displaymath}
given by \eqref{eq:uniform-upperbound-of-G-for-denominator-d}
to bound above
\eqref{eq:definition-T9}:
\begin{align*}
  &
    T_9(\alpha,\lambda,\mu,\rho,\rho',\delta_0)
    \\
  &\ll
    \tau\left(b^{\mu_2}\right)
    \sum_{\delta_0+1\leq \delta \leq \mu_2}
    b^{\nu+\mu_2-\delta}  
    \sum_{\delta'=\mu+\delta-\mu_2-\rho_2}^{\mu+\delta-\mu_2}
    b^{\delta'}
    b^{\frac{-2\Upsilon_b}{\log d}(\delta'-\rho_2-\rho-1)}\\
  &\ll
    \tau\left(b^{\mu_2}\right)b^\lambda 
    b^{\frac{-2\Upsilon_b}{\log d}(\delta_0-\rho_2-\rho-1)}
    ,
\end{align*}
where $\Upsilon_b$ is defined by \eqref{eq:definition-Upsilon_b}.  By
\eqref{eq:S_II-T9} we deduce
\begin{multline}\label{eq:type-II-small-d}
  S_{II}(\alpha,\lambda,\mu)^2
  \\
  \ll
  \mu_2\tau (b^{\mu_2})b^{2\lambda-\rho_2}
  +
  b^{\lambda}
  b^{\mu_2(1+2\eta_b)} \log b^{\mu_2}
  +
  \frac{b^{2\lambda+\rho+\rho'}}{P^-(b)^{(1-2\eta_b)\mu_2}}
  \tau(b^{\mu_2}) \log b^{\mu_2}
  \\
  +
  \left(\tau\left(b^{\mu_2}\right)\right)^2
  \frac{b^{2\lambda}}{P^-(b)^{(1-2\eta_b)(\mu_2-\delta_0)}}
  +
  \tau\left(b^{\mu_2}\right)b^{2\lambda}
  b^{\frac{-2\Upsilon_b}{\log d}(\delta_0-\rho_2-\rho-1)}
  .
\end{multline}

\begin{lemma}\label{lemma:conclusion-S_II-individual}
  Let $0<\beta_1 < \beta_2 < (1+2\eta_b)^{-1}$. There exists
  $C = C(b,\beta_1,\beta_2) \in \left]0,1\right[$ such that if
  $\beta_1\lambda\leq \mu \leq \beta_2\lambda$ then for any
  $\alpha =h/d$ with
  \begin{math}
    (b^2-1) b^\lambda h \not\equiv 0 \bmod d,
  \end{math}
  we have 
  \begin{equation}\label{eq:conclusion-S_II-individual}
    S_{II}(\alpha,\lambda,\mu)
    \ll
    b^{\lambda (1-C)}
    +
    \lambda^{\frac{\omega(b)}2}
    b^{\lambda\left(1-\frac{\Upsilon_b\beta_1}{2\log d}\right)}
    .
  \end{equation}
\end{lemma}
\begin{proof}
  Assume that $\alpha =h/d$ with
  \begin{math}
    (b^2-1) b^\lambda h \not\equiv 0 \bmod d.
  \end{math} 
  We proved that \eqref{eq:type-II-small-d} holds for any positive
  integers $\rho,\rho',\rho_2,\mu_2,\delta_0$ satisfying
  \eqref{eq:condition-rho-rho'}, \eqref{eq:definition-mu2},
  \eqref{eq:condition-delta0} and \eqref{eq:definition-rho2}.  If
  $\rho < \min(\mu,\tfrac{\lambda-\mu}2)$ then
  \eqref{eq:condition-rho-rho'}, \eqref{eq:definition-mu2},
  \eqref{eq:condition-delta0} and \eqref{eq:definition-rho2} are
  satisfied with
  \begin{displaymath}
    \rho'=\rho =\rho_2, \qquad \mu_2 = \mu+2\rho, \qquad \delta_0 = \mu +\rho,
  \end{displaymath}
  hence, since by \eqref{eq:tau-b-lambda} we have
  $\tau(b^{\mu_2})\leq \tau(b) \mu_2^{\omega(b)} \leq \tau(b)
  \lambda^{\omega(b)}$, we obtain
  \begin{multline*}
    S_{II}(\alpha,\lambda,\mu)^2
    \ll
    \lambda^{1+\omega(b)}b^{2\lambda-\rho}
    +
    \lambda b^{\lambda+(\mu+2\rho)(1+2\eta_b)} 
    \\
    +
    \lambda^{1+\omega(b)}
    b^{2\lambda+2\rho-(1-2\eta_b)(\mu+2\rho)\frac{\log P^-(b)}{\log b}}
    \\
    +
    \lambda^{2\omega(b)}
    b^{2\lambda-(1-2\eta_b)\rho \frac{\log P^-(b)}{\log b}}
    +
    \lambda^{\omega(b)}
    b^{2\lambda-\frac{2\Upsilon_b}{\log d}(\mu-\rho)}
    .
  \end{multline*}
  Since $\beta_2 < (1+2\eta_b)^{-1}$, it is possible to choose
  explicitly $c_3 = c(b,\beta_1,\beta_2) > 0$ small enough so that
  \begin{displaymath}
    c_3 < \min\left(\beta_1,\frac{1-\beta_2}2\right),
    \quad
    (\beta_2 + 2 c_3)(1+2\eta_b) < 1,
  \end{displaymath}
  \begin{displaymath}
    2c_3-(1-2\eta_b)(\beta_1+2c_3)\frac{\log P^-(b)}{\log b}
    <0,
    \quad
    \beta_1-c_3 \geq \frac{\beta_1}{2}.
  \end{displaymath}
  Assuming that $\beta_1\lambda\leq \mu \leq \beta_2\lambda$, we can
  choose $\rho=\lfloor c_3 \lambda \rfloor$ and then obtain
  \begin{displaymath}
    S_{II}(\alpha,\lambda,\mu)^2
    \ll
    b^{2\lambda (1-C)}
    +
    \lambda^{\omega(b)}
    b^{2\lambda\left(1-\frac{\Upsilon_b\beta_1}{2\log d}\right)}
  \end{displaymath}
  for some explicit $C = C(b,\beta_1,\beta_2) \in \left]0,1\right[$,
  which gives~\eqref{eq:conclusion-S_II-individual}.
\end{proof}

\section{Bound of $S_{II}(\alpha,\lambda,\mu)$ on average over $\alpha$}

Let $D\geq 1$ and let $\Delta$ be the unique integer
such that
\begin{equation}
  \label{eq:definition-Delta}
  b^{\Delta-1}\leq D < b^\Delta.
\end{equation}
Let us first derive from \eqref{eq:S_II-T9} an upper bound for
\begin{displaymath}
  \frac{1}{D}
  \sum_{\alpha\in {\mathcal A}_D} 
  S_{II}(\alpha,\lambda,\mu)
  .
\end{displaymath}
We notice that both sums in \eqref{eq:definition-T9}
are rather short.
Using the inequality
\begin{equation}
  \forall x_1\geq 0,\ldots,\ \forall x_\ell\geq 0,\
  \sqrt{x_1+\cdots + x_\ell}
  \leq
  \sqrt{x_1}+\cdots +\sqrt{x_\ell}
  ,
\end{equation}
it follows from \eqref{eq:definition-T9} that
\begin{multline}\label{eq:T8^¹/2-G^1/2}
  \frac{1}{D} \sum_{\alpha\in {\mathcal A}_D} b^{\frac{\lambda}2}
  T_9(\alpha,\lambda,\mu,\rho,\rho',\delta_0)^{\frac12}
  \\
  \ll \tau\left(b^{\mu_2}\right)^{\frac12} b^{\frac{\lambda}{2}}
  \sum_{\delta_0+1\leq \delta \leq \mu_2}
  b^{\frac12(\nu+\mu_2-\delta)}
  \sum_{\delta'=\mu+\delta-\mu_2-\rho_2}^{\mu+\delta-\mu_2}
  b^{\frac{\delta'}{2}}\, T_{10}(\lambda,\mu,\rho,\rho',\delta',\delta,D)
\end{multline}
with
\begin{equation}\label{eq:def_T10}
  T_{10}(\lambda,\mu,\rho,\rho',\delta',\delta,D)
  =
  \frac{1}{D}
  \sum_{\alpha\in {\mathcal A}_D} 
  G_{\delta'-\rho_2-\rho-1}^{1/2}
  \left(
      \alpha b^{\lambda-\mu_2+\delta-\delta'+\rho_2} (b^2-1)
    \right)
  .
\end{equation}
Since the denominators of the $\alpha$'s are coprime to $b(b^2-1)$, we
have
\begin{displaymath}
  T_{10}(\lambda,\mu,\rho,\rho',\delta',\delta,D)
  =
  \frac{1}{D}
  \sum_{\alpha\in {\mathcal A}_D} 
  G_{\delta'-\rho_2-\rho-1}^{1/2} \left( \alpha \right)
  .
\end{displaymath}
We intend to apply
Lemma~\ref{lemma:sobolev-gallagher-combined-with-holder}
with an appropriate value of $\kappa$.
By \eqref{eq:definition-zeta_b}
and \eqref{eq:max-Tb-small-kappa}
we have $\zeta_{b,1}>0$.
Moreover by~\eqref{eq:max-Tb-simple},
\begin{math}
  b^{1-\zeta_{b,\kappa}}
  = b\max_{\alpha\in\R} T_{b,\kappa}(\alpha)
\end{math}
converges to $1$
as $\kappa \to + \infty$.
Hence for any large enough integer $\kappa$ we have
\begin{equation}\label{eq:b^-zeta_b-small}
  b^{-\zeta_{b,1}}
  b^{1-\zeta_{b,\kappa}}
  < 1
  .
\end{equation}
We choose $\kappa = \kappa_b$ to be the smallest positive integer such
that \eqref{eq:b^-zeta_b-small} is satisfied:
\begin{equation}\label{eq:def_kappa_b}
  \kappa_b = \min\{\kappa \in \N : b^{-\zeta_{b,1}}
  b^{1-\zeta_{b,\kappa}} < 1\}
  .
\end{equation}
An investigation of the size of $\kappa_b$ will be provided in
Section~\ref{section_study_xi_0}.

Let us assume that $D$ is small enough so that $\Delta$
satisfies 
\begin{equation}\label{eq:condition-kappa-Delta}
  (\kappa_b+1)2\Delta
  \leq
  \mu+\delta_0-\mu_2-2\rho_2-\rho 
  =
  \delta_0-2\rho_2-2\rho-\rho' .
\end{equation}
For
\begin{math}
  \delta \geq \delta_0+1
\end{math}
and 
\begin{math}
  \delta' \geq \mu + \delta - \mu_2 - \rho_2,
\end{math}
we have
\begin{math}
  (\kappa_b+1)2\Delta \leq \delta'-\rho_2-\rho-1
\end{math}
and since, by the definition of $\mathcal{A}_D$ given in
\eqref{eq:definition-A-d}, the sequence of the $\alpha$'s in
$\mathcal{A}_D$ is $b^{-2\Delta}$ well spaced modulo~$1$, we get by
Lemma~\ref{lemma:sobolev-gallagher-combined-with-holder}
\begin{displaymath}
  T_{10}(\lambda,\mu,\rho,\rho',\delta',\delta,D)
  \ll
  b^{-\Delta}
  b^{(1-\zeta_{b,1})\Delta} \,
  (\kappa_b+2)^{1/2}
  b^{(1-\zeta_{b,\kappa_b})\Delta}
  ,
\end{displaymath}
which, by \eqref{eq:b^-zeta_b-small}, implies that
\begin{displaymath}
  T_{10}(\lambda,\mu,\rho,\rho',\delta',\delta,D)
  \ll_b
  b^{-\varepsilon_b \Delta} \,
  ,
\end{displaymath}
where $\varepsilon_b>0$ is defined by the equality
\begin{displaymath}
  b^{-\varepsilon_b}
  =
  b^{-\zeta_{b,1}}
  b^{1-\zeta_{b,\kappa_b}}
  <1.
\end{displaymath}
Inserting this in \eqref{eq:T8^¹/2-G^1/2}, we obtain
\begin{multline*}
  \frac{1}{D}
  \sum_{\alpha\in {\mathcal A}_D} b^{\frac{\lambda}2}
  T_9(\alpha,\lambda,\mu,\rho,\rho',\delta_0)^{\frac12}
  \\
  \ll
  \tau\left(b^{\mu_2}\right)^{\frac12}
  b^{\frac{\lambda}{2}-\varepsilon_b\Delta}
  \sum_{\delta_0+1\leq \delta \leq \mu_2}
  b^{\frac12(\nu+\mu_2-\delta)}  
  \sum_{\delta'=\mu+\delta-\mu_2-\rho_2}^{\mu+\delta-\mu_2}
  b^{\frac{\delta'}{2}}
  \\
  \ll
  \tau\left(b^{\mu_2}\right)^{\frac12}
  b^{\frac{\lambda}{2}-\varepsilon_b\Delta}
  \sum_{\delta_0+1\leq \delta \leq \mu_2}
  b^{\frac12(\nu+\mu_2-\delta)}
  b^{\frac12(\mu+\delta-\mu_2)}
  \ll
  \tau\left(b^{\mu_2}\right)^{\frac12}
  \mu_2
  b^{\lambda-\varepsilon_b\Delta}
  .
\end{multline*}
By \eqref{eq:S_II-T9}, we deduce that if $D$ is small enough so that
$\Delta$ satisfies~\eqref{eq:condition-kappa-Delta} then
\begin{multline}\label{eq:average-S_II}
  \frac{1}{D}
  \sum_{\alpha\in {\mathcal A}_D} 
  S_{II}(\alpha,\lambda,\mu)
  \ll
  \mu_2^{\frac12}
  \tau\left(b^{\mu_2}\right)^{\frac12}
  b^{\lambda+\Delta-\frac{\rho_2}{2}}
  +
  b^{\frac{\lambda}{2}+\Delta+\frac12\mu_2(1+2\eta_b)}
  \left(\log b^{\mu_2}\right)^{\frac12}
  \\
  +
  \frac{b^{\lambda+\Delta+\frac12(\rho+\rho')}}{P^-(b)^{\frac12(1-2\eta_b)\mu_2}}
  \tau(b^{\mu_2})^{\frac12} \left( \log b^{\mu_2} \right)^{\frac12}
  \\
  +
  \tau\left(b^{\mu_2}\right)
  \frac{b^{\lambda+\Delta}}{P^-(b)^{\frac12(1-2\eta_b)(\mu_2-\delta_0)}}
  +
  \tau\left(b^{\mu_2}\right)^{\frac12}
  \mu_2
  b^{\lambda-\varepsilon_b\Delta}
  .
\end{multline}

Let us now focus on the sum 
\begin{displaymath}
  \sum_{\alpha\in \mathcal{A}_D}
  W(\alpha)
  S_{II}(\alpha,\lambda,\mu)
  .
\end{displaymath}
Let $\Delta_0 < \Delta$ be an integer and $D_0= b^{\Delta_0}$.
We write
\begin{equation}\label{eq:decomp_sum_W_SII}
  \sum_{\alpha\in \mathcal{A}_D}
  W(\alpha)
  S_{II}(\alpha,\lambda,\mu)
  =
  V_1(\lambda,\mu,D_0)
  +
  V_2(\lambda,\mu,D_0,D)
\end{equation}
where, in $V_1(\lambda,\mu,D_0)$, the denominator of $\alpha$ is $\leq D_0$ :
\begin{equation}\label{eq:def_V_1}
  V_1(\lambda,\mu,D_0) = \sum_{\alpha\in \mathcal{A}_{D_0}}
  W(\alpha)
  S_{II}(\alpha,\lambda,\mu)
  \leq
  D_0
  \sup_{\alpha \in \mathcal{A}_{D_0}} S_{II}(\alpha,\lambda,\mu)
\end{equation}
and in $V_2(\lambda,\mu,D_0,D)$, the denominator $d$ of $\alpha$ satisfies
$D_0<d\leq D$. By $b$-adic splitting, we have
\begin{multline}\label{eq:maj_V_2_ini}
  V_2(\lambda,\mu,D_0,D)
  \leq
  \sum_{0\leq i < \Delta-\Delta_0} 
  \sum_{\alpha\in \mathcal{A}_{D b^{-i}}\setminus \mathcal{A}_{D b^{-(i+1)}}}
  W(\alpha)
  S_{II}(\alpha,\lambda,\mu)\\
  \leq
  \sum_{0\leq i < \Delta-\Delta_0} 
  \frac{b^{i+1}}{D}
  \sum_{\alpha\in \mathcal{A}_{D b^{-i}}}
  S_{II}(\alpha,\lambda,\mu)
  .
\end{multline}
If $D$ is small enough so that $\Delta$
satisfies~\eqref{eq:condition-kappa-Delta} then, since
\eqref{eq:condition-kappa-Delta} is also satisfied with $\Delta$
replaced by $\Delta-i$, we can apply \eqref{eq:average-S_II} with $D$
replaced by $Db^{-i}$ so that $\Delta$ is replaced by $\Delta-i$,
which leads to
\begin{multline}\label{eq:maj_V2}
  V_2(\lambda,\mu,D_0,D)
  \ll
  \mu_2^{\frac12}
  \tau\left(b^{\mu_2}\right)^{\frac12}
  b^{\lambda+\Delta-\frac{\rho_2}{2}}
  +
  b^{\frac{\lambda}{2}+\Delta+\frac12\mu_2(1+2\eta_b)}
  \left(\log b^{\mu_2}\right)^{\frac12}
  \\
  +
  \frac{b^{\lambda+\Delta+\frac12(\rho+\rho')}}{P^-(b)^{\frac12(1-2\eta_b)\mu_2}}
  \tau(b^{\mu_2})^{\frac12} \left( \log b^{\mu_2} \right)^{\frac12}
  \\
  +
  \tau\left(b^{\mu_2}\right)
  \frac{b^{\lambda+\Delta}}{P^-(b)^{\frac12(1-2\eta_b)(\mu_2-\delta_0)}}
  +
  \tau\left(b^{\mu_2}\right)^{\frac12}
  \mu_2
  b^{\lambda-\varepsilon_b\Delta_0}
  .
\end{multline}

To conclude this section, we will deduce the following result.  For
convenience, we introduce
\begin{equation}\label{eq:def_u_b}
  u_b
  =
  \frac{1-2\eta_b}2\frac{\log P^-(b)}{\log b}
  \in\left]0,\frac12\right[
\end{equation}
and
\begin{equation}\label{eq:def_iota_b}
  \iota_b
  =
  \frac{u_b}
  {1 + 2(\kappa_b+1)u_b}
  =
    \frac{1}{2(\kappa_b+1)}\left(1-\frac{1}{1+2(\kappa_b+1)u_b}\right)
  > 0
  .
\end{equation}
We note that
\begin{equation}\label{eq:formula_iota_u_kappa}
  \frac{\iota_b}{1+6\iota_b}
  =
  \frac{u_b}{1+2(\kappa_b+4)u_b}
  =
  \frac{1}{2(\kappa_b+4)}\left(1-\frac{1}{1+2(\kappa_b+4)u_b}\right)
  .
\end{equation}

\begin{lemma}\label{lemma:conclusion-S_II-sum-on-alpha}
  Let $0<\beta_1 < \beta_2 < (1+2\eta_b)^{-1}$,
  \begin{equation}\label{eq:def_C_b,beta1,beta2}
    C'(b,\beta_1,\beta_2)
    =
    \min\left(
      \frac{1-(1+2\eta_b)\beta_2}{2(3+4\eta_b)},\,
      \frac{u_b}{3-4u_b} \beta_1,\,
      \frac{\iota_b}{1+6\iota_b}\beta_1
    \right)
    >0
  \end{equation}
  and 
  \begin{math}
    \xi \in \left]0,C'(b,\beta_1,\beta_2)\right[.
  \end{math}
  There exists $c' = c'(b,\beta_1) >0$ and
  $\lambda'_0(b,\beta_1,\beta_2,\xi)\geq 2$ such that if
  $\lambda\geq\lambda'_0(b,\beta_1,\beta_2,\xi)$ and
  $\beta_1\lambda\leq \mu \leq \beta_2\lambda$ then, denoting $D =
  b^{\xi \lambda}$,
  \begin{equation}\label{eq:conclusion-S_II-sum-on-alpha}
    \sum_{\alpha\in \mathcal{A}_D}
    W(\alpha)
    S_{II}(\alpha,\lambda,\mu)
    \ll
    b^{\lambda-c'\sqrt{\lambda}}
    .
  \end{equation}
\end{lemma}
\begin{proof}
  Let $D\geq 1$ and let $\Delta$ be the unique integer such that
  $b^{\Delta-1}\leq D < b^\Delta$, according
  to~\eqref{eq:definition-Delta}.  Let also $\Delta_0 < \Delta$ be an
  integer and $D_0= b^{\Delta_0}$.  By~\eqref{eq:decomp_sum_W_SII}, it
  suffices to bound $V_1(\lambda,\mu,D_0)$ and
  $V_2(\lambda,\mu,D_0,D)$
  and then choose
  $\Delta_0$ so that both upper bounds are admissible under the
  assumptions of the lemma.

  Let us first bound $V_2(\lambda,\mu,D_0,D)$.
  Since \eqref{eq:maj_V2} holds for any positive integers
  $\rho,\rho',\rho_2,\mu_2,\delta_0$ satisfying
  \eqref{eq:condition-rho-rho'}, \eqref{eq:definition-mu2},
  \eqref{eq:condition-delta0}, \eqref{eq:definition-rho2} and
  \eqref{eq:condition-kappa-Delta}, we intend to make a suitable
  choice of the parameters $\rho,\rho',\rho_2,\mu_2,\delta_0$.  If
  $\rho$ and $\delta_0$ are such that
  \begin{equation}\label{eq:cond_rho_delta_0}
    \rho < \tfrac{\lambda-\mu}2
    \quad \text{and} \quad
    (\kappa_b+1)2\Delta + 5\rho \leq \delta_0 < \mu + 2\rho
    ,
  \end{equation}
  then \eqref{eq:condition-rho-rho'}, \eqref{eq:definition-mu2},
  \eqref{eq:condition-delta0}, \eqref{eq:definition-rho2} and
  \eqref{eq:condition-kappa-Delta} are satisfied with
  \begin{displaymath}
    \rho'=\rho =\rho_2, \qquad \mu_2 = \mu+2\rho,
  \end{displaymath}
  hence, since by \eqref{eq:tau-b-lambda},
  $\tau(b^{\mu_2})\leq \tau(b) \mu_2^{\omega(b)} \leq \tau(b)
  \lambda^{\omega(b)}$, we get
  \begin{multline*}
    V_2(\lambda,\mu,D_0,D)
    \ll
    \lambda^{\frac{1+\omega(b)}2}
    b^{\lambda+\Delta-\frac{\rho}{2}}
    +
    \lambda^{\frac12}
    b^{\frac{\lambda}{2}+\Delta+\frac{1+2\eta_b}2(\mu+2\rho)}
    \\
    +
    \lambda^{\frac{1+\omega(b)}2}
    b^{\lambda+\Delta+\rho-(\mu+2\rho)u_b}
    \\
    +
    \lambda^{\omega(b)}
    b^{\lambda+\Delta- (\mu+2\rho-\delta_0)u_b}
    +
    \lambda^{\frac{\omega(b)}2+1}
    b^{\lambda-\varepsilon_b\Delta_0}
  \end{multline*}
  where $u_b$ is defined by~\eqref{eq:def_u_b}.  In order to make a
  good choice of the parameter $\delta_0$, we define
  \begin{math}
    x(b,\rho,\mu) 
  \end{math}
  as the solution $x$ of the equation
  \begin{displaymath}
    \frac{x- 5\rho}{2(\kappa_b+1)}
    =
    (\mu+2\rho-x)u_b
    ,
  \end{displaymath}
  namely
  \begin{displaymath}
    x
    =
    \mu+2\rho
    - \frac{\mu-3\rho}{1+2(\kappa_b+1)u_b}
  \end{displaymath}
  which is equivalent to
  \begin{displaymath}
    \frac{x- 5\rho}{2(\kappa_b+1)}
    =
    (\mu-3\rho)\iota_b
  \end{displaymath}
  with $\iota_b$ defined by~\eqref{eq:def_iota_b}.
  If $\rho$ satisfies
  \begin{equation}\label{eq:cond_rho}
    \rho < \tfrac{\lambda-\mu}2,
    \quad
    \Delta \leq (\mu-3\rho)\iota_b
    \quad \text{and} \quad
    \mu-3\rho \geq 1 + 2(\kappa_b+1)u_b,
  \end{equation}
  then it follows from the definition of $x(b,\rho,\mu)$ that
  \eqref{eq:cond_rho_delta_0} is satisfied with
  $\delta_0 = \ceil{x(b,\rho,\mu)}$, hence
  \begin{multline}\label{eq:maj_V2_after_choice_delta_0}
    V_2(\lambda,\mu,D_0,D)
    \ll
    \lambda^{\frac{1+\omega(b)}2}
    b^{\lambda+\Delta-\frac{\rho}{2}}
    +
    \lambda^{\frac12}
    b^{\frac{\lambda}{2}+\Delta+\frac{1+2\eta_b}2(\mu+2\rho)}
    \\
    \qquad
    +
    \lambda^{\frac{1+\omega(b)}2}
    b^{\lambda+\Delta+\rho-(\mu+2\rho)u_b}
    \\
    +
    \lambda^{\omega(b)}
    b^{\lambda+\Delta- (\mu-3\rho)\iota_b}
    +
    \lambda^{\frac{\omega(b)}2+1}
    b^{\lambda-\varepsilon_b\Delta_0}.
  \end{multline}
  In order to make a good choice of the parameter $\rho$, we introduce
  the functions
  \begin{displaymath}
    f_{\lambda,\mu}(x) = \frac{\lambda}2 - \frac{1+2\eta_b}2(\mu+2x),
    \quad
    g_{\mu}(x) = (\mu+2x)u_b-x,
    \quad
    h_{\mu}(x) = (\mu-3x)\iota_b
  \end{displaymath}
  and
  \begin{displaymath}
    M_{\lambda,\mu}(x)
    =
    \min\left(
      \frac{x}2,\, f_{\lambda,\mu}(x),\, g_{\mu}(x),\, h_{\mu}(x)
    \right)
    .
  \end{displaymath}
  If we assume that $\rho$ satisfies \eqref{eq:cond_rho} then the
  bound~\eqref{eq:maj_V2_after_choice_delta_0} implies that
  \begin{displaymath}
    V_2(\lambda,\mu,D_0,D)
    \ll
    \lambda^{\omega(b)}
    b^{\lambda + \Delta - M_{\lambda,\mu}(\rho)}
    +
    \lambda^{\frac{\omega(b)}2+1}
    b^{\lambda-\varepsilon_b\Delta_0}.
  \end{displaymath}
  Let us define $x_1$, $x_2$ and $x_3$ by the equalities
  \begin{displaymath}
    \frac{x_1}2 = f_{\lambda,\mu}(x_1),
    \quad
    \frac{x_2}2 = g_{\mu}(x_2),
    \quad
    \frac{x_3}2 = h_{\mu}(x_3),
  \end{displaymath}
  namely
  \begin{displaymath}
    x_1 = \frac{1}{3+4\eta_b}\lambda - \frac{1+2\eta_b}{3+4\eta_b}\mu,
    \quad
    x_2 = \frac{2u_b}{3-4u_b} \mu,
    \quad
    x_3 = \frac{2\iota_b}{1+6\iota_b}\mu,
  \end{displaymath}
  and let
  \begin{displaymath}
    x_0 = \min(x_1,x_2,x_3).
  \end{displaymath}
  Since $f_{\lambda,\mu}$, $g_{\mu}$
  and $h_{\mu}$ are decreasing, the function
  \begin{math}
    x\mapsto M_{\lambda,\mu}(x)
  \end{math}
  coincides with $x\mapsto \frac{x}2$ (thus increases) on
  $\left]-\infty,x_0\right]$ and coincides with
  $x\mapsto \min\left(f_{\lambda,\mu}(x),g_{\mu}(x),h_{\mu}(x)\right)$
  (thus decreases) on $\left[x_0,+\infty\right[$. In particular, we
  have
  \begin{displaymath}
    h_{\mu}(\floor{x_0})\geq h_{\mu}(x_0)\geq M_{\lambda,\mu}(x_0)
    = \frac{x_0}{2}
    \qquad \text{and} \qquad 
    M_{\lambda,\mu}(\floor{x_0})
    =\frac{\floor{x_0}}{2}.
  \end{displaymath}
  It follows that if
  \begin{equation}\label{eq:conditions_F_lambda_mu}
    x_0 < \tfrac{\lambda-\mu}2,
    \quad
    \Delta < \frac{x_0}{2}
    \quad \text{and} \quad
    \mu-3x_0 \geq 1 + 2(\kappa_b+1)u_b,
  \end{equation}
  then \eqref{eq:cond_rho} is satisfied with
  \begin{math}
    \rho = \floor{x_0},
  \end{math}
  hence
  \begin{displaymath}
    V_2(\lambda,\mu,D_0,D)
    \ll
    \lambda^{\omega(b)}
    b^{\lambda + \Delta - \frac{x_0}{2}}
    +
    \lambda^{\frac{\omega(b)}2+1}
    b^{\lambda-\varepsilon_b\Delta_0}.
  \end{displaymath}
  Since 
  \begin{math}
    x_0 \leq x_1 < \frac{\lambda-\mu}{3},
  \end{math}
  the first condition in \eqref{eq:conditions_F_lambda_mu} is true.
  Moreover, since
  \begin{math}
    x_0 \leq x_3,
  \end{math}
  we have
  \begin{math}
    \mu-3x_0
    \geq 
    \frac{1}{1+6\iota_b}\mu,
  \end{math}
  so if we assume that $\mu\geq (1 + 2(\kappa_b+1)u_b)(1+6\iota_b)$
  then the third condition in \eqref{eq:conditions_F_lambda_mu} is
  true. Also, the definition~\eqref{eq:def_C_b,beta1,beta2} of
  $C'(b,\beta_1,\beta_2)$ ensures that
  \begin{displaymath}
    \beta_1\lambda\leq \mu \leq \beta_2\lambda
    \quad
    \Rightarrow
    \quad
    \frac{x_0}2 \geq C'(b,\beta_1,\beta_2) \lambda.
  \end{displaymath}
  We deduce that if
  \begin{displaymath}
    \beta_1\lambda\leq \mu \leq \beta_2\lambda,
    \ \
    \Delta < C'(b,\beta_1,\beta_2) \lambda
    \quad\text{and}\quad
    \lambda \geq (1 + 2(\kappa_b+1)u_b)(1+6\iota_b)\beta_1^{-1},
  \end{displaymath}
  then \eqref{eq:conditions_F_lambda_mu} is satisfied, hence
  \begin{equation}\label{eq:bound_V_2}
    V_2(\lambda,\mu,D_0,D)
    \ll
    \lambda^{\omega(b)}
    b^{\lambda + \Delta - C'(b,\beta_1,\beta_2)\lambda}
    +
    \lambda^{\frac{\omega(b)}2+1}
    b^{\lambda-\varepsilon_b\Delta_0}.
  \end{equation}
  
  Let us now bound $V_1(\lambda,\mu,D_0)$.  For
  $\alpha = \frac{h}d \in \mathcal{A}_{D_0}$, we have
  \begin{math}
    (b^2-1) b^\lambda h \not\equiv 0 \bmod d
  \end{math}
  and $d \leq D_0 = b^{\Delta_0}$. It follows from \eqref{eq:def_V_1}
  and Lemma~\ref{lemma:conclusion-S_II-individual} that if
  $\beta_1\lambda\leq \mu \leq \beta_2\lambda$ then
  \begin{equation}\label{eq:bound_V_1}
    V_1(\lambda,\mu,D_0) 
    \ll
    b^{\Delta_0}
    \left( b^{(1-C(b,\beta_1,\beta_2))\lambda}
      +
      \lambda^{\frac{\omega(b)}2}
      b^{\left(1-\frac{\Upsilon_b\beta_1}{2\Delta_0 \log b}\right)\lambda}
    \right)
  \end{equation}
  for some $C(b,\beta_1,\beta_2)\in\left]0,1\right[$.

  Let 
  \begin{math}
    \xi \in \left]0,C'(b,\beta_1,\beta_2)\right[.
  \end{math}
  We choose $D = b^{\xi \lambda}$ and
  $D_0 = b^{\floor{c_0\sqrt{\lambda}}}$ with $c_0>0$ defined by the equality
  \begin{displaymath}
    c_0 - \frac{\Upsilon_b\beta_1}{2c_0\log b} = - \varepsilon_b c_0,
  \end{displaymath}
  namely
  \begin{displaymath}
    c_0 = \left(\frac{\Upsilon_b\beta_1}{2(1+\varepsilon_b)\log b}\right)^{1/2}.
  \end{displaymath}
  Therefore $\Delta = \floor{\xi \lambda} +1$ and
  $\Delta_0 = \floor{c_0\sqrt{\lambda}}$.  By
  \eqref{eq:decomp_sum_W_SII}, \eqref{eq:bound_V_2} and
  \eqref{eq:bound_V_1}, if
  $\lambda \geq \lambda_1(b,\beta_1,\beta_2,\xi)$ is large enough and
  $\beta_1\lambda\leq \mu \leq \beta_2\lambda$ then
  \begin{multline*}
    \sum_{\alpha\in \mathcal{A}_D}
    W(\alpha)
    S_{II}(\alpha,\lambda,\mu)
    \ll 
    b^{\lambda}
    \Big(
      b^{c_0\sqrt{\lambda}-C(b,\beta_1,\beta_2)\lambda}
      +
      \lambda^{\frac{\omega(b)}2}
      b^{c_0\sqrt{\lambda}-\frac{\Upsilon_b\beta_1}{2c_0 \log b}\sqrt{\lambda}}
      \\
      +
      \lambda^{\omega(b)}
      b^{(\xi - C'(b,\beta_1,\beta_2))\lambda}
      +
      \lambda^{\frac{\omega(b)}2+1}
      b^{-\varepsilon_b c_0\sqrt{\lambda}}
    \Big)
    .
  \end{multline*}
  Hence, taking for instance  
  \begin{math}
    c'=c'(b,\beta_1)
    = \frac{\varepsilon_bc_0}2
    >0,
  \end{math}
  we obtain that if $\lambda \geq \lambda_2(b,\beta_1,\beta_2,\xi)$
  is large enough and $\beta_1\lambda\leq \mu \leq \beta_2\lambda$
  then
  \begin{displaymath}
    \sum_{\alpha\in \mathcal{A}_D}
    W(\alpha)
    S_{II}(\alpha,\lambda,\mu)
    \ll 
    b^{\lambda-c'\sqrt{\lambda}},
  \end{displaymath}
  which ends the proof of the lemma.
\end{proof}

\section{Type I sums}\label{section:type-I-sums}

Our goal is to estimate the sums arising
from~\eqref{eq:maj_gen_type_I} with $\mathcal{A}$ replaced by
$\mathcal{A}_D$ defined by~\eqref{eq:definition-A-d}, $W(\alpha)$
defined by~\eqref{eq:def-W} and $f(\alpha,n)$ replaced by
$\e(\alpha R_{\lambda}(n))$.

Let $\lambda\geq 2$, $\mu \geq 1$ and $\nu\geq 1$ be integers
satisfying \eqref{eq:definition-lambda}, namely
\begin{displaymath}
  \lambda = \mu + \nu
  .
\end{displaymath}

For $I_\mu$ defined by \eqref{eq:definition-I-lambda},
we consider
\begin{equation}\label{eq:definition-SI-sup}
  S_I(\alpha,\lambda,\mu)
  =
  \sup_{t\in\left[b^{\lambda-1},b^{\lambda}\right]}
  \sum_{m\in I_\mu}
  \abs{
    \sum_{\substack{n\\ b^{\lambda-1} \leq mn < t}}
    \e(\alpha R_{\lambda}(mn))
  }
  .
\end{equation}

We will obtain in Lemma~\ref{lemma:conclusion-S_I-individual} an
individual upper bound of $S_I(\alpha,\lambda,\mu)$ for specific
values of $\alpha$ and in Lemma~\ref{lemma:conclusion-S_I-sum-on-alpha},
an upper bound of the sum
\begin{displaymath}
  \sum_{\alpha\in\A_D} W(\alpha) S_{I}(\alpha,\lambda,\mu)
\end{displaymath}
under the condition that
$D=b^{\xi\lambda}$ for some $\xi>0$ and
\begin{equation}\label{eq:condition-mu-nu-typeI}
  \mu \leq \beta_1 \lambda
\end{equation}
for some real number $\beta_1$ such that $0<\beta_1<1$.

In a first step we transform the sums over $n$ into complete sums
involving the discrete Fourier transform $\mathcal{F}_\lambda$.  Then
the next two subsections are devoted to prove
Lemma~\ref{lemma:conclusion-S_I-individual} and
Lemma~\ref{lemma:conclusion-S_I-sum-on-alpha} respectively.

For $m\in I_\mu$, we write
\begin{align*}
  &
    \sum_{\substack{n\\ b^{\lambda-1} \leq mn < t}} \e(\alpha R_{\lambda}(mn))
  \\
  &=
    \sum_{b^{\lambda-1} \leq \ell < t} \e(\alpha R_{\lambda}(\ell))
    \frac{1}{m}
    \sum_{0\leq k<m} \e\left(\frac{-k\ell}{m}\right)
  \\
  &=
    \frac{1}{m}
    \sum_{0\leq k<m}
    \left(
    \sum_{0\leq \ell <t} 
    \e\left(\alpha R_{\lambda}(\ell)-\frac{k\ell}{m}\right)
    -
    \sum_{0\leq \ell <  b^{\lambda-1}} 
    \e\left(\alpha R_{\lambda}(\ell)-\frac{k\ell}{m}\right)
    \right)
    ,
\end{align*}
so that,
since
\begin{math}
  b^{\lambda-1} \leq t \leq b^{\lambda}
  ,
\end{math}
by Lemma~\ref{lemma:b-adic-splitting-of-type-I-sums}
we have
\begin{displaymath}
  \abs{\sum_{\substack{n\\ b^{\lambda-1} \leq mn < t}} \e(\alpha R_{\lambda}(mn))}
  \leq
  \frac{1}{m}  \sum_{0\leq k<m}
  2(b-1) \sum_{0\leq \lambda'\leq \lambda}
  b^{\lambda'}
  \abs{
    \mathcal{F}_{\lambda'}\left(\alpha b^{\lambda-\lambda'},\frac{k}{m}\right)
  }
  .
\end{displaymath}
By Lemma \ref{lemma:sum-function-fractions-gcd}, we deduce that
\begin{equation}\label{eq:intermediate-S_I-individual}
  S_I(\alpha,\lambda,\mu)
  \ll
    \sum_{m'<b^{\mu}}
    \frac{1}{m'}
    \sum_{\substack{0\leq k'<m'\\\gcd(k',m')=1}}
  \sum_{0\leq \lambda'\leq \lambda}
  b^{\lambda'}
  \abs{
  \mathcal{F}_{\lambda'}\left(\alpha b^{\lambda-\lambda'},\frac{k'}{m'}\right)
  }
  .
\end{equation}

\medskip
\subsection{An individual bound for $S_{I}(\alpha,\lambda,\mu)$ for $\alpha =h/d$
  with $d$ small}~

In this part we assume that $\alpha =h/d$, with $1 \leq h < d$ and
\begin{math}
  (b^2-1) b^\lambda h \not\equiv 0 \bmod d
  .
\end{math}

By splitting $b$-adically the sum over $m'$, we get
\begin{equation}\label{eq:maj_SI_alpha_M1}
  S_{I}(\alpha,\lambda,\mu)
  \ll
  \sum_{0\leq v \leq \mu}
  \sum_{0\leq \lambda'\leq \lambda}
  b^{\lambda'}
  \frac{1}{b^v}
  M_1\left(\alpha b^{\lambda-\lambda'},\lambda',v\right)
  ,
\end{equation}
where
\begin{displaymath}
  M_1(\alpha',\lambda',v)
  =
  \sum_{b^{v-1}<m\leq b^{v}}
  \sum_{\substack{0\leq k<m\\\gcd(k,m)=1}}
  \abs{
    \mathcal{F}_{\lambda'}\left(\alpha',\frac{k}{m}\right)
  }
  .
\end{displaymath}
We split the right-hand side of \eqref{eq:maj_SI_alpha_M1} as
$C_1(\alpha)+C_2(\alpha)$ where, in $C_1(\alpha)$, $2v \geq \lambda'$
and in $C_2(\alpha)$, $2v<\lambda'$.
In $C_1 (\alpha)$, since the sums on $k$ and $m$ are long, it will be
sufficient to apply Lemma~\ref{lemma:L1-mean-value-in-alpha} to take
advantage of the $L^1$ mean value associated to
$\mathcal{F}_{\lambda'}$.  For $C_2 (\alpha)$, the sums on $k$ and $m$
are short.  Using the product formula~\eqref{eq:FT-splitting-product},
we will combine Lemma~\ref{lemma:L1-mean-value-in-alpha} with the
individual bound given in
Lemma~\ref{lemma:uniform-upperbound-of-F-for-denominator-d}, which can
successfully be applied since the denominators of the fractions
$\alpha$ are small.

If $2v \geq \lambda'$ then, since the fractions $\frac{k}{m}$ are $b^{-2v}$
well spaced modulo $1$, it follows
from Lemma~\ref{lemma:L1-mean-value-in-alpha} that
\begin{displaymath}
  M_1\left(\alpha b^{\lambda-\lambda'},\lambda',v\right)
  \ll
  b^{2v}  b^{(\eta_b-1)\lambda'},
\end{displaymath}
hence
\begin{displaymath}
  C_1(\alpha)
  \ll
  \sum_{0\leq v \leq \mu}
  \sum_{0\leq \lambda'\leq 2v}
  b^{\lambda'}
  \frac{1}{b^v}
  b^{2v} b^{(\eta_b-1)\lambda'}
  \ll
  \sum_{0\leq v \leq \mu}
  b^{v + 2\eta_b v}
  \ll
  b^{(1+2\eta_b)\mu}.
\end{displaymath}

If $2v < \lambda'$ then by~\eqref{eq:FT-splitting-product}, we have
for any $\vartheta \in \R$,
\begin{displaymath}
  \abs{\mathcal{F}_{\lambda'}\left(\alpha b^{\lambda-\lambda'},\vartheta\right)}
  =
  \abs{\mathcal{F}_{2v}\left(\alpha b^{\lambda-2v},\vartheta\right)}
  \abs{\mathcal{F}_{\lambda'-2v}\left(\alpha b^{\lambda-\lambda'},\vartheta b^{2v}\right)}.
\end{displaymath}
For $2v+2 \leq \lambda' \leq \lambda$, since
\begin{math}
  b^{\lambda'-2v-2}(b^2-1)hb^{\lambda-\lambda'} = b^{\lambda-2-2v}(b^2-1)h \not\equiv 0 \bmod d,
\end{math}
we obtain by Lemma~\ref{lemma:uniform-upperbound-of-F-for-denominator-d}:
\begin{displaymath}
  \abs{
    \mathcal{F}_{\lambda'-2v}\left(\frac{h}{d} b^{\lambda-\lambda'},\vartheta b^{2v}\right)
  }
  \ll
  b^{-\Upsilon_b \frac{\lambda'-2v}{\log d}}
  ,
\end{displaymath}
where $\Upsilon_b$ is defined by \eqref{eq:definition-Upsilon_b}.
Since $\abs{\mathcal{F}_1(\cdot,\cdot)}\leq 1$, the inequality above
is also true for $\lambda'=2v+1$.  Denoting
\begin{equation}\label{eq:definition-Upsilon'_b}
  d'
  =
  \max\left( d, \Upsilon'_b \right)
  \quad
  \text{with}
  \quad
  \Upsilon'_b
  =
  \exp\left(\frac{4\Upsilon_b}{1-2\eta_b}\right)
  ,
\end{equation}
we deduce that for $2v < \lambda'$,
\begin{displaymath}
  \abs{
    \mathcal{F}_{\lambda'-2v}\left(\frac{h}{d} b^{\lambda-\lambda'},\vartheta b^{2v}\right)
  }
  \ll
  b^{-\Upsilon_b \frac{\lambda'-2v}{\log d'}}.
\end{displaymath}
Moreover, by
Lemma~\ref{lemma:L1-mean-value-in-alpha}, we have
\begin{equation}\label{eq:SG_m_I}
  \sum_{b^{v-1}<m\leq b^{v}}
  \sum_{\substack{0\leq k<m\\\gcd(k,m)=1}}
  \abs{\mathcal{F}_{2v}\left(\alpha b^{\lambda-2v},\frac{k}{m}\right)}
  \ll
  b^{2v\eta_b}.
\end{equation}
It follows that if $2v < \lambda'$ then
\begin{displaymath}
  M_1(\alpha b^{\lambda-\lambda'},\lambda',v)
  \ll
  b^{-\Upsilon_b \frac{\lambda'-2v}{\log d'}}
  b^{2v\eta_b}
  =
  b^{-\Upsilon_b \frac{\lambda'}{\log d'}}
  \,
  b^{\left(\frac{2\Upsilon_b}{\log d'}+2\eta_b\right)v}
  ,
\end{displaymath}
hence, since by \eqref{eq:definition-Upsilon'_b}
we have
\begin{math}
  1-2\eta_b-\frac{2\Upsilon_b}{\log d'} \geq \frac{1-2\eta_b}2
\end{math}
and
\begin{math}
  1-\frac{\Upsilon_b}{\log d'}
  \geq
  1-\frac{1-2\eta_b}{4}
  \geq \frac{3}{4},
\end{math}
we get
\begin{displaymath}
  C_2(\alpha)
  \ll
  \sum_{0\leq v \leq \mu}
  \sum_{2v< \lambda'\leq \lambda}
  b^{\left(1-\frac{\Upsilon_b}{\log d'}\right)\lambda'}
  \,
  b^{-\left(1-2\eta_b-\frac{2\Upsilon_b}{\log d'}\right)v}
  \ll
  b^{\left(1-\frac{\Upsilon_b}{\log d'}\right)\lambda}
  .
\end{displaymath}
Combining the previous upper bound for $C_1(\alpha)$ and
$C_2(\alpha)$, we conclude that
\begin{equation}\label{eq:maj_S_I_ind}
  S_{I}(\alpha,\lambda,\mu)
  \ll
   b^{(1+2\eta_b)\mu} + b^{\left(1-\frac{\Upsilon_b}{\log d'}\right)\lambda}.
\end{equation}
This permits us to formulate the following result.
\begin{lemma}\label{lemma:conclusion-S_I-individual}
  Let $0<\beta_1 < (1+2\eta_b)^{-1}$. If
  $1\leq \mu \leq \beta_1\lambda$ then for any
  $\alpha =h/d$ with
  \begin{math}
    (b^2-1) b^\lambda h \not\equiv 0 \bmod d,
  \end{math}
  we have 
  \begin{equation}\label{eq:conclusion-S_I-individual}
    S_{I}(\alpha,\lambda,\mu)
    \ll
    b^{\beta_1(1+2\eta_b)\lambda}
    +
    b^{\left(1-\frac{\Upsilon_b}{\log \max(d,\Upsilon'_b)}\right)\lambda}
  \end{equation}
  where
  $\Upsilon'_b$ is defined by \eqref{eq:definition-Upsilon'_b}.
\end{lemma}
\begin{proof}
  This immediately follows from~\eqref{eq:maj_S_I_ind}.
\end{proof}

\medskip
\subsection{Bound of $S_{I}(\alpha,\lambda,\mu)$ on average over $\alpha$}~

Let $\Delta \geq 1$ be an integer.
Summing  \eqref{eq:intermediate-S_I-individual}
on the $\alpha=h/d$
for admissible pairs $(h,d)$
with $1\le h < d\le b^\Delta$, we have
\begin{multline*}
  \sum_{\substack{d\leq b^\Delta\\ \gcd(d,b(b^2-1))=1}}
  \frac{1}{d}
  \sum_{\substack{1\leq h < d\\ \gcd(h,d)=1}}
  S_{I}\left( \frac{h}{d},\lambda,\mu\right)
  \\
  \ll
  \sum_{0\leq \lambda'\leq \lambda}
  b^{\lambda'}
  \hspace{-1.5em}
  \sum_{\substack{d\leq b^\Delta\\ \gcd(d,b(b^2-1))=1}}
  \frac{1}{d}
  \sum_{\substack{1\leq h < d\\ \gcd(h,d)=1}}
  \sum_{m'<b^{\mu}}
  \frac{1}{m'}
  \sum_{\substack{0\leq k'<m'\\\gcd(k',m')=1}}
  \abs{
    \mathcal{F}_{\lambda'}\left(
      \frac{h}{d} b^{\lambda-\lambda'},\frac{k'}{m'}
    \right)
  }
  \\
  =
  \sum_{0\leq \lambda'\leq \lambda}
  b^{\lambda'}
  \sum_{\substack{d\leq b^\Delta\\ \gcd(d,b(b^2-1))=1}}
  \frac{1}{d}
  \sum_{\substack{1\leq h < d\\ \gcd(h,d)=1}}
  \sum_{m<b^{\mu}}
  \frac{1}{m}
  \sum_{\substack{0\leq k<m\\\gcd(k,m)=1}}
  \abs{
    \mathcal{F}_{\lambda'}\left(\frac{h}{d},\frac{k}{m}\right)
  }
  .
\end{multline*}
By splitting $b$-adically the sums over $d$ and $m$
we get
\begin{multline}\label{eq:maj_typeI_ini}
  \sum_{\substack{d\leq b^\Delta\\ \gcd(d,b(b^2-1))=1}}
  \frac{1}{d}
  \sum_{\substack{1\leq h < d\\ \gcd(h,d)=1}}
  S_{I}\left( \frac{h}{d} ,\lambda,\mu\right)
  \\
  \ll
  \sum_{0\leq u \leq \Delta}
  \sum_{0\leq v \leq \mu}
  \sum_{0\leq \lambda'\leq \lambda}
  b^{\lambda'}
  \frac{1}{b^{u}}
  \frac{1}{b^{v}}
  M_2(\lambda',u,v)
\end{multline}
where
\begin{displaymath}
  M_2(\lambda',u,v)
  =
  \sum_{\substack{b^{u-1}< d\leq b^{u}\\ \gcd(d,b(b^2-1))=1}}
  \sum_{\substack{1\leq h < d\\ \gcd(h,d)=1}}
  \sum_{b^{v-1}<m\leq b^{v}}
  \sum_{\substack{0\leq k<m\\\gcd(k,m)=1}}
  \abs{
    \mathcal{F}_{\lambda'}\left(\frac{h}{d},\frac{k}{m}\right)
  }
  .
\end{displaymath}

We will use several times the following property.
By \eqref{eq:FT-splitting-product},
for any $0 \leq \lambda'' \leq \lambda'$
we have
\begin{displaymath}
  \abs{\mathcal{F}_{\lambda'}\left(\alpha,\frac{k}{m}\right)}
  \leq
  \abs{\mathcal{F}_{\lambda''}\left(\alpha b^{\lambda'-\lambda''},\frac{k}{m}\right)}
\end{displaymath}
and since the fractions $\frac{k}{m}$ are $b^{-2v}$
well spaced modulo $1$,
by Lemma~\ref{lemma:L1-mean-value-in-alpha}
we obtain,
for any $0 \leq \lambda'' \leq \lambda'$
and any $\alpha\in \R$, 
\begin{equation}
  \label{eq:SG_m}
   \sum_{b^{v-1}<m\leq b^{v}}
  \sum_{\substack{0\leq k<m\\\gcd(k,m)=1}}
  \abs{
    \mathcal{F}_{\lambda'}\left(\alpha,\frac{k}{m}\right)
  }
  \ll
  \left(b^{2v} + b^{\lambda''}\right) b^{(\eta_b-1)\lambda''}.
\end{equation}

Let
\begin{displaymath}
  \mathcal{S}
  =
  \{ 0,\ldots,\Delta \}
  \times
  \{ 0,\ldots,\mu \}
  \times
  \{ 0,\ldots,\lambda \}
\end{displaymath}
be the range of summation in the right hand side of
\eqref{eq:maj_typeI_ini}. We have
\begin{displaymath}
  \mathcal{S}
  =
  \mathcal{S}_1
  \cup
  \mathcal{S}_2
  \cup
  \mathcal{S}_3
  \cup
  \mathcal{S}_4
  \cup
  \mathcal{S}_5
\end{displaymath}
with
\begin{align*}
  \mathcal{S}_1
  &=
    \left\{
    (u,v,\lambda') \in \mathcal{S}:\
    \max(u,v) \geq \lambda'/2
    \right\}
  \\
  \mathcal{S}_2
  &=
    \left\{
    (u,v,\lambda') \in \mathcal{S}:\
    \max(u,v) < \lambda'/2
    \text{ and }
    \lambda'\leq \lambda_1
    \right\}
  \\
  \mathcal{S}_3
  &=
    \left\{
    (u,v,\lambda') \in \mathcal{S}:\
     v_0 \leq \max(u,v) < \lambda'/2
    \text{ and } c\cdot \min(u,v) \leq \max(u,v)
    \right\}
  \\
  \mathcal{S}_4
  &=
    \left\{
    (u,v,\lambda') \in \mathcal{S}:\
    \max(u,v) < v_0
    \right\}
  \\
  \mathcal{S}_5
  &=
    \big\{
    (u,v,\lambda') \in \mathcal{S}:\
    v_0 \leq \max(u,v) < \lambda'/2
  \\
  &  \hspace{45mm}
    \text{ and }
    c\cdot \min(u,v) > \max(u,v)
    \text{ and }
    \lambda'\geq \lambda_1
    \big\}
\end{align*}
where $\lambda_1$, $v_0$ and $c\geq 1$
are parameters to be chosen in Section \ref{Conclusion-Type-I}.
Notice that the sets $\mathcal{S}_i$ are not disjoint.

For $i\in\{1,2,3,4,5\}$, let
\begin{displaymath}
  C_i
  =
  \sum_{(u,v,\lambda')\in \mathcal{S}_i}
  b^{\lambda'}
  \frac{1}{b^{u}}
  \frac{1}{b^{v}}
  M_2(\lambda',u,v)
  .
\end{displaymath}
It follows from \eqref{eq:maj_typeI_ini} that
\begin{equation}\label{eq:majoration-by-C_1+C_2+C_3+C_4+C_5}
    \sum_{\substack{d\leq b^\Delta\\ \gcd(d,b(b^2-1))=1}}
  \frac{1}{d}
  \sum_{\substack{1\leq h < d\\ \gcd(h,d)=1}}
  S_{I}\left( \frac{h}{d} ,\lambda,\mu \right)
  \ll
  C_1+C_2+C_3+C_4+C_5
  .
\end{equation}
In $C_1$, $\max (u,v)$ is large and we will
use the bound \eqref{eq:SG_m}.  This formula will also be applied for
$C_2$ devoted to the small $\lambda'$ and for $C_3$ where one of the
parameters $u,v$ is large and the other is very small.  In $C_4$ both
$u$ and $v$ will be very small and we will apply
Lemma~\ref{lemma:uniform-upperbound-of-F-for-denominator-d}.  The most
difficult part is the remaining part $C_5$ where
$\max (u,v)<\lambda'/2$ but with no very small parameter $u$ or
$v$. We will handle this part in a quite similar way as in
\cite{DMRSS_reversible_primes} (which is devoted to the base $2$).
In this range it is necessary to exploit the two sequences of
fractions $h/d$ and $k/m$, by combining
the Cauchy-Schwarz inequality
with the different individual and mean values estimates for the functions
$\mathcal{F}_\lambda$ and $G_\lambda$
obtained in Section~\ref{section:fourier-transform}.

\medskip
\subsubsection{Estimate of $C_1$}~\label{subsection:C1}
\smallskip

If $u\leq v$ then we choose $\lambda'' = \lambda'$ in \eqref{eq:SG_m}
and since $2v \geq \lambda'$, we obtain
\begin{displaymath}
  M_2(\lambda',u,v)
  \ll
  b^{2u} b^{2v} b^{(\eta_b-1)\lambda'}.
\end{displaymath}
By exchanging the roles of the points $\frac{h}{d}$ and $\frac{k}{m}$,
we see that this upper bound remains true if $v<u$.

It follows that 
\begin{align*}
  C_1
  &\ll
    \sum_{0\leq u \leq \Delta}
    \sum_{0\leq v \leq \mu}
     \sum_{0\leq \lambda'\leq 2\max(u,v)}
    b^{\lambda'}
    \frac{1}{b^{u}}
    \frac{1}{b^{v}}
    b^{2u} b^{2v} b^{(\eta_b-1)\lambda'}\\
  &\ll
    \sum_{0\leq u \leq \Delta}
    \sum_{0\leq v \leq \mu}
    b^{u}b^{v}
    \sum_{0\leq \lambda'\leq 2\max(u,v)}
    b^{\eta_b \lambda'}\\
  &\ll
    \sum_{0\leq u \leq \Delta}
    \sum_{0\leq v \leq \mu}
    b^{u}b^{v}
    b^{2\eta_b\max(u,v)}
  \ll
    b^{\Delta}b^{\mu}
    b^{2\eta_b\max(\Delta,\mu)}
    .
\end{align*}

\medskip
\subsubsection{Estimate of $C_2 $}\label{sec:C2}~
\smallskip

If $u\leq v$ then we apply \eqref{eq:SG_m} with
$\lambda''=2v$ (which is possible since $v<\lambda'/2$) and we obtain
\begin{displaymath}
  M_2(\lambda',u,v)
  \ll
  b^{2u} b^{2v} b^{(\eta_b-1)2v} = b^{2u}b^{2\eta_b v}
  .
\end{displaymath}
If $u > v$ then, by exchanging the roles of the points $\frac{h}{d}$
and $\frac{k}{m}$, we obtain
\begin{displaymath}
  M_2(\lambda',u,v)
  \ll
  b^{2v} b^{2u} b^{(\eta_b-1)2u} = b^{2v}b^{2\eta_b u}
  .
\end{displaymath}
It follows that 
\begin{align*}
  C_2
  &\ll
    \sum_{0\leq u \leq \Delta}
    \sum_{0\leq v \leq \mu}
    \sum_{0\leq \lambda'\leq \lambda_1}
    b^{\lambda'}
    \frac{1}{b^{u}}
    \frac{1}{b^{v}}
    b^{2\min(u,v)}b^{2\eta_b\max(u,v)}\\
  &\ll
    b^{\lambda_1}
    \left(
    \sum_{0\leq u \leq \min(\Delta,\mu)}
    \sum_{u\leq v \leq \mu}
    b^{u}
    \frac{1}{b^{v(1-2\eta_b)}}
    +
    \sum_{0\leq v \leq \min(\Delta,\mu)}
    \sum_{v\leq u \leq \Delta}
    \frac{1}{b^{u(1-2\eta_b)}}
    b^{v}
    \right)\\
  &\ll
    b^{\lambda_1}
    \left(
    \sum_{0\leq u \leq \min(\Delta,\mu)}
    b^{2u\eta_b}
    +
    \sum_{0\leq v \leq \min(\Delta,\mu)}
    b^{2v\eta_b}
    \right)\\
  &\ll
    b^{\lambda_1} b^{2\min(\Delta,\mu)\eta_b} 
    .
\end{align*}

\medskip
\subsubsection{Estimate of $C_3$}~
\label{subsection:C3}\smallskip

By using the bounds for $M_2(\lambda',u,v)$ of Section~\ref{sec:C2}, we obtain
\begin{align*}
  C_3
  &\ll
    \sum_{\substack{u,v\\0\leq u \leq \Delta\\ 0\leq v \leq \mu\\\max(u,v)\geq v_0\\c\cdot \min(u,v)\leq \max(u,v)}}
  \sum_{0\leq \lambda'\leq \lambda}
  b^{\lambda'}
  \frac{1}{b^{u}}
  \frac{1}{b^{v}}
  b^{2\min(u,v)}b^{2\eta_b\max(u,v)}\\
  &\ll
    b^{\lambda}
    \left(
    \sum_{v_0\leq v\leq \mu}
    \sum_{0\leq u \leq \frac{v}{c}}
    b^{u}
    \frac{1}{b^{v(1-2\eta_b)}}
    +
    \sum_{v_0\leq u \leq \Delta}
    \sum_{0\leq v\leq \frac{u}{c}}
    \frac{1}{b^{u(1-2\eta_b)}}
    b^{v}
    \right)\\
  &\ll
    b^{\lambda}
    \left(
    \sum_{v_0\leq v\leq \mu}
    \frac{1}{b^{v(1-2\eta_b-\frac{1}{c})}}
    +
    \sum_{v_0\leq u \leq \Delta}
    \frac{1}{b^{u(1-2\eta_b-\frac{1}{c})}}
    \right).
\end{align*}
It follows that if $1-2\eta_b-\frac{1}{c}>0$ then
\begin{displaymath}
  C_3 \ll_c \frac{b^{\lambda}}{b^{v_0(1-2\eta_b-\frac{1}{c})}}.
\end{displaymath}

\medskip
\subsubsection{Estimate of $C_4$}~ \label{subsection:C4}
\smallskip

By Lemma~\ref{lemma:uniform-upperbound-of-F-for-denominator-d},
using $d\leq b^u \leq b^{v_0}$, we have
\begin{displaymath}
  M_2(\lambda',u,v)
  \ll
  b^{2u}
  b^{2v}
  b^{-\Upsilon_b \frac{\lambda'}{\log b^{v_0}}},
\end{displaymath}
hence, if $1- \frac{\Upsilon_b}{\log b^{v_0}}\geq \frac12$ then
\begin{align*}
  C_4
  &\ll
  \sum_{0\leq u < v_0}
  \sum_{0\leq v < v_0}
  \sum_{0\leq \lambda'\leq \lambda}
  b^{\lambda'}
  \frac{1}{b^{u}}
  \frac{1}{b^{v}}
  b^{2u}
  b^{2v}
  b^{-\Upsilon_b \frac{\lambda'}{\log b^{v_0}}}\\
  &\ll
  \sum_{0\leq \lambda'\leq \lambda}
    b^{\lambda'\left(1-\frac{\Upsilon_b}{\log b^{v_0}}\right)}
    \sum_{0\leq u < v_0} b^{u}
    \sum_{0\leq v < v_0} b^v\\
  &\ll
    b^{\lambda\left(1-\frac{\Upsilon_b}{\log b^{v_0}}\right)}
    b^{2v_0}
    = b^{\lambda} b^{2v_0-\Upsilon_b \frac{\lambda}{\log b^{v_0}}}.
\end{align*}

\medskip
\subsubsection{Estimate of $C_5$}~ \smallskip

By \eqref{eq:FT-splitting-product}, for integers $\lambda'_1\geq 0$ and
$\lambda'_2\geq 0$ such that $\lambda'_1 + \lambda'_2 \leq \lambda'$ 
we have
\begin{displaymath}
  \abs{
  \mathcal{F}_{\lambda'}\left(
  \frac{h}{d},\frac{k}{m}\right)
  }  
  \leq
  \abs{
  \mathcal{F}_{\lambda'_1}\left(
  \frac{h}{d} b^{\lambda'-\lambda'_1},\frac{k}{m}\right)
  }  
  \abs{
  \mathcal{F}_{\lambda'_2}\left(
  \frac{h}{d},\frac{k}{m}  b^{\lambda'-\lambda'_2}\right)
}
.
\end{displaymath}
By Cauchy-Schwarz we get
\begin{displaymath}
  M_2(\lambda',u,v)
  \leq
  M_{21}(\lambda_1',u,v)^{1/2}
  \
  M_{22}(\lambda_2',u,v)^{1/2}
\end{displaymath}
where
\begin{multline*}
  M_{21}(\lambda_1',u,v)
  \\
  =
  \sum_{\substack{b^{u-1}< d\leq b^{u}\\ \gcd(d,b(b^2-1))=1}}
  \sum_{\substack{1\leq h < d\\ \gcd(h,d)=1}}
  \sum_{b^{v-1}<m\leq b^{v}}
  \sum_{\substack{0\leq k<m\\\gcd(k,m)=1}}
  \abs{
  \mathcal{F}_{\lambda'_1}\left(
  \frac{h}{d} b^{\lambda'-\lambda'_1},\frac{k}{m}\right)
  }^2
\end{multline*}
and
\begin{multline*}
  M_{22}(\lambda_2',u,v)
  \\
  =
  \sum_{\substack{b^{u-1}< d\leq b^{u}\\ \gcd(d,b(b^2-1))=1}}
  \sum_{\substack{1\leq h < d\\ \gcd(h,d)=1}}
  \sum_{b^{v-1}<m\leq b^{v}}
  \sum_{\substack{0\leq k<m\\\gcd(k,m)=1}}
  \abs{
    \mathcal{F}_{\lambda'_2}\left(
      \frac{h}{d},\frac{k}{m}  b^{\lambda'-\lambda'_2}\right)
  }^2
  .
\end{multline*}
Let us first consider $M_{21}(\lambda_1',u,v)$.
By \eqref{eq:FT-splitting-product}
and Lemma \ref{lemma:pointwise-upperbound-of-F},
we have for $\lambda'_{11}+\lambda'_{12} \leq \lambda'_{1}$,
\begin{align*}
  &\abs{
  \mathcal{F}_{\lambda'_1}\left(
  \frac{h}{d} b^{\lambda'-\lambda'_1},\frac{k}{m}\right)
    }
    \\
  &\leq
    \abs{
    \mathcal{F}_{\lambda'_{11}}\left(
    \frac{h}{d} b^{\lambda'-\lambda'_1} b^{\lambda'_{1}-\lambda'_{11}},\frac{k}{m}\right)
    }
    \abs{
    \mathcal{F}_{\lambda'_{12}}\left(
    \frac{h}{d} b^{\lambda'-\lambda'_1},\frac{k}{m} b^{\lambda'_{1}-\lambda'_{12}}\right)
    }
  \\
  &
    \leq
    \abs{
    \mathcal{F}_{\lambda'_{11}}\left(
    \frac{h}{d} b^{\lambda'-\lambda'_{11}},\frac{k}{m}
    \right)
    }
    G_{\lambda'_{12}-1}^{1/2}\left(
    \frac{h}{d} b^{\lambda'-\lambda'_1} (b^2-1)\right)
    ,
\end{align*}
hence
\begin{multline*}
  M_{21}(\lambda_1',u,v)
  \leq
  \sum_{\substack{b^{u-1}< d\leq b^{u}\\ \gcd(d,b(b^2-1))=1}}
  \sum_{\substack{1\leq h < d\\ \gcd(h,d)=1}}
  G_{\lambda'_{12}-1}\left( \frac{h}{d} b^{\lambda'-\lambda'_1}
    (b^2-1) \right)
  \\
  \sum_{b^{v-1}<m\leq b^{v}}
  \sum_{\substack{0\leq k<m\\\gcd(k,m)=1}}
  \abs{ \mathcal{F}_{\lambda'_{11}}\left(
      \frac{h}{d} b^{\lambda'-\lambda'_{11}},\frac{k}{m}\right) }^2
\end{multline*}
By the large sieve inequality (Lemma \ref{lemma:large-sieve}) this gives
\begin{displaymath}
  M_{21}(\lambda_1',u,v)
  \leq
  \left(b^{2v} + b^{\lambda'_{11}} \right)
  b^{-\lambda'_{11}}
  \sum_{\substack{b^{u-1}< d\leq b^{u}\\ \gcd(d,b(b^2-1))=1}}
  \sum_{\substack{1\leq h < d\\ \gcd(h,d)=1}}
  G_{\lambda'_{12}-1}\left( \frac{h}{d} \right).
\end{displaymath}
We apply the Sobolev-Gallagher inequality given in
Lemma~\ref{lemma:sobolev-gallagher} and the
bounds of the $L_1$-norms of the functions $G_{\lambda'}$
and $\left(G_{\lambda'}\right)'$ 
provided by Lemma \ref{lemma:Lkappa-norm-of-G}
and Lemma \ref{lemma:L_1_norm_of-G'}. We obtain
\begin{displaymath}
  M_{21}(\lambda_1',u,v)
  \ll
  \left(b^{2v}b^{-\lambda'_{11}} + 1 \right)
  \left( b^{2u} + b^{\lambda'_{12}} \right)
  b^{-\zeta_{b,1}\lambda'_{12}}
  ,
\end{displaymath}
where $\zeta_{b,1}$ is defined by \eqref{eq:definition-zeta_b}.
If the size of $\lambda'_1$ permits to choose
\begin{displaymath}
  \lambda'_{11} \geq  2v,\quad
  \lambda'_{12} \leq 2u,
\end{displaymath}
we obtain
\begin{displaymath}
  M_{21}(\lambda_1',u,v)
    \ll
  b^{2u-\zeta_{b,1}\lambda'_{12}}
   .
\end{displaymath}

Let us now consider $M_{22}(\lambda_2',u,v)$.
By \eqref{eq:FT-splitting-product}
and Lemma \ref{lemma:pointwise-upperbound-of-F},
we have for $\lambda'_{21}+\lambda'_{22} \leq \lambda'_{2}$:
\begin{align*}
  &
  \abs{
  \mathcal{F}_{\lambda'_2}\left(
  \frac{h}{d},\frac{k}{m}  b^{\lambda'-\lambda'_2} \right)
    }
  \\
  &
  \leq
  \abs{
  \mathcal{F}_{\lambda'_{21}}\left(
  \frac{h}{d} b^{\lambda'_2-\lambda'_{21}},\frac{k}{m} b^{\lambda'-\lambda'_2} \right)
  }
  \abs{
  \mathcal{F}_{\lambda'_{22}}\left(
  \frac{h}{d},\frac{k}{m}  b^{\lambda'-\lambda'_2} b^{\lambda'_2-\lambda'_{22}}\right)
  }
  \\
  &
  \leq
  G_{\lambda'_{21}-1}^{1/2}\left(
  \frac{k}{m} b^{\lambda'-\lambda'_2} (b^2-1) \right)
  \abs{
  \mathcal{F}_{\lambda'_{22}}\left(
  \frac{h}{d},\frac{k}{m}  b^{\lambda'-\lambda'_{22}}\right)
  }
\end{align*}
hence
\begin{multline*}
  M_{22}(\lambda_2',u,v)
  \leq
  \sum_{b^{v-1}<m\leq b^{v}}
  \sum_{\substack{0\leq k<m\\\gcd(k,m)=1}}
  G_{\lambda'_{21}-1}\left( \frac{k}{m} b^{\lambda'-\lambda'_2}
    (b^2-1) \right)
  \\
  \sum_{\substack{b^{u-1}< d\leq b^{u}\\ \gcd(d,b(b^2-1))=1}}
  \sum_{\substack{1\leq h < d\\ \gcd(h,d)=1}}
  \abs{ \mathcal{F}_{\lambda'_{22}}\left( \frac{h}{d},\frac{k}{m}
      b^{\lambda'-\lambda'_{22}}\right) }^2 .
\end{multline*}
By the large sieve inequality (Lemma \ref{lemma:large-sieve}) this gives
\begin{multline*}
  M_{22}(\lambda_2',u,v)
  \\
  \leq
  \left(b^{2u} + b^{\lambda'_{22}}\right)
  b^{-\lambda'_{22}}
  \sum_{b^{v-1}<m\leq b^{v}}
  \sum_{\substack{0\leq k<m\\\gcd(k,m)=1}}
  G_{\lambda'_{21}-1}\left(
    \frac{k}{m} b^{\lambda'-\lambda'_2} (b^2-1)
  \right)
  .
\end{multline*}
We observe that the denominators $m$ are not always
coprime to $b(b^2-1)$.
By filtering by
\begin{math}
  d_0= \gcd(b^{\lambda'-\lambda'_2} (b^2-1), m)
\end{math}
we get
\begin{align*}
  &
    \sum_{b^{v-1}<m\leq b^{v}}
    \sum_{\substack{0\leq k<m\\\gcd(k,m)=1}}
  G_{\lambda'_{21}-1}\left(
  \frac{k}{m} b^{\lambda'-\lambda'_2} (b^2-1)
  \right)
  \\
  &
    =
    \sum_{\substack{d_0\dv b^{\lambda'-\lambda'_2} (b^2-1)\\ d_0 \leq b^v}}
  \hspace{-1em}
  \sum_{\substack{\frac{b^{v-1}}{d_0}<m'\leq \frac{b^{v}}{d_0}\\ \gcd\left(m',\frac{b^{\lambda'-\lambda'_2} (b^2-1)}{d_0}\right)=1}}
  \hspace{-1em}
  \sum_{\substack{0\leq k < d_0 m'\\ \gcd(k,d_0 m')=1}}
  \hspace{-1em}
  G_{\lambda'_{21}-1}\left(
    \frac{k}{m'}\cdot \frac{b^{\lambda'-\lambda'_2} (b^2-1)}{d_0} \right)
  \\
  &\leq
    \sum_{\substack{d_0\dv b^{\lambda'-\lambda'_2} (b^2-1)\\ d_0 \leq b^v}}
  d_0
  \sum_{\substack{\frac{b^{v-1}}{d_0}<m'\leq \frac{b^{v}}{d_0}\\ \gcd\left(m',\frac{b^{\lambda'-\lambda'_2} (b^2-1)}{d_0}\right)=1}}
  \sum_{\substack{0\leq k < m'\\ \gcd(k,m')=1}}
  G_{\lambda'_{21}-1}\left( \frac{k}{m'}\right)
  .
\end{align*}
We apply the Sobolev-Gallagher inequality given in
Lemma~\ref{lemma:sobolev-gallagher} and the
bounds of the $L_1$-norms of the functions $G_{\lambda'}$
and $\left(G_{\lambda'}\right)'$ 
provided by Lemma \ref{lemma:Lkappa-norm-of-G}
and Lemma \ref{lemma:L_1_norm_of-G'}. We obtain,
using Lemma \ref{lemma:sigma-z},
\begin{align*}
  &
    \sum_{b^{v-1}<m\leq b^{v}}
  \sum_{\substack{0\leq k<m\\\gcd(k,m)=1}}
  G_{\lambda'_{21}-1}\left(
    \frac{k}{m} b^{\lambda'-\lambda'_2} (b^2-1)
  \right)
  \\
  &\ll
     \sum_{\substack{d_0\dv b^{\lambda'-\lambda'_2} (b^2-1)\\ d_0 \leq b^v}}
  d_0
  \left(\left(\frac{b^{v}}{d_0}\right)^2 + b^{\lambda'_{21}}\right)
  b^{-\zeta_{b,1} \lambda'_{21}}\\
  &\ll
    \left(b^{2v} + b^{\lambda'_{21}} b^{\min(v,\lambda'-\lambda_2')} \lambda'^{\omega(b)} \right)b^{-\zeta_{b,1} \lambda'_{21}}
    \ll \lambda'^{\omega(b)} \left(b^{2v} + b^{\lambda'_{21}} b^{v} \right)b^{-\zeta_{b,1} \lambda'_{21}}
    .
\end{align*}
If the size of $\lambda'_2$ permits to choose
\begin{displaymath}
  \lambda'_{22} \geq 2u,\quad
  \lambda'_{21} \leq v,
\end{displaymath}
we obtain
\begin{displaymath}
  M_{22}(\lambda_2',u,v)
  \ll
  \lambda'^{\omega(b)}
  b^{2v-\zeta_{b,1}\lambda'_{21}}
  .
\end{displaymath}
Let $\varepsilon_1>0$.
If we assume that $(2+\varepsilon_1)(u+v)\leq \lambda'$ then we can
choose
\begin{displaymath}
  \lambda'_{11} = 2v,
  \quad
  \lambda'_{12} = \lfloor \varepsilon_1 u\rfloor
  \quad
  \lambda'_{22} = 2u,
  \quad
  \lambda'_{21} = \lfloor \varepsilon_1 v\rfloor,
\end{displaymath}
which leads to
\begin{align*}
  M_2(\lambda',u,v)
  &\ll
    \lambda'^{\frac{\omega(b)}{2}}
    \left(b^{2u-\zeta_{b,1}\lambda'_{12}}\right)^{1/2}
    \
    \left(b^{2v-\zeta_{b,1}\lambda'_{21}}\right)^{1/2}
    \\
  &\qquad
    =
    \lambda'^{\frac{\omega(b)}{2}}\
    b^{u+v-\frac{\zeta_{b,1}}{2}\left(\lambda'_{12}+\lambda'_{21}\right)}
  \ll
    \lambda'^{\frac{\omega(b)}{2}}\
    b^{(u+v)(1-\varepsilon_1\frac{\zeta_{b,1}}{2})}
    .
\end{align*}
If $(2+\varepsilon_1)(1+c)\min(\Delta,\mu) \leq \lambda_1$ then for
any $u\leq \Delta$, $v\leq \mu$ and $\lambda'\geq \lambda_1$ such that
$c\cdot \min(u,v)>\max(u,v)$, we have
\begin{displaymath}
  (2+\varepsilon_1)(u+v)
  \leq
  (2+\varepsilon_1)(c+1) \min(u,v)
  \leq
  (2+\varepsilon_1)(c+1) \min(\Delta,\mu)
  \leq \lambda_1
  \leq \lambda'.
\end{displaymath}
It follows that
if
$(2+\varepsilon_1)(1+c)\min(\Delta,\mu) \leq \lambda_1$ then
\begin{displaymath}
  C_5
  \ll
  \sum_{\substack{u,v\\0\leq u \leq \Delta\\0\leq v \leq \mu\\v_0\leq \max(u,v) < c\cdot \min(u,v)}}
  \sum_{\lambda_1\leq \lambda'\leq \lambda}
  b^{\lambda'}
  \frac{1}{b^{u}}
  \frac{1}{b^{v}}
  \lambda'^{\frac{\omega(b)}{2}}\
  b^{(u+v)(1-\varepsilon_1\frac{\zeta_{b,1}}{2})}
\end{displaymath}
so that 
if
$(2+\varepsilon_1)(1+c)\min(\Delta,\mu) \leq \lambda_1$ then
\begin{multline*}
  C_5
  \ll
    \lambda^{\frac{\omega(b)}{2}}\
    b^{\lambda}
    \Bigg(
    \sum_{\frac{v_0}{c} < u \leq \Delta}
    \sum_{u\leq v < c u} \frac{1}{b^{u\varepsilon_1\frac{\zeta_{b,1}}{2}}}
    \frac{1}{b^{v\varepsilon_1\frac{\zeta_{b,1}}{2}}}
  \\
  +
    \sum_{\frac{v_0}{c} < v \leq \mu}
    \sum_{v < u < cv} \frac{1}{b^{u\varepsilon_1\frac{\zeta_{b,1}}{2}}}
    \frac{1}{b^{v\varepsilon_1\frac{\zeta_{b,1}}{2}}}
    \Bigg)
\end{multline*}
hence
\begin{displaymath}
  C_5 \ll_{\varepsilon_1}
    \lambda^{\frac{\omega(b)}{2}}\
    b^{\lambda}
    \left(
    \sum_{\frac{v_0}{c} < u \leq \Delta}
    \frac{1}{b^{u\varepsilon_1\zeta_{b,1}}}
    +
    \sum_{\frac{v_0}{c} < v \leq \mu}
    \frac{1}{b^{v\varepsilon_1\zeta_{b,1}}}
    \right)
    \ll_{\varepsilon_1}
    b^{\lambda}
    \frac{\lambda^{\frac{\omega(b)}{2}}}{b^{\frac{v_0}{c}\varepsilon_1
    \zeta_{b,1}}}
    .
\end{displaymath}

\medskip
\subsubsection{Completion of the bound of $S_{I}(\alpha,\lambda,\mu)$ on average over $\alpha$}~ \label{Conclusion-Type-I}\smallskip

Combining \eqref{eq:majoration-by-C_1+C_2+C_3+C_4+C_5}
and the previous bounds on $C_i$, $i=1,\ldots,5$, we obtain
that for $c\geq 1$, $v_0\geq 1$, $\lambda_1\geq 1$
and $\varepsilon_1>0$
such that
\begin{displaymath}
  (i)\
  1-2\eta_b-\frac{1}{c}>0,
  \ \
  (ii)\
  1- \frac{\Upsilon_b}{\log b^{v_0}} \geq \frac12,
  \ \
  (iii)\
  \left(2+\varepsilon_1\right)(1+c)\min(\Delta,\mu) \leq \lambda_1
\end{displaymath}
we have
\begin{multline*}
  \sum_{\substack{d\leq b^\Delta\\ \gcd(d,b(b^2-1))=1}}
  \frac{1}{d}
  \sum_{\substack{1\leq h < d\\ \gcd(h,d)=1}}
  S_{I}\left( \frac{h}{d},\lambda,\mu \right)
  \\
  \ll_{c,\varepsilon_1}
  b^{\Delta}b^{\mu}
  b^{2\eta_b\max(\Delta,\mu)}
  +
  b^{\lambda_1} b^{2\eta_b\min(\Delta,\mu)}
  +
  \frac{b^{\lambda}}{b^{v_0(1-2\eta_b-\frac{1}{c})}}
  \\
    + b^{\lambda} b^{2v_0-\Upsilon_b \frac{\lambda}{\log b^{v_0}}}
    +b^{\lambda}
   \frac{\lambda^{\frac{\omega(b)}{2}}}{b^{\frac{v_0}{c}\varepsilon_1 \zeta_{b,1}}}.
\end{multline*}
Let us choose $c=15$.
By \eqref{eq:upperbound-eta-b} we have
\begin{math}
  \eta_b\leq \eta_3=\frac{\log  5}{\log 3}-1
\end{math}
so that
\begin{displaymath}
  1-2\eta_b-\frac{1}{c}
  \geq
  1-2\eta_3-\frac{1}{c}
  > 0,  
\end{displaymath}
hence (i) is satisfied.
We choose
\begin{displaymath}
  \lambda_1
  =\ceil{  \left(2+\varepsilon_1\right)(1+c)\min(\Delta,\mu)}
\end{displaymath}
so that (iii) is satisfied.
We also take
$v_0 = \delta_b \sqrt{\lambda}$ for some $\delta_b>0$
to be chosen later depending only on $b$.
For
\begin{math}
  \lambda \geq \left(\frac{2 \Upsilon_b}{\delta_b \log b}\right)^2
\end{math}
the condition (ii) is satisfied,
hence
\begin{multline*}
  \sum_{\substack{d\leq b^\Delta\\ \gcd(d,b(b^2-1))=1}}
  \frac{1}{d}
  \sum_{\substack{1\leq h < d\\ \gcd(h,d)=1}}
  S_{I}\left( \frac{h}{d},\lambda,\mu \right)
  \\
  \ll_{\varepsilon_1}
    b^{\Delta}b^{\mu} b^{2\eta_b\max(\Delta,\mu)}
    +
    b^{\min(\Delta,\mu)(2\eta_b+(2+\varepsilon_1)(1+c))}
    +
    b^{\lambda} b^{-\delta_b (1-2\eta_b-\frac{1}{c})\sqrt{\lambda}}
    \\
    + b^{\lambda} b^{-\left(\frac{\Upsilon_b}{\delta_b\log b}-2\delta_b\right)\sqrt{\lambda}}
    +b^{\lambda} 
    \lambda^{\frac{\omega(b)}{2}}b^{-\frac{\varepsilon_1 \zeta_{b,1}\delta_b}{c} \sqrt{\lambda}}.
\end{multline*}
We then choose
$\delta_b =
\left(\frac{\Upsilon_b}{4\log b}\right)^{1/2}>0$ so
that $\frac{\Upsilon_b}{\delta_b\log b}-2\delta_b>0$,
and
\begin{displaymath}
  \varepsilon_1
  = \frac{1-2\eta_b}{1+c}
  = \frac{1-2\eta_b}{16}
  > 0
  .
\end{displaymath}
We get for
\begin{math}
  \lambda
  \geq
  \left(\frac{2 \Upsilon_b}{\delta_b \log b}\right)^2
  ,
\end{math}
\begin{multline*}
  \sum_{\substack{d\leq b^\Delta\\ \gcd(d,b(b^2-1))=1}}
  \frac{1}{d}
  \sum_{\substack{1\leq h < d\\ \gcd(h,d)=1}}
  S_{I}\left( \frac{h}{d},\lambda,\mu \right)
  \\
  \ll
  b^{\Delta+\mu+2\eta_b\max(\Delta,\mu)}
  +
  b^{33 \min(\Delta,\mu)}
  +
  b^{\lambda-\delta_b (1-2\eta_b-\frac{1}{c})\sqrt{\lambda}}
  \\
  + b^{\lambda-\left(\frac{\Upsilon_b}{\delta_b\log b}-2\delta_b\right)\sqrt{\lambda}}
  +
  \lambda^{\frac{\omega(b)}{2}}
  b^{\lambda-\frac{\varepsilon_1 \zeta_{b,1}\delta_b}{c} \sqrt{\lambda}}.
\end{multline*}
We obtain that if
\begin{equation}\label{eq:condition-Delta-type-I}
  \Delta+\mu+2\eta_b\max(\Delta,\mu) \leq \lambda -\sqrt{\lambda},
  \quad
  33 \min(\Delta,\mu) \leq  \lambda -\sqrt{\lambda},
\end{equation}
then for some $\tilde{\delta}_{b}>0$ (depending only on $b$) and
\begin{math}
  \lambda
  \geq
  \left(\frac{2 \Upsilon_b}{\delta_b \log b}\right)^2
  ,
\end{math}
\begin{equation}\label{eq:S_I-final-upper-bound}
  \sum_{\substack{d\leq b^\Delta\\ \gcd(d,b(b^2-1))=1}}
  \frac{1}{d}
  \sum_{\substack{1\leq h < d\\ \gcd(h,d)=1}}
  S_{I}\left( \frac{h}{d},\lambda,\mu \right)
  \ll
  b^{\lambda-\tilde{\delta}_{b}\sqrt{\lambda}}
  .
\end{equation}

This leads us to the following conclusion.
\begin{lemma}\label{lemma:conclusion-S_I-sum-on-alpha}
  Let $0<\beta_1 < (1+2\eta_b)^{-1}$,
  \begin{equation}\label{eq:def_C''_b,beta1,beta2}
    C''(b,\beta_1)
    =
    \min\left(
      1-\beta_1(1+2\eta_b),\,
      \frac{1}{33}
    \right)
    >0
  \end{equation}
  and 
  \begin{math}
    \xi \in \left]0,C''(b,\beta_1)\right[.
  \end{math}
  There exists $\tilde{\delta}=\tilde{\delta}_{b} >0$ and
  $\lambda''_0(b,\beta_1,\xi)\geq 2$ such that if
  $\lambda\geq\lambda''_0(b,\beta_1,\xi)$ and
  $1\leq \mu \leq \beta_1\lambda$ then,
  denoting $D = b^{\xi \lambda}$,
  \begin{equation}\label{eq:conclusion-S_I-sum-on-alpha}
    \sum_{\alpha\in \mathcal{A}_D}
    W(\alpha)
    S_{I}(\alpha,\lambda,\mu)
    \ll
    b^{\lambda-\tilde{\delta}\sqrt{\lambda}}
    .
  \end{equation}
\end{lemma}
\begin{proof}
  Let us remark first that if
  \begin{equation}\label{eq:condition-Delta-type-I-simplified}
    \Delta+\mu(1+2\eta_b) \leq \lambda -\sqrt{\lambda},
    \quad
    33 \Delta \leq  \lambda -\sqrt{\lambda},
  \end{equation}
  then   \eqref{eq:condition-Delta-type-I} is satisfied
  in the case $\max(\Delta,\mu) = \mu$, and
  if $\max(\Delta,\mu) = \Delta$ we have
  \begin{displaymath}
    \Delta+\mu+2\eta_b\max(\Delta,\mu)
    \leq
    (2+2\eta_b)\Delta
    \leq 3 \Delta
    \leq \lambda -\sqrt{\lambda},    
  \end{displaymath}
  thus
  \eqref{eq:condition-Delta-type-I} is also satisfied.
  Therefore   \eqref{eq:condition-Delta-type-I-simplified}
  implies   \eqref{eq:condition-Delta-type-I}.
  
  Let
  \begin{math}
    \xi \in \left]0,C''(b,\beta_1)\right[.
  \end{math}
  Taking $\Delta = \floor{\xi\lambda}+1$, 
  there exists
  \begin{math}
    \lambda''_0(b,\beta_1,\xi)
    \geq
    \left(\frac{2 \Upsilon_b}{\delta_b \log b}\right)^2    
  \end{math}
  and large enough,
  such that for $\lambda\geq \lambda''_0(b,\beta_1,\xi)$,
  we have
  \begin{displaymath}
    \Delta+\beta_1\lambda(1+2\eta_b) \leq \lambda -\sqrt{\lambda},
    \quad
    33 \, \Delta \leq  \lambda -\sqrt{\lambda},
  \end{displaymath}
  hence the conditions~\eqref{eq:condition-Delta-type-I-simplified}
  and therefore \eqref{eq:condition-Delta-type-I} are satisfied
  for $\mu \leq \beta_1\lambda$.
  It follows that \eqref{eq:S_I-final-upper-bound} holds.
  Moreover, since $\Delta \geq \xi\lambda$,
  by \eqref{eq:definition-A-d} and \eqref{eq:def-W} we obtain
  \eqref{eq:conclusion-S_I-sum-on-alpha}.
\end{proof}

\section{Conclusion on exponential sums}\label{section:conclusion_expo_sums}

\subsection{Conclusion in the case $\alpha =h/d$ with $d$ small}

The goal of this section is to complete the proof of
Theorem~\ref{theorem:bound_expo_sum_conclusion}.

Let $\lambda\geq 2$ be an integer and $(h,d)\in\Z^2$ with $d\geq 2$ and
\begin{math}
  (b^2-1) b^\lambda h \not\equiv 0 \bmod d.
\end{math}
We intend to apply Lemma~\ref{lemma:maj_sum_alpha_vaughan} with
$\mathcal{A} = \{\alpha\}$ where $\alpha = \frac{h}{d}$,
$W(\alpha)=1$, $f(\alpha,n)=\e(\alpha R_{\lambda}(n))$ and
\begin{displaymath}
  \beta_1 = \frac14,
  \qquad
  \beta_2 = \beta_2(b) = 
  \frac12 \left(\frac12 + (1+2\eta_b)^{-1}\right) 
  \in \left]\frac12,(1+2\eta_b)^{-1}\right[
  .
\end{displaymath}

By Lemma~\ref{lemma:conclusion-S_II-individual}, there exists
$C(b) \in \left]0,1\right[$ such that if
$\beta_1\lambda\leq \mu \leq \beta_2\lambda$ then 
we have 
\begin{displaymath}
  S_{II}(\alpha,\lambda,\mu)
  \ll
  b^{\lambda (1-C(b))}
  +
  \lambda^{\frac{\omega(b)}2}
  b^{\lambda\left(1-\frac{\Upsilon_b\beta_1}{2\log d}\right)}
  \ll
  \lambda^{\frac{\omega(b)}2}
  b^{\lambda-\frac{c_{II}(b)\lambda}{\log d}}
\end{displaymath}
with
\begin{math}
  c_{II}(b)
  =
  \min\left(C(b)\log 2, \frac{\Upsilon_b \beta_1}2\right)
  > 0
  .
\end{math}

By Lemma~\ref{lemma:conclusion-S_I-individual}, if
$1\leq \mu \leq \beta_1\lambda$ then we have
\begin{displaymath}
  S_{I}(\alpha,\lambda,\mu)
  \ll
  b^{\beta_1(1+2\eta_b)\lambda}
  +
  b^{\left(1-\frac{\Upsilon_b}{\log \max(d,\Upsilon'_b)}\right)\lambda}
  \ll
  b^{\lambda-\frac{c_{I}(b)\lambda}{\log d}}
\end{displaymath}
where
$\Upsilon'_b$ is defined by \eqref{eq:definition-Upsilon'_b}
and 
\begin{displaymath}
  c_{I}(b)
  =
  \min\left((1-\beta_1(1+2\eta_b))\log 2,
    \frac{\Upsilon_b\log 2}{\log \Upsilon'_b},
    \Upsilon_b\right) > 0
  .
\end{displaymath}

There exists $\lambda_0(b)$ such that for
$\lambda \geq \lambda_0(b)$, we have
\begin{math}
  1 < \left(b^{\floor{\beta_1\lambda}}-1\right)^3 < b^{\lambda-1}
\end{math}
and
\begin{math}
  \ceil{\frac{\lambda}2}+1 \leq \beta_2 \lambda.
\end{math}
It follows from Lemma~\ref{lemma:maj_sum_alpha_vaughan} that for
$\lambda \geq \lambda_0(b)$, we have
\begin{displaymath}
  \sup_{t \in\left[b^{\lambda-1}, b^{\lambda}\right]}
  \abs{\sum_{b^{\lambda-1} \leq n < t}
    \Lambda(n)\e\left(\alpha R_{\lambda}(n)\right)}
  \ll
  \lambda^{2+\frac{\omega(b)}2}
  b^{\lambda-\frac{c(b)\lambda}{\log d}}
\end{displaymath}
with $c(b)= \min(c_{I}(b),c_{II}(b))>0$, which remains true for any
$\lambda \geq 2$ (by taking a larger implicit constant, depending at
most on $b$, if necessary).
This proves~\eqref{eq:bound_expo_sum_conclusion}.

The second part of Theorem~\ref{theorem:bound_expo_sum_conclusion}
follows by taking
\begin{displaymath}
  c'=c'(b,A)=
  c(b)(\log b)(2+\frac{\omega(b)}2+A)^{-1}
  >0.
\end{displaymath}
This ends the proof of
Theorem~\ref{theorem:bound_expo_sum_conclusion}.

\subsection{Conclusion on average over $\alpha$}

The goal of this section is to complete the proof of
Theorem~\ref{theorem:bound_expo_sum_average}.
Let
\begin{equation}\label{eq:def_xi_0}
  \xi_0 = \xi_0(b) = 
  \min\left(\frac{1-(1+2\eta_b)\frac{1}{2}}{2(3+4\eta_b)},\,
    \frac{u_b}{3-4u_b} \frac{1}{3},\,
    \frac{\iota_b}{1+6\iota_b}\frac{1}{3}
    ,\frac{1}{33}\right)
  >0
\end{equation}
where $\eta_b$, $u_b$ and $\iota_b$ are defined by
\eqref{eq:definition-eta_b}, \eqref{eq:def_u_b} and
\eqref{eq:def_iota_b} respectively.

Let $\xi \in \left] 0,\xi_0 \right[$.
There exists $\varepsilon=\varepsilon(b,\xi)>0$
such that
\begin{equation}\label{eq:maj_xi_epsilon}
  0 <
  \xi < 
  \min
  \left(
    \frac{1-(1+2\eta_b)\left(\frac{1}{2}+\varepsilon\right)}{2(3+4\eta_b)},\,
    \frac{u_b\left(\frac{1}{3}-\varepsilon\right)}{3-4u_b} ,\,
    \frac{\iota_b\left(\frac{1}{3}-\varepsilon\right)}{1+6\iota_b}
    ,\frac{1}{33}\right)
  .
\end{equation}
Let
\begin{equation}\label{eq:choice_beta_1-beta_2}
  \beta_1=\beta_1(b,\xi)=\frac{1}{3}-\varepsilon(b,\xi),
  \qquad
  \beta_2=\beta_2(b,\xi)=\frac{1}{2}+\varepsilon(b,\xi).
\end{equation}

We observe that \eqref{eq:maj_xi_epsilon} implies that
\begin{math}
  0 < \beta_1 < \beta_2 < (1+2\eta_b)^{-1}
\end{math}
and 
\begin{math}
  0 < \xi < C'(b,\beta_1,\beta_2)
\end{math}
where $C'(b,\beta_1,\beta_2)$ is defined
by~\eqref{eq:def_C_b,beta1,beta2}.  By
Lemma~\ref{lemma:conclusion-S_II-sum-on-alpha}, there exists
$c' = c'(b,\xi) >0$ and $\lambda'_0(b,\xi)\geq 2$ such that if
$\lambda\geq\lambda'_0(b,\xi)$ and $\mu$ is an integer such that
$\beta_1\lambda\leq \mu \leq \beta_2\lambda$ then, for $D = b^{\xi \lambda}$,
\begin{displaymath}
  \sum_{\alpha\in \mathcal{A}_D}
  W(\alpha)
  S_{II}(\alpha,\lambda,\mu)
  \ll
  b^{\lambda-c'\sqrt{\lambda}}
  .
\end{displaymath}
  
By \eqref{eq:upperbound-eta-b} and \eqref{eq:choice_beta_1-beta_2},
we have $\eta_b<\frac{1}{2}$ and $\beta_1 < \frac{1}3$, so that by
\eqref{eq:def_C''_b,beta1,beta2} we get
$C''(b,\beta_1) = \frac{1}{33}$ and therefore using
\eqref{eq:maj_xi_epsilon}, we obtain $0 < \xi < C''(b,\beta_1)$.
Applying Lemma~\ref{lemma:conclusion-S_I-sum-on-alpha}, there exists
$\tilde{\delta}=\tilde{\delta}_{b} >0$ and
$\lambda''_0(b,\xi)\geq 2$ such that if
$\lambda\geq\lambda''_0(b,\xi)$ and $\mu$ is an integer such that
$1\leq \mu \leq \beta_1\lambda$ then, for
$D = b^{\xi \lambda}$,
\begin{displaymath}
  \sum_{\alpha\in \mathcal{A}_D}
  W(\alpha)
  S_{I}(\alpha,\lambda,\mu)
  \ll
  b^{\lambda-\tilde{\delta}\sqrt{\lambda}}
  .
\end{displaymath}  

Moreover, there exists $\lambda'''_0(b,\xi)$ such that for
$\lambda \geq \lambda'''_0(b,\xi)$, we have
\begin{displaymath}
  1 < \left(b^{\floor{\beta_1\lambda}}-1\right)^3 < b^{\lambda-1}
\end{displaymath}
and
\begin{math}
  \ceil{\frac{\lambda}2}+1 \leq \beta_2 \lambda.
\end{math}
It follows from Lemma~\ref{lemma:maj_sum_alpha_vaughan} that for
\begin{displaymath}
  \lambda \geq
  \max(\lambda'_0(b,\xi),\lambda''_0(b,\xi),\lambda'''_0(b,\xi))
\end{displaymath}
and
$D = b^{\xi \lambda}$,
\begin{displaymath}
  \sum_{\alpha \in \mathcal{A}_D}
  W(\alpha)
  \sup_{t \in\left[b^{\lambda-1}, b^{\lambda}\right]}
  \abs{\sum_{b^{\lambda-1} \leq n < t}
    \Lambda(n)\e\left(\alpha R_{\lambda}(n)\right)}
  \ll
  \lambda^2 b^{\lambda-\min(c',\tilde{\delta})\sqrt{\lambda}}
  ,
\end{displaymath}
so that there exists $\lambda_0(b,\xi)$ such that for
$\lambda\geq \lambda_0(b,\xi)$ and
$c = c(b,\xi) = \frac{99}{100}\min(c',\tilde{\delta})$, we
obtain~\eqref{eq:maj_conclusion_sum_expo}.
By \eqref{eq:definition-A-d} and \eqref{eq:def-W},
this completes the proof of
Theorem~\ref{theorem:bound_expo_sum_average}.

\section{Study of $\xi_0(b)$}\label{section_study_xi_0}
Since $\xi_0(b)$ depends on $\iota_b$ (by \eqref{eq:def_xi_0}), which
in turn depends on $\kappa_b$ (by \eqref{eq:def_iota_b}), we first
study~$\kappa_b$.  We then provide values of $\xi_0(b)$ for
$2\leq b \leq 10$ and investigate the size of $\xi_0(b)$ for large
values of~$b$.  In this section, the implicit constants in the
notations $\ll$, $\gg$ and $O(\cdot)$ do not depend on $b$ and are
absolute.

\subsection{Study of $\kappa_b$}

For any integer $b\geq 2$ and any real number $\kappa>0$,
by \eqref{eq:max-Tb-simple} and \eqref{eq:definition-zeta_b}
we have
\begin{displaymath}
  b^{-\zeta_{b,\kappa}}
  =
  \max_{\alpha\in\R} T_{b,\kappa}(\alpha)
  <
  \frac{1}{b}+\sqrt{\frac{6(b+1)}{\pi(b-1)\kappa}}
  ,
\end{displaymath}
hence
\begin{displaymath}
  b^{-\zeta_{b,1}}
  b^{1-\zeta_{b,\kappa}}
  <
  b^{-\zeta_{b,1}}
  \left(1+\sqrt{\frac{6 b^2\, (b+1)}{\pi (b-1)\kappa}}\right)
  ,
\end{displaymath}
thus the integer $\kappa_b$ defined by \eqref{eq:def_kappa_b} satisfies
\begin{equation}\label{eq:kappa_b-upperbound}
  \kappa_b
  \leq
  \floor{
    \frac{6 b^2\, (b+1)}{
      \pi (b-1) \left(b^{\zeta_{b,1}}-1\right)^{2}}
  }
  +1
  .
\end{equation}

For small values of $b$, this upper bound is far from being optimal.
In particular, for $b=2$, by \eqref{eq:definition-zeta_b},
\eqref{eq:max-Tb-small-kappa}, \eqref{eq:definition-Tb} and
\eqref{eq:definition-Kb} we have
\begin{align*}
  2^{-\zeta_{2,1}}
  =
  \max_{\alpha\in\R} T_{2,1}(\alpha)
  &=
  T_{2,1}\left(\frac{1}{2}\right)
  \\
  &=
  \frac12 K_2\left(\frac{1}{3} \norm{\frac{1}{4}}
  \right)
  +
  \frac12 K_2\left(\frac{1}{3} \norm{\frac{3}{4}} \right)
  =
  \cos\frac{\pi}{12}
  ,
\end{align*}
which leads to the upper bound
\begin{displaymath}
  \kappa_2
  \leq
  \floor{
    6\cdot 4\cdot 3 \ \pi^{-1}
    \left(\frac{1}{\cos\frac{\pi}{12}}-1\right)^{-2}}
  + 1
  = 18418.
\end{displaymath}
Let us show by a numerical analysis that
\begin{displaymath}
  \kappa_2=25.
\end{displaymath}
For $\kappa \geq 1$, since $T_{2,\kappa}$ is $1$-periodic and even,
we may assume that $0\leq \alpha \leq \frac12$.
By \eqref{eq:definition-Tb}, we have
\begin{align*}
  T_{2,\kappa}\left(\alpha\right)
  &=
  \frac{1}{2} 
  K_2^{\kappa}\left(\frac{1}{3} \norm{\frac{\alpha}{2}} \right)
  +
  \frac{1}{2} 
  K_2^{\kappa}\left(\frac{1}{3} \norm{\frac{\alpha+1}{2}} \right)
  \\
  &=
  \frac{1}{2} 
  \cos^{\kappa}\left(\frac{\pi \alpha}{6}\right)
  +
  \frac{1}{2}
  \cos^{\kappa}\left(\frac{\pi (1-\alpha)}{6}\right)
\end{align*}
and
\begin{math}
  \kappa \mapsto T_{2,\kappa}(\alpha)
\end{math}
is decreasing.
By an elementary numerical computation,
it follows that
for $\kappa \in\{1,\ldots,24\}$,
\begin{multline*}
  T_{2,\kappa}\left(\frac{1}{20}\right)
  \geq
  T_{2,24}\left(\frac{1}{20}\right)
  =
  \frac12 \cos^{24} \frac{\pi}{120}
  +
  \frac12 \cos^{24} \frac{19\,\pi}{120} 
  \\
  >
  2^{\zeta_{2,1}-1}
  =
  \left(2 \cos \frac{\pi}{12} \right)^{-1}
  ,
\end{multline*}
so that $\kappa_2\geq 25$.
Introducing
\begin{multline*}
  \left(\alpha_0,\ldots,\alpha_{13}\right)
  \\
  =
  (0,0.03,0.04,0.05,0.06,0.07,0.08,0.10,0.13,0.17,0.23,0.32,0.43,0.5)
  ,
\end{multline*}
we have for $j\in\{0,\ldots,12\}$,
\begin{displaymath}
  \forall \alpha\in\left[\alpha_j,\alpha_{j+1}\right],\
  T_{2,\kappa}\left(\alpha\right)
  \leq
  \frac{1}{2} 
  \cos^{\kappa}\left(\frac{\pi \alpha_j}{6}\right)
  +
  \frac{1}{2}
  \cos^{\kappa}\left(\frac{\pi (1-\alpha_{j+1})}{6}\right)
  ,
\end{displaymath}
and checking numerically for $\kappa=25$ that
for $j\in\{0,\ldots,12\}$ we have
\begin{displaymath}
  \frac{1}{2} 
  \cos^{\kappa}\left(\frac{\pi \alpha_j}{6}\right)
  +
  \frac{1}{2}
  \cos^{\kappa}\left(\frac{\pi (1-\alpha_{j+1})}{6}\right)
  < 2^{\zeta_{2,1}-1} = \left( 2 \cos \frac\pi{12} \right)^{-1}
  ,
\end{displaymath}
we conclude that $\kappa_2=25$.

This value of $\kappa_2$ can also be obtained by a collection of plots,
as well as the next further values of $\kappa_b$:
\begin{table}[H]
  \centering
  \begin{equation*}
    \begin{tabu}{|c|c|c|c|c|c|c|c|c|c|}
      \hline
      b        &  2 &  3 &  4 &  5 &  6 &  7  &  8  &  9  &  10 \\
      \hline
      \kappa_b & 25 & 37 & 52 & 70 & 92 & 117 & 145 & 176 & 211 \\
      \hline
    \end{tabu}
  \end{equation*}
  \caption{Values of $\kappa_b$ for $2 \leq b \leq 10$}
  \label{table:values_kappa_b}
\end{table}

Let us now focus on the large values of $b$. We recall
\eqref{eq:kappa_b-upperbound}.  By \eqref{eq:definition-zeta_b},
\eqref{eq:max-Tb-small-kappa}, \eqref{eq:approximation-Tb} and
\eqref{eq:integral-approximation-Tb} we have
\begin{displaymath}
  b^{-\zeta_{b,1}}
  =
  T_{b,1}\left(\frac{b+1}{2}\right)
  =
  T_{\infty,1}
  + O\left(\frac1b\right)
  ,
\end{displaymath}
with
\begin{equation}\label{eq:definition-T_infty_1}
  T_{\infty,1}
  =
  \int_{-1/2}^{1/2} \frac{\sin \pi t}{\pi t} \, dt
  =
  0.87265429946\ldots
  .
\end{equation}
Hence
\begin{displaymath}
  b^{\zeta_{b,1}}
  =
  T_{\infty,1}^{-1}
  + O\left(\frac1b\right)
  ,
\end{displaymath}
thus
\begin{displaymath}
  \left( b^{\zeta_{b,1}} - 1 \right)^2
  =
  \left(T_{\infty,1}^{-1}-1\right)^2
  + O\left(\frac1b\right)
  .
\end{displaymath}
It follows that
\begin{equation}\label{eq:upper_bound_kappa_b_b_large}
  \kappa_b
  \leq
  \floor{
    \frac{6 b^2\, (b+1)}{
      \pi (b-1) \left(b^{\zeta_{b,1}}-1\right)^{2}}
  }
  +1
  =
  \frac{6 b^2}{\pi \left(T_{\infty,1}^{-1}-1\right)^2}
  +
  O(b)
  .
\end{equation}

\subsection{Values of $\xi_0(b)$ for $2\leq b \leq 10$}
From \eqref{eq:def_xi_0}, \eqref{eq:definition-eta_b},
\eqref{eq:def_u_b}, \eqref{eq:def_iota_b} and
Table~\ref{table:values_kappa_b}, we obtain the following table of
values of $\xi_0(b)$.
\begin{table}[H]
  \centering
  \begin{equation*}
    \begin{tabu}{||c|c||c|c||c|c||}
      \hline
      b &  \xi_0(b) & b &  \xi_0(b) & b &  \xi_0(b)
      \\
      \hline
      2 & 0.00441445\ldots & 5 & 0.00206095\ldots & 8 & 0.00101824\ldots\\
      3 & 0.00301523\ldots & 6 & 0.00149811\ldots & 9 & 0.00088096\ldots\\
      4 & 0.00226727\ldots & 7 & 0.00131958\ldots & 10 & 0.00072623\ldots\\
      \hline
    \end{tabu}
  \end{equation*}
  \caption{Values of $\xi_0(b)$ for $2 \leq b \leq 10$}
  \label{table:values_xi_0_b}
\end{table}

\subsection{Study of $\xi_0(b)$ for $b$ large}

For any integer $b\geq 2$,
for any integer $\kappa'_b \geq \kappa_b$,
we have by \eqref{eq:def_xi_0} and \eqref{eq:formula_iota_u_kappa},
\begin{displaymath}
  \xi_0(b) \geq
  \xi'_0(b) := 
  \min\left(\frac{1-(1+2\eta_b)\frac{1}{2}}{2(3+4\eta_b)},\,
    \frac{u_b}{3-4u_b} \frac{1}{3},\,
    \frac{u_b}{1+2(\kappa'_b+4)u_b} \frac{1}{3},
    \frac{1}{33}\right)
  .
\end{displaymath}
For $\kappa'_b \gg b$ and $b$ large enough, let us show that the third
term provides the minimum.  Since $\eta_b \to 0$ as $b \to +\infty$
(by \eqref{eq:majoration-psi-b} and \eqref{eq:definition-eta_b}) and
$u_b \gg (\log b)^{-1}$ (by \eqref{eq:def_u_b}), if $\kappa'_b \gg b$
then for $b$ large enough we have
\begin{math}
 \kappa'_b + 6 > u_b^{-1}
\end{math}
which implies that
\begin{displaymath}
  3-4u_b < 1+2(\kappa'_b+4)u_b
  ,
\end{displaymath}
so that
\begin{displaymath}
  \frac{u_b}{3-4u_b} \frac{1}{3}
  >
  \frac{u_b}{1+2(\kappa'_b+4)u_b} \frac{1}{3}
  \to 0,
\end{displaymath}
and therefore, for $b$ large enough,
\begin{displaymath}
  \xi'_0(b)
  = 
  \frac{u_b}{1+2(\kappa'_b+4)u_b} \frac{1}{3}
  ,
\end{displaymath}
hence
\begin{displaymath}
  \xi'_0(b)
  = 
  \frac{1}{6(\kappa'_b+4)} 
  \frac{1}{1+\frac{1}{2(\kappa'_b+4)u_b}}
  =
  \frac{1}{6(\kappa'_b+4)}
  +
  O\left(
    \frac{1}{(\kappa'_b+4)^2u_b}
  \right)
  .
\end{displaymath}
If we take
\begin{displaymath}
  \kappa'_b =
  \floor{
    \frac{6 b^2\, (b+1)}{
      \pi (b-1) \left(b^{\zeta_{b,1}}-1\right)^{2}}
  }
  + 1
\end{displaymath}
then, by~\eqref{eq:upper_bound_kappa_b_b_large}, $\kappa'_b$ fulfills
the above conditions, hence for $b$ large enough,
\begin{displaymath}
  \xi'_0(b)
  = 
  \frac{\pi \left(T_{\infty,1}^{-1}-1\right)^2}{36 \, b^2}
  +
  O\left( \frac{1}{b^3} \right)
\end{displaymath}
and we conclude that 
\begin{equation}\label{eq:lower_bound_xi_0}
  \xi_0(b)
  \geq
  \frac{\pi \left(T_{\infty,1}^{-1}-1\right)^2}{36 \, b^2}
  +
  O\left( \frac{1}{b^3} \right)
  .
\end{equation}

\section{Completion of the proofs of the main results}

\subsection{Proof of Theorem~\ref{theorem:type_bombieri-vinogradov}}
\label{section:proof_thm_type_BV}

Let $d,a \in \Z$, $d\geq 2$ and
$t \in\left[b^{\lambda-1}, b^{\lambda}\right]$.
We write
\begin{displaymath}
  \pimirror_\lambda(t,a,d)
  =
  \frac{1}{d} \sum_{h=1}^{d} \e\left(\frac{-ha}{d}\right)
  \sum_{b^{\lambda-1} \leq p < t}
  \e\left(\frac{h R_{\lambda}(p)}{d}\right).
\end{displaymath}
Since the contribution of $h=d$ is $\pi_\lambda(t)/d$, we have
\begin{displaymath}
  \abs{
      \pimirror_\lambda(t,a,d)
      - \frac{ \pi_\lambda(t)}{d}
    }
    \leq
    \frac{1}{d} \sum_{h=1}^{d-1} 
    \abs{\sum_{b^{\lambda-1} \leq p < t}
      \e\left(\frac{h R_{\lambda}(p)}{d}\right)}
    .
\end{displaymath}
Moreover, by partial summation, we have for any $\alpha \in\R$,
\begin{align*}
  &\abs{ \sum_{b^{\lambda-1} \leq p < t} \e\left(\alpha R_{\lambda}(p)\right)}
  \\
  &\leq
  \frac{1}{\log b^{\lambda-1}}
   \sup_{u \in \left[b^{\lambda-1}, t \right]}
  \abs{
    \sum_{b^{\lambda-1} \leq p < u} (\log p) \e\left(\alpha R_{\lambda}(p)\right)
  }\\
  &\leq
   \frac{1}{\log b^{\lambda-1}}
   \sup_{u \in \left[b^{\lambda-1}, t \right]}
  \abs{
    \sum_{b^{\lambda-1} \leq n < u} \Lambda(n) \e\left(\alpha R_{\lambda}(n)\right)
  }
    + \frac{\pi(\sqrt{t})  \log t}{\log b^{\lambda-1}}.
\end{align*}
It follows that
\begin{multline}\label{eq:maj_after_partial_summation}
  \sup_{t \in\left[b^{\lambda-1}, b^{\lambda}\right]}
  \sup_{1\le a\le d}
  \abs{
    \pimirror_\lambda(t,a,d)
    - \frac{ \pi_\lambda(t)}{d}
  }
  \\
  \leq
  \frac{1}{\log b^{\lambda-1}}\frac{1}{d} \sum_{h=1}^{d-1}
   \sup_{u \in \left[b^{\lambda-1}, b^{\lambda} \right]}
  \abs{
    \sum_{b^{\lambda-1} \leq n < u} \Lambda(n) \e\left(\frac{h R_{\lambda}(n)}{d}\right)
  }
  +
  \frac{\pi(\sqrt{b^{\lambda}})  \log b^{\lambda}}{\log b^{\lambda-1}}.
\end{multline}
Filtering according to the value of $\gcd(h,d)$, we see that for any
real number $D\geq 2$,
\begin{multline}\label{eq:maj_after_filtering_gcd}
  \sum_{\substack{d \leq D\\ \gcd(d,b(b^2-1))=1}}
  \frac{1}{d} \sum_{h=1}^{d-1}
  \sup_{u \in \left[b^{\lambda-1}, b^{\lambda} \right]}
  \abs{ \sum_{b^{\lambda-1} \leq n < u} \Lambda(n)
    \e\left(\frac{h R_{\lambda}(n)}{d}\right) }
  \\
  \ll
  (\log D)
  \hspace{-1em}
  \sum_{\substack{d' \leq D\\ \gcd(d',b(b^2-1))=1}}
  \frac{1}{d'} \sum_{\substack{1\leq h'<d'\\\gcd(h',d')=1}}
  \sup_{u \in \left[b^{\lambda-1}, b^{\lambda} \right]}
  \abs{ \sum_{b^{\lambda-1} \leq n < u} \Lambda(n)
    \e\left(\frac{h' R_{\lambda}(n)}{d'}\right) }
  \\=
  (\log D)
  \sum_{\alpha \in \mathcal{A}_D}
  W(\alpha)
  \sup_{u \in\left[b^{\lambda-1}, b^{\lambda}\right]}
    \abs{\sum_{b^{\lambda-1} \leq n < u}
      \Lambda(n)\e\left(\alpha R_{\lambda}(n)\right)}
  .
\end{multline}
Now let $\xi_0 = \xi_0(b)$ be defined by~\eqref{eq:def_xi_0} and
$\xi \in\left]0,\xi_0\right[$.  By
Theorem~\ref{theorem:bound_expo_sum_average}, there exists
$c = c(b,\xi) >0$ and $\lambda_0(b,\xi)\geq 2$ such that for any
$\lambda\geq\lambda_0(b,\xi)$ and $D = b^{\xi \lambda}$, we have
\begin{displaymath}
  \sum_{\alpha \in \mathcal{A}_D}
  W(\alpha)
  \sup_{u \in\left[b^{\lambda-1}, b^{\lambda}\right]}
  \abs{\sum_{b^{\lambda-1} \leq n < u}
    \Lambda(n)\e\left(\alpha R_{\lambda}(n)\right)}
  \ll
  b^{\lambda-c\sqrt{\lambda}},
\end{displaymath}
which, combined with \eqref{eq:maj_after_partial_summation} and
\eqref{eq:maj_after_filtering_gcd}, leads to
\begin{multline*}
  \sum_{\substack{d \leq b^{\xi \lambda}\\ \gcd(d,b(b^2-1))=1}}
  \sup_{t \in\left[b^{\lambda-1}, b^{\lambda}\right]}
  \sup_{1\le a\le d}
  \abs{
    \pimirror_\lambda(t,a,d)
    - \frac{ \pi_\lambda(t)}{d}
  }
  \\
  \ll
  \frac{\log b^{\xi \lambda}}{\log b^{\lambda-1}}\, b^{\lambda-c\sqrt{\lambda}}
  + b^{\xi \lambda} \,\frac{\pi(\sqrt{b^{\lambda}})  
    \log b^{\lambda}}{\log b^{\lambda-1}}
  .
\end{multline*}
Since $\xi < \xi_0 < \frac{1}{2}$, the right-hand side above is
$\ll b^{\lambda-c\sqrt{\lambda}}$ for $\lambda\geq \lambda_1(b,\xi)$
large enough.
This ends the proof of Theorem~\ref{theorem:type_bombieri-vinogradov}.

\subsection{Proof of Theorem~\ref{theorem:distribution-level}}
\label{section:proof_thm_distribution-level}

Let $\lambda, i, d$ be integers such that $\lambda\geq 3$,
$i \in \{1,\ldots,b-1\}$, $d\geq 1$ and $\gcd(d,b(b^2-1))=1$.
By \eqref{eq:def_P_lambda_i} and \eqref{eq:def_pi_lambda_t}, we have
\begin{align*}
  \abs{\mathscr{P}_{\lambda,i}}
  &=
  \abs{\{i b^{\lambda-1}\leq p < (i+1)b^{\lambda-1}: p \text{ prime}\}}
  \\
  &=
  \pi_{\lambda}\left((i+1)b^{\lambda-1}\right)
  -
  \pi_{\lambda}\left(i b^{\lambda-1}\right)
  .
\end{align*}
By \eqref{def_T_lambda_d} and \eqref{eq:def_pimirror_lambda_t_a_d}, we have
\begin{align*}
  T_{\lambda,i}(d)
  &=
  |\{ p \in \mathscr{P}_{\lambda,i}: R_\lambda(p) \equiv 0 \bmod d\}|
  \\
  &=
  \pimirror_{\lambda}\left((i+1)b^{\lambda-1},0,d\right)
  -
  \pimirror_{\lambda}\left(i b^{\lambda-1},0,d\right)
  .
\end{align*}
Moreover, since $\gcd(d,b(b^2-1))=1$, we have by \eqref{eq:def_g}
\begin{displaymath}
  g(d) = d^{-1}.
\end{displaymath}
It follows from \eqref{eq:def_Rd} that
\begin{displaymath}
  \abs{E_{\lambda,i}(d)}
  =  
  \abs{T_{\lambda,i}(d) - g(d) \abs{\mathscr{P}_{\lambda,i}}}
  \leq 
  2 \sup_{t \in\left[b^{\lambda-1}, b^{\lambda}\right]}
  \abs{
    \pimirror_\lambda(t,0,d)
    - \frac{ \pi_\lambda(t)}{d}
  }
  ,
\end{displaymath}
hence, for any real number $D>1$,
\begin{multline*}
  \sum_{\substack{d \leq D\\\gcd(d,b(b^2-1))=1}}
  \sup_{1\leq i \leq b-1}
  \abs{E_{\lambda,i}(d)}
  \\
  \leq
  2\sum_{\substack{d \leq D\\ \gcd(d,b(b^2-1))=1}}
  \sup_{t \in\left[b^{\lambda-1}, b^{\lambda}\right]}
  \sup_{1\le a\le d}
  \abs{
    \pimirror_\lambda(t,a,d)
    - \frac{ \pi_\lambda(t)}{d}
  }
  .
\end{multline*}
Theorem~\ref{theorem:distribution-level} then appears as an immediate
consequence of Theorem~\ref{theorem:type_bombieri-vinogradov}.

\subsection{Proof of Theorem~\ref{theorem:theta_upper_bound} and
  Theorem~\ref{theorem:theta_lower_bound}}
\label{section:proof_bounds_Theta_lambda_z}

Let $\xi \in \left]0,\xi_0/2\right[$ where $\xi_0=\xi_0(b)$ is
given by~\eqref{eq:def_xi_0}. We define
$\xi_1 = \xi_1(b,\xi) = \frac12 \left(2 \xi + \xi_0\right) $ so that
$2 \xi < \xi_1 < \xi_0$ and let $z = b^{\xi \lambda}$ and
$D = b^{\xi_1 \lambda} \geq z$.
We recall the notations of Section~\ref{section-sieves}. 

By Theorem~\ref{theorem:distribution-level}, since
$\xi_1 \in \left]0,\xi_0\right[$, there exists $c = c(b,\xi) >0$ and
$\lambda_1(b,\xi)\geq 2$ with the property that for any integer
$\lambda\geq\lambda_1(b,\xi)$, we have
\begin{displaymath}
  \sum_{\substack{d \leq D\\d \mid P(z)}}
  \sup_{1\leq i \leq b-1} \abs{E_{\lambda,i}(d)}
  \leq
  \sum_{\substack{d \leq b^{\xi_1 \lambda}\\\gcd(d,b(b^2-1))=1}}
  \sup_{1\leq i \leq b-1}
  \abs{E_{\lambda,i}(d)}
  \ll
  b^{\lambda-c\sqrt{\lambda}},
\end{displaymath}
hence, by~\eqref{eq:upper_bound_Theta_lambda_z} and
\eqref{eq:lower_bound_Theta_lambda_z}, we have for any
$i \in \{1,\ldots,b-1\}$,
\begin{equation}\label{eq:maj_varTheta_i_inter}
  \varTheta_i\left(\lambda, b^{\xi \lambda}\right)
  <
  \left(F\left(\tfrac{\xi_1}{\xi}\right)
    + O\left((\log b^{\xi_1 \lambda})^{-1/6}\right)\right)
  \abs{\mathscr{P}_{\lambda,i}} V\left(b^{\xi \lambda}\right)
  +
  O\left(b^{\lambda-c\sqrt{\lambda}}\right)
  ,
\end{equation}
\begin{equation}\label{eq:min_varTheta_i_inter}
  \varTheta_i\left(\lambda, b^{\xi \lambda}\right)
  >
  \left(f\left(\tfrac{\xi_1}{\xi}\right)
    + O\left((\log b^{\xi_1 \lambda})^{-1/6}\right)\right)
  \abs{\mathscr{P}_{\lambda,i}} V\left(b^{\xi \lambda}\right)
  +
  O\left(b^{\lambda-c\sqrt{\lambda}}\right)
  .
\end{equation}
Moreover, by the Prime Number Theorem, we have for any
$i \in\{1,\ldots,b-1\}$,
\begin{equation}\label{eq:asymp_P_lambda_i}
  \abs{\mathscr{P}_{\lambda,i}}
  =
  \frac{b^{\lambda-1}}{\log b^{\lambda-1}}
  \left(1+O\left(\frac{1}{\lambda}\right)\right)
\end{equation}
and by~\eqref{eq:asymp_formula_V}, if $b^{\xi \lambda} > b(b^2-1)$ then
\begin{equation}\label{eq:asymp_V_evaluated}
  V\left(b^{\xi \lambda}\right)
  =
  \frac{e^{-\gamma}}{\log b^{\xi \lambda}}
  \left(1+O \left(\frac{1}{\log b^{\xi \lambda}}\right)\right)
  \prod_{p \dv b(b^2-1)} \left(1-p^{-1}\right)^{-1}
  .
\end{equation}

It follows from~\eqref{eq:maj_varTheta_i_inter} that there exists
$\lambda'_0(b,\xi)$ such that for any
$\lambda \geq \lambda'_0(b,\xi)$, we have
\begin{displaymath}
  \sup_{1\leq i \leq b-1} \varTheta_i\left(\lambda, b^{\xi \lambda}\right)
  \ll_{\xi} \frac{b^{\lambda}}{\lambda^2}.
\end{displaymath}
By taking a larger implicit constant if necessary, we see that this
remains true for any $\lambda\geq 1$. Since
$z \mapsto \varTheta_i(\lambda,z)$ is decreasing, this proves
Theorem~\ref{theorem:theta_upper_bound}.

Since $2\xi < \xi_1$, we have $f\left(\tfrac{\xi_1}{\xi}\right)>0$ and
it follows from~\eqref{eq:min_varTheta_i_inter} that there exists
$\lambda_0(b,\xi)$ such that for any $\lambda \geq \lambda_0(b,\xi)$,
we have
\begin{displaymath}
  \inf_{1\leq i \leq b-1}\varTheta_i\left(\lambda, b^{\xi \lambda}\right)
  \gg_{\xi} \frac{b^{\lambda}}{\lambda^2},
\end{displaymath}
which proves Theorem~\ref{theorem:theta_lower_bound}.

\subsection{Proof of Theorem~\ref{theorem:application_weighted_sieve}}
\label{section:proof_theorem:application_weighted_sieve}

Let $\xi \in \left]0,\xi_0\right[$ where $\xi_0=\xi_0(b)$ is given
by~\eqref{eq:def_xi_0} and let $i\in\{1,\ldots,b-1\}$. We recall the
notations of Section~\ref{section-sieves} and we define
$D = b^{\xi \lambda}$.  By Theorem~\ref{theorem:distribution-level},
since $\xi \in \left]0,\xi_0\right[$, there exists $c = c(b,\xi) >0$
and $\lambda_1(b,\xi)\geq 2$ with the property that for any integer
$\lambda\geq\lambda_1(b,\xi)$, we have
\begin{displaymath}
  \sum_{d \leq D}
  \abs{E_{\lambda,i}(d)}
  =
  \sum_{\substack{d \leq D\\\gcd(d,b(b^2-1))=1}}
  \abs{E_{\lambda,i}(d)}
  \ll
  b^{\lambda-c\sqrt{\lambda}}
  ,
\end{displaymath}
hence by \eqref{eq:asymp_P_lambda_i},
\begin{displaymath}
  \sum_{d \leq D}
  \abs{E_{\lambda,i}(d)}
  \ll
  \abs{\mathscr{P}_{\lambda,i}} (\log \abs{\mathscr{P}_{\lambda,i}})^{-3}
  .
\end{displaymath}
Moreover, by omitting the fact that the integers counted in
$T_{\lambda,i} (p^2)$ are prime numbers and next using the one-to-one
correspondence between an integer and its reverse, we obtain for any
real number $z\geq 1$,
\begin{align*}
\sum_{p>z}T_{\lambda,i} (p^2) 
  &\le \sum_{p>z}\abs{\{ib^{\lambda -1} \le n < (i+1)b^{\lambda-1} :
    p^2 \mid R_\lambda(n)\} }\\
& \le\sum_{z<p <b^{\lambda /2}}\abs{ \{ m<b^\lambda : p^2 \mid m \}}\\
& \ll \sum_{p>z}\frac{b^\lambda}{p^2}\ll \frac{b^\lambda}{z},
\end{align*}
hence, taking $z = D^{1/4}$, there exists $\lambda_2(b,\xi)\geq 2$
such that for any $\lambda\geq\lambda_2(b,\xi)$, we have
\begin{displaymath}
  \sum_{p>D^{1/4}} T_{\lambda,i} (p^2)
  \ll
  \abs{\mathscr{P}_{\lambda,i}} (\log \abs{\mathscr{P}_{\lambda,i}})^{-3}.
\end{displaymath}
Let $R$ be an integer such that $\Lambda_R > \xi^{-1}$ with
$\Lambda_R$ defined by \eqref{eq:def_Lambda_R} and let
\begin{math}
  \varepsilon = \varepsilon(\xi,R) = \xi - \Lambda_R^{-1} >0 
\end{math}
so that
\begin{math}
  D = b^{\xi \lambda}
  = \left(b^{\lambda}\right)^{\Lambda_R^{-1} +\varepsilon}.
\end{math}

It follows from Lemma~\ref{lemma-weighted-sieve} and
\eqref{eq:asymp_V_evaluated} that there exists
$\lambda_0(b,\xi,R)\geq 2$ such that for any integer
$\lambda\geq\lambda_0(b,\xi,R)$, we have
\begin{displaymath}
  |\{p \in \mathscr{P}_{\lambda,i}: 
  \Omega(R_{\lambda}(p)) \leq R \text{ and } 
  P^{-}(R_\lambda(p))\geq b^{\frac{\xi }{4}\lambda}\}|
  \asymp_{\xi,R}
  \frac{b^{\lambda}}{\lambda^2}.
\end{displaymath}
This ends the proof of Theorem~\ref{theorem:application_weighted_sieve}.

\subsection{Proof of Theorem~\ref{theorem:true-reversible-primes}}
\label{section:proof_thm_upper_bound_reversible_primes}

Let $\lambda$ and $n$ be integers such that
$b^2 \leq b^{\lambda-1} \leq n < b^{\lambda}$.

If $n$ is prime then, since 
\begin{math}
  \varepsilon_{\lambda-1}(R_{\lambda}(n))
  =
  \varepsilon_0(n) \geq 1
  ,
\end{math}
we have $R_{\lambda}(n) \geq b^{\lambda-1}$.

If $R_{\lambda}(n)$ is prime and $R_{\lambda}(n) \geq b^{\lambda-1}$
then, since
\begin{math}
  \varepsilon_{\lambda-1}(n)
  =
  \varepsilon_0(R_{\lambda}(n))
  ,
\end{math}
we have $\gcd(\varepsilon_{\lambda-1}(n),b)=1$.

Recalling \eqref{eq:def_P_lambda_i} and
\eqref{eq:def_varTheta_lambda_z}, it follows that for any integer
$\lambda \geq 3$ and any real number $z\leq b^{\lambda-1}$,
\begin{multline*}
  \abs{\{b^{\lambda-1} \leq p < b^\lambda :
    R_\lambda(p) \text{ is prime} \}
  }
  \\
  =
  \sum_{\substack{1 \leq i \leq b-1\\\gcd(i,b)=1}}
  \abs{\{ p \in \mathscr{P}_{\lambda,i} :
    R_\lambda(p) \text{ is prime} \}
  }
  \leq
  \sum_{\substack{1 \leq i \leq b-1\\\gcd(i,b)=1}}
  \varTheta_i(\lambda,z)
  .
\end{multline*}
Theorem~\ref{theorem:true-reversible-primes} then appears as an
immediate consequence of Theorem~\ref{theorem:theta_upper_bound}.

\bigskip

\subsection{Proof of a weaker version of Theorem~\ref{thm-almost-prime}}
\label{section:proof_thm_almost_primes_weaker}

If $\lambda,n\geq 1$ are integers and $\xi$ is a positive real number
such that $n < b^{\lambda}$ and all prime factors of $n$ are
$\geq b^{\xi \lambda}$ then
\begin{displaymath}
  b^{\xi \lambda \Omega(n)} \leq n < b^{\lambda},
\end{displaymath}
hence $\Omega(n) < \ceil{\xi^{-1}}$. 
Recalling that by \eqref{eq:def_varTheta_lambda_z}, 
\begin{displaymath}
  \varTheta_1(\lambda,b^{\xi \lambda})
  =
  \abs{\{b^{\lambda-1} \leq p < 2 b^{\lambda-1} :
   P^-(R_{\lambda}(p)) \geq b^{\xi \lambda}\}}
\end{displaymath}
we obtain
\begin{equation}\label{eq:min_Omega_varTheta_inter}
  \abs{\{b^{\lambda-1} \leq p < 2 b^{\lambda-1} :
    \Omega(R_\lambda(p))\leq \ceil{\xi^{-1}}-1\}}
  \geq 
  \varTheta_1(\lambda,b^{\xi \lambda}).
\end{equation}

Let
\begin{equation}\label{eq:def_Omega_b_tilde}
  \widetilde{\Omega}_b = \floor{2/\xi_0(b)}
\end{equation}
where $\xi_0(b) >0$ is defined by \eqref{eq:def_xi_0}. We introduce
$\xi_1(b) = (\widetilde{\Omega}_b+1)^{-1}$ and
$\xi = \xi(b) = \frac12 \left(\xi_1(b)+ \frac{\xi_0(b)}2\right)$, so that,
since $\frac{1}{\xi_1(b)} > \frac{2}{\xi_0(b)}$, we have
$\xi_1(b) < \xi < \frac{\xi_0(b)}2$ and
$\ceil{\xi^{-1}} = \widetilde{\Omega}_b+1$.
Then, by \eqref{eq:min_Omega_varTheta_inter} and
Theorem~\ref{theorem:theta_lower_bound}, there exists
$\lambda_0(b) \geq 2$ such that for any integer
$\lambda \geq \lambda_0(b)$, we have
\begin{displaymath}
  \abs{\{b^{\lambda-1} \leq p < 2 b^{\lambda-1} :
    \Omega(R_\lambda(p))\leq \widetilde{\Omega}_b\}}
  \geq
  \varTheta_1(\lambda,b^{\xi\lambda})
  \gg \frac{b^{\lambda}}{\lambda^2}.
\end{displaymath}
This shows that Theorem~\ref{thm-almost-prime} holds with $\Omega_b$
replaced by $\widetilde{\Omega}_b$.

\subsection{Proof of Theorem~\ref{thm-almost-prime} }
\label{section:proof_thm_almost_primes}
Let 
\begin{equation}\label{eq:def_Omega_b_weighted_sieves}
  \Omega_b=1+\ceil{1/\xi_0(b)}
\end{equation}
where $\xi_0(b) >0$ is defined by \eqref{eq:def_xi_0}.  We recall that
$\Lambda_R$ defined by \eqref{eq:def_Lambda_R} satisfies
$R-1 < \Lambda_R \leq R$.  By applying
Theorem~\ref{theorem:application_weighted_sieve} with $R = \Omega_b$
and
\begin{math}
  \xi = \frac{1}{2}\left(\frac{1}{\Lambda_R}+\frac{1}{R-1}\right)
  ,
\end{math}
which satisfy
\begin{math}
  \frac{1}{\Lambda_R} < \xi < \frac{1}{R-1} \leq \xi_0(b)
  ,
\end{math}
we obtain that there exists $\lambda_0(b) \geq 2$ such that for any
integer $\lambda\geq \lambda_0(b)$ and any $i\in \{1,\ldots,b-1\}$,
\begin{displaymath}
  |\{p \in \mathscr{P}_{\lambda,i}: 
  \Omega(R_{\lambda}(p)) \leq \Omega_b \}|
  \gg
  \frac{b^{\lambda}}{\lambda^2}
  .
\end{displaymath}

For $2 \leq b \leq 10$, we deduce the values of $\Omega_b$
given in Table~\ref{table:values_Omega_b} from Table~\ref{table:values_xi_0_b}
and \eqref{eq:def_Omega_b_weighted_sieves}.

For larger bases $b$, the upper bound of $\Omega_b$ given in
\eqref{eq:maj_Omega_b} follows from \eqref{eq:lower_bound_xi_0} and
\eqref{eq:def_Omega_b_weighted_sieves}.

\bigskip

\subsection{Proof of Theorem~\ref{theorem:reversed-siegel-walfisz}}
\label{section:proof-thm-reversed-siegel-walfisz}

Let $d,a \in \Z$, $d\geq 2$ and
$t \in\left[b^{\lambda-1}, b^{\lambda}\right]$. As in Section
\ref{section:proof_thm_type_BV} we employ exponential sums:
\begin{displaymath}
  \pimirror_\lambda(t,a,d)
  =
  \frac{1}{d} \sum_{h=1}^{d} \e\left(\frac{-ha}{d}\right)
  \sum_{b^{\lambda-1} \leq p < t}
  \e\left(\frac{h R_{\lambda}(p)}{d}\right).
\end{displaymath}
We split this sum in two parts
\begin{displaymath}
  \pimirror_\lambda(t,a,d)=M_\lambda (t,a,d)+E_\lambda (t,a,d)
\end{displaymath}
where $M_\lambda (t,a,d)$ is the contribution of the $h$ such that
$h(b^2-1)b^\lambda\equiv 0\bmod d$ and $E_\lambda (t,a,d)$ is the
contribution of the other~$h$.

The integers $h$ involved in $M_\lambda (t,a,d)$ can be written as
\begin{displaymath}
  h = k \, \frac{d}{\gcd (d, (b^2-1)b^\lambda)}
\end{displaymath}
with $1\le k \leq  \gcd (d,(b^2-1)b^\lambda)$.  We deduce that
\begin{align*}
  M_\lambda (t,a,d)
  &=
    \frac{1}{d}\sum_{1\le k \leq \gcd (d,(b^2-1)b^\lambda)} 
    \sum_{b^{\lambda -1}\le p<t}
    \e\left(\frac{k(R_\lambda (p)-a)}{\gcd (d,(b^2-1) b^\lambda)}\right)\\
  &=
    \frac{\gcd (d, (b^2-1)b^\lambda)}{d}
    \pimirror_\lambda (t,a,\gcd (d,(b^2-1)b^\lambda)).
\end{align*}

The error term is bounded by 
\begin{displaymath}
  \abs{E_\lambda (t,a,d)}
  \le
  \sup_{\substack{1\le h\le d\\h (b^2-1)b^\lambda \not\equiv 0 \bmod d}}
  \abs{
  \sum_{b^{\lambda -1}\le p <t}
  \e\left (\frac{hR_\lambda (p)}{d}
  \right )
  }.
\end{displaymath}
By partial summation exactly as in
Section~\ref{section:proof_thm_type_BV}, we insert the von Mangoldt
function:
\begin{multline*}
  \abs{E_\lambda (t,a,d)}
  \\
  \ll 
  \frac{1}{\lambda}
  \sup_{\substack{1\le h\le d\\h (b^2-1)b^\lambda \not\equiv 0 \bmod d}}
  \sup_{u\in [b^{\lambda -1}, b^\lambda]}
  \abs{
    \sum_{b^{\lambda -1}\le n <u}\Lambda (n)\e \left (\frac{hR_\lambda (n)}{d}\right )
  }
  + \frac{\sqrt{b^{\lambda}}}{\lambda}.
\end{multline*}
By Theorem~\ref{theorem:bound_expo_sum_conclusion}, there exists
$c'(b)>0$ such that for any integer $h$ with
$h (b^2-1)b^\lambda \not\equiv 0 \bmod d$, we have
\begin{displaymath}
  \sup_{u\in [b^{\lambda -1}, b^\lambda]}
  \abs{
    \sum_{b^{\lambda -1}\le n <u}\Lambda (n)
    \e \left (\frac{hR_\lambda (n)}{d}\right )
  }
  \ll
  \lambda^{2+\frac{\omega(b)}2}
  b^{\lambda} \exp\left(-\frac{c'(b)\lambda}{\log d} \right),
\end{displaymath}
hence
\begin{displaymath}
  \abs{E_\lambda (t,a,d)}
  \ll
  \lambda^{1+\frac{\omega(b)}2}
  b^{\lambda}
  \exp\left(-\frac{c'(b)\lambda}{\log d} \right)
  + \frac{\sqrt{b^{\lambda}}}{\lambda}.
\end{displaymath}
Taking 
\begin{math}
  c_1(b)
  =
  \min\left(
    c'(b)\left(2+\frac{\omega(b)}2\right)^{-1}
    ,
    \frac{(\log 2)(\log b)}2
  \right)
  >0
\end{math}
and assuming that
\begin{math}
  2 \leq d \leq \exp\left(\frac{c_1(b) \lambda}{\log \lambda}\right),
\end{math}
since
\begin{displaymath}
  \frac{\lambda}{\log d}(c'(b)-c_1(b))
  \geq
  \left(\frac{c'(b)}{c_1(b)}-1\right) \log \lambda
  \geq
  \left(1+\frac{\omega(b)}2\right) \log \lambda
\end{displaymath}
and
\begin{math}
  \frac{c_1(b)\lambda}{\log d} \leq \frac{\lambda\log b}{2},
\end{math}
it follows that
\begin{displaymath}
  \abs{E_\lambda (t,a,d)}
  \ll
  b^{\lambda}
  \exp\left(-\frac{c_1(b)\lambda}{\log d} \right).
\end{displaymath}
We conclude that for any integer $b\geq 2$, there exists $c_1(b)>0$
such that for $a,d,\lambda\in\Z$ with $\lambda\geq 2$ and
\begin{equation}\label{eq:range_d_proof}
  2 \leq d \leq \exp\left(\frac{c_1(b) \, \lambda}{\log \lambda}\right)
  ,
\end{equation}
we have for any $t \in\left[b^{\lambda-1}, b^{\lambda}\right]$,
\begin{multline}
  \label{eq:reversed-siegel-walfisz-BS}
  \pimirror_\lambda(t,a,d)
  =
  \frac{\gcd(d,(b^2-1)b^\lambda)}{d}
  \pimirror_\lambda\left(
    t,a,\gcd\left(d,\left(b^2-1\right)b^\lambda\right)
  \right)
  \\
  +
  O\left(b^\lambda \exp\left(-\frac{c_1(b) \, \lambda}{\log d}\right)\right)
  .
\end{multline}
This extends the result of Bhowmik and
Suzuki~\cite{bhowmik-suzuki-2024-arxiv} who
obtained~\eqref{eq:reversed-siegel-walfisz-BS} for $b\geq 31699$ (and
$t=b^\lambda$),
and also Chourasiya and Johnston~\cite{chourasiya-johnston-arxiv-2025}
who essentially
proved~\eqref{eq:reversed-siegel-walfisz-BS} for $b\ge 26000$.

Let us now prove \eqref{eq:reversed-siegel-walfisz} under the
condition \eqref{eq:range_d_proof}. Let
\begin{displaymath}
  \lambda_0(b)
  =
  \exp\left( \frac{c_1(b)}{\log 2} \right)
\end{displaymath}
and let us first assume that $\lambda \geq \lambda_0(b)$.  For
$p\dv b$, we have $p^{v_p(d)} \leq d$, hence by
\eqref{eq:range_d_proof},
\begin{displaymath}
  v_p(d)
  \leq
  \frac{c_1(b) \lambda}{(\log p)(\log \lambda)}
  \leq \frac{\log 2}{\log p} \lambda
  \leq\lambda \leq v_p(b^\lambda)
  ,
\end{displaymath}
thus,
for $\lambda\geq \lambda_0(b)$,
\begin{align*}
  \gcd(d,(b^2-1)b^\lambda)
  &=
  \gcd(d,(b^2-1))
  \prod_{p\dv b} p^{\min(v_p(d),v_p(b^\lambda))}
  \\
  &=
  \gcd(d,(b^2-1)) \prod_{p\dv b} p^{v_p(d)}
  =
  \Pi_d(b)
  ,
\end{align*}
where $\Pi_d(b)$ is defined by \eqref{def:Pi_d(b)}.
We deduce that \eqref{eq:reversed-siegel-walfisz} holds under the
condition \eqref{eq:range_d_proof} and $\lambda \geq \lambda_0(b)$.
The range $\lambda \geq \lambda_0(b)$ may be extended to
$\lambda \geq 2$ by choosing in \eqref{eq:reversed-siegel-walfisz}
an appropriate implicit constant depending on $b$.
This completes the proof of
Theorem~\ref{theorem:reversed-siegel-walfisz}.

\bibliographystyle{siam}

\bibliography{biblio}
%\bibliography{articles,books,theses} % pour Joel
%\bibliography{/home/cathy/Dropbox/tex/bib/articles,/home/cathy/Dropbox/tex/bib/books,/home/cathy/Dropbox/tex/bib/theses} % pour Cathy

\end{document}